\documentclass{article}
\usepackage[margin=1in]{geometry}

\usepackage[utf8]{inputenc}
\usepackage{lmodern}
\usepackage[T1]{fontenc}

\usepackage{amsmath, amssymb, amsthm, amsfonts}

\usepackage[polish, german, english]{babel}

\usepackage{bbm}
\usepackage[retainorgcmds]{IEEEtrantools}
\usepackage{thm-restate}
\usepackage{enumerate}
\usepackage{listings}
\usepackage{verbatim}
\usepackage{faktor}

\usepackage{tikz}
\usetikzlibrary{calc, arrows, decorations.markings, positioning, fit,
shapes.geometric}
\tikzset{>=stealth'}

\theoremstyle{plain}
\newtheorem{theorem}{Theorem}
\newtheorem{lemma}[theorem]{Lemma}
\newtheorem{corollary}[theorem]{Corollary}
\newtheorem{claim}[theorem]{Claim}

\theoremstyle{definition}
\newtheorem{definition}[theorem]{Definition}
\newtheorem{example}[theorem]{Example}
\newtheorem{conjecture}[theorem]{Conjecture}

\theoremstyle{remark}
\newtheorem{remark}[theorem]{Remark}

\DeclareMathOperator*{\EE}{\mathbb{E}}
\DeclareMathOperator{\EEs}{\mathbb{E}}

\DeclareMathOperator{\Var}{Var}
\DeclareMathOperator*{\Cov}{Cov}
\let\Re\relax 
\DeclareMathOperator{\Re}{Re}
\let\Im\relax
\DeclareMathOperator{\Im}{Im}
\DeclareMathOperator{\sgn}{sgn}

\newcommand{\bbR}{\mathbb{R}}

\newcommand{\1}{\mathbbm{1}}
\newcommand{\cE}{\mathcal{E}}
\newcommand{\cG}{\mathcal{G}}

\newcommand{\cS}{\mathcal{S}}
\newcommand{\beats}{\succ}
\newcommand{\eps}{\varepsilon}
\newcommand{\fhat}{\widehat{f}}
\newcommand{\ghat}{\widehat{g}}
\newcommand{\hhat}{\widehat{u}}
\newcommand{\uhat}{\widehat{u}}
\newcommand{\gtilde}{\widetilde{g}}
\newcommand{\Xtilde}{\widetilde{X}}
\newcommand{\cO}{\mathcal{O}}
\newcommand{\Otilde}{\widetilde{\mathcal{O}}}
\newcommand{\Bcompl}{\overline{B}}
\newcommand{\cond}{\;\vert\;}
\newcommand{\mcond}{\;\middle\vert\;}
\newcommand{\VarAcond}{\Var_{A,\mathrm{cond}}}
\newcommand{\VarBcond}{\Var_{B,\mathrm{cond}}}
\newcommand{\CVABcond}{CV_{AB,\mathrm{cond}}}

\usepackage{xcolor}
\definecolor{DSgray}{cmyk}{0,0,0,0.7}
\definecolor{DSred}{cmyk}{0,0.7,0,0.7}

\usepackage[pdftex]{hyperref}

\title{Intransitive dice tournament
is not quasirandom}
\author{Elisabetta Cornacchia\thanks{EPFL, Lausanne, 
Switzerland. E-mail:
{\tt elisabetta.cornacchia@epfl.ch}.}
\and Jan Hązła\thanks{AIMS Rwanda, Kigali. E-mail: {\tt jan.hazla@aims.ac.rw}.
J.H.~is supported by the AIMS Rwanda research chair funding from the Alexander von Humboldt Foundation, as well
as the DAAD grant No. 57610033.}}
\date{}

\begin{document}\maketitle
\begin{abstract}
We settle a version of the conjecture about intransitive dice posed by Conrey, Gabbard, Grant, Liu and Morrison in 2016 and Polymath in 2017. We consider generalized dice with $n$ faces and
we say that a die $A$ beats $B$ if a random
face of $A$ is more likely to show a higher number than 
an independently chosen random face of $B$.
We study random dice with faces drawn iid from the 
uniform distribution on $[0,1]$
and conditioned on the sum of the faces equal to $n/2$. 
Considering the ``beats'' relation for three such random dice,
Polymath showed that
each of eight possible tournaments between them is asymptotically equally likely. In particular, three dice form an intransitive cycle with probability converging to $1/4$. In this paper we prove that for \emph{four} random dice \emph{not} all tournaments are equally likely and the probability of a transitive tournament is strictly higher than $3/8$.
\end{abstract}

\section{Introduction}
\label{sec:introduction}

Intransitive dice are an accessible example of counterintuitive behavior 
of combinatorial objects. For our purposes, a \emph{die}
is a vector of $n$ real numbers $A=(a_1,\ldots,a_n)$. The traditional cube dice have $n=6$. 
Given two $n$-sided dice $A$ and $B$,
we say that $A$ beats $B$, writing it as $A\beats B$, if
\begin{align*}
  \left(\sum_{i,j=1}^n \1[a_i>b_j]-\1[a_i<b_j]\right)>0\;.
\end{align*}
In other words, $A$ beats $B$ if a random roll (uniformly
chosen face)
of $A$ is more likely to display a larger number than a random roll of $B$.
The term ``intransitive dice'' refers to the fact that the ``beats''
relation is not transitive. One well-known example
due to Efron~\cite{Gar70} with $n=6$ is
\begin{align*}
    A = (0,0,4,4,4,4)\;,\qquad
    B = (3,3,3,3,3,3)\;,\qquad
    C = (2,2,2,2,6,6)\;,\qquad
    D = (1,1,1,5,5,5)\;,
\end{align*}
where one checks that $A\beats B\beats C\beats D\beats A$.
In particular, this set of dice allows for an amusing game: If two players sequentially choose one die each and make a throw, with the player rolling a higher number receiving a dollar from the other player, then it is the second player who has a strategy
with positive expected payout.

The intransitivity of dice is a noted subject in popular
mathematics~\cite{Gar70, Sch00, Gri17}. It has been studied under various
guises for a considerable time~\cite{ST59, Try61, Try65, MM67, Sav94}.
There is a multitude of constructions of intransitive
dice with various properties~\cite{MM67, FT00, BB13, AD17, SS17, BGH18, AS19},
as well as studies of game-theoretic aspects~\cite{Rum01, SMB06, HW19}. In particular, it is known (e.g., \cite{MM67, AD17}) that for every tournament
on $k$ vertices, there exists a set of $k$ dice whose
``beats'' relation realizes the tournament. 

We also note another related
result by Buhler, Graham and Hales~\cite{BGH18} that constructs an unexpected example
which they dub ``maximally nontransitive dice''. Given a set of $k$ dice, let $T_1,\ldots,T_\ell,\ldots$ be the sequence
of tournaments on $k$ vertices defined in the following way:
There is an edge from $i$ to $j$ in the tournament $T_{\ell}$
if and only if the sum of $\ell$ independent rolls of die $i$
is more likely to yield a higher value than the sum
of $\ell$ independent rolls of die $j$.
\cite{BGH18}~constructs, for
every $k$, a set of $k$ dice 
where every possible tournament on $k$ vertices occurs infinitely often in the sequence $(T_\ell)_{\ell\in\mathbb{N}}$.

While most of these results concern interesting examples and constructions, 
there has been recent interest in understanding how common intransitive dice
really are. In other words, is there a natural way of generating intransitive examples?
Conrey, Gabbard, Grant, Liu and Morrison~\cite{CGG16}
proposed several models of random dice to study this question.
A natural starting point is to pick a distribution of a single face,
e.g., uniform over $[0,1]$ or $[n]:=\{1,\ldots,n\}$ and sample dice
with $n$ iid faces from this distribution. Some thought reveals that this does not generate intransitivity
(see, e.g., Theorem~6 in~\cite{HMRZ20} for details). 
The reason
is that with high probability, in a small set of dice
sampled from this model, $A$ beats $B$ if and only
if $\sum_{i=1}^n a_i>\sum_{i=1}^n b_i$. That is, the \emph{face-sum} 
$\sum_{i=1}^n a_i$ is
a proxy of how strong a die is 
and the ``beats'' relation on the set of dice is transitive.

\cite{CGG16} suggested and presented experimental evidence that to ``reveal''
intransitivity one can sample \emph{random balanced dice}. What
we mean by that is the
faces are iid and the die is 
conditioned on having the expected face-sum. They conjectured
(in a related model where random die is a uniform balanced multiset
of $[n]$ rather
than iid sequence) that three dice in such model form an intransitive
cycle with asymptotic probability $1/4$. A collaborative 
Polymath project~\cite{Pol17a} was launched to investigate this conjecture
and indeed proved it for dice with iid uniform faces: 
\begin{theorem}[\cite{Pol17}]\label{thm:polymath-main}
If $A,B,C$ are three random dice with $n$ faces iid uniform in $[n]$ and
conditioned on all face-sums equal to $n(n+1)/2$, then the probability that
$A,B,C$ are intransitive is $1/4+o(1)$.
\end{theorem}
For later reference we now give an equivalent statement of 
Theorem~\ref{thm:polymath-main}. What it says is that 
a random balanced die
is unbiased with respect to
the ``beats'' relation:
\begin{theorem}[\cite{Pol17}]
\label{thm:polymath-reduced}
If $A,B$ are two random dice with $n$ faces iid uniform in $[n]$ and
conditioned on all face-sums equal to $n(n+1)/2$, then, except with probability
$o(1)$, we have $\Pr[A\beats B\cond B]=1/2+o(1)$.
\end{theorem}
The equivalence of the two theorems is proved
using general properties of tournaments and establishing
it is the first step in the Polymath proof of Theorem~\ref{thm:polymath-main}.

Indeed, it seems that the balanced face-sum for uniform faces is the
``right'' model for generating intransitive dice.
Perhaps surprisingly, it turns out that both balancedness and uniformity are important.
In particular, a follow-up work~\cite{HMRZ20} 
showed that if dice are balanced with
faces iid from any (continuous) distribution that is \emph{not} uniform
(and also in some other models with correlated Gaussian faces), then the 
intransitivity does not show up and three dice are transitive with high probability.

Recall that there are eight tournaments on three vertices, six of them 
transitive and the other two cycles. By symmetry considerations, tournaments
in each of those two groups are equally likely when sampling
three random dice. Therefore, 
Theorem~\ref{thm:polymath-main} means that each tournament on three 
random balanced dice
is asymptotically equally likely.
In that context,
\cite{CGG16} made a natural further-reaching conjecture in the multiset model on $[n]$.
They conjectured that, 
for any fixed $k$, all tournaments on $k$ dice are equally likely.
Using terminology from the paper by Chung and Graham~\cite{CG91}, what
they conjectured is that
the balanced dice tournament is \emph{quasirandom}. 
However, in another surprising turn,
experiments done by Polymath~\cite{Pol17b} suggested that 
for four balanced dice with faces iid uniform in $[n]$ this conjecture is false, with the probability of
a transitive outcome hovering around $0.39$ rather than
$3/8=0.375$ expected in a quasirandom tournament
(cf.~Figure~\ref{fig:tournaments}). 
However, it seems
that the question was ultimately left open by Polymath~\cite{Pol17c}.

\begin{figure}[!ht]\centering\begin{tikzpicture}
    \node [draw, circle] (A1) at (0,4.5) {\phantom{$A$}};
    \node [draw, circle] (B1) at (0,3) {\phantom{$A$}};
    \node [draw, circle] (C1) at (0,1.5) {\phantom{$A$}};
    \node [draw, circle] (D1) at (0,0) {\phantom{$A$}};
    \draw [thick, ->] (A1)--(B1);
    \draw [thick, ->] (B1)--(C1);
    \draw [thick, ->] (C1)--(D1);
    \draw [thick, ->] (A1.east) to [bend left] (C1.east);
    \draw [thick, ->] (B1.west) to [bend right] (D1.west);
    \draw [thick, ->] (A1.east) to [bend left=45] (D1.east);
    
    \node [draw, circle] (A2) at (4, 4) {\phantom{$A$}};
    \node [draw, circle] (B2) at (4, 2.5) {\phantom{$A$}};
    \node [draw, circle] (C2) at (5, 1) {\phantom{$A$}};
    \node [draw, circle] (D2) at (3, 1) {\phantom{$A$}};
    \draw [thick, ->] (A2)--(B2);
    \draw [thick, ->] (A2.east) to [bend left] (C2.north);
    \draw [thick, ->] (A2.west) to [bend right] (D2.north);
    \draw [thick, ->] (B2)--(C2);
    \draw [thick, ->] (C2)--(D2);
    \draw [thick, ->] (D2)--(B2);

    \node [draw, circle] (A3) at (7.5, 3.25) {\phantom{$A$}};
    \node [draw, circle] (B3) at (9.5, 3.25) {\phantom{$A$}};
    \node [draw, circle] (C3) at (9.5, 1.25) {\phantom{$A$}};
    \node [draw, circle] (D3) at (7.5, 1.25) {\phantom{$A$}};
    \draw [thick, ->] (A3)--(B3);
    \draw [thick, ->] (B3)--(C3);
    \draw [thick, ->] (C3)--(D3);
    \draw [thick, ->] (D3)--(A3);
    \draw [thick, ->] (A3)--(C3);
    \draw [thick, ->] (B3)--(D3);
\end{tikzpicture}
\caption{Three types of tournaments on four elements. From the left:
A transitive ordering (there are 24 in total), an overall winner on top of a 3-cycle (16 in total together with
a symmetric case of overall loser) and
a 4-cycle (24 in total).}
\label{fig:tournaments}
\end{figure}
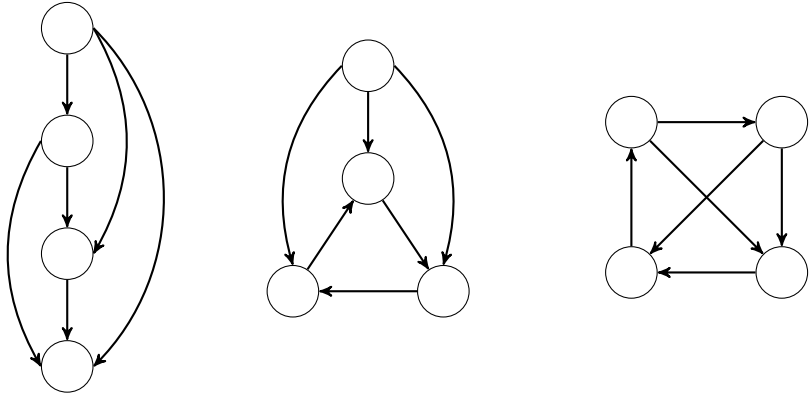

In this paper we settle a version of this conjecture in
accordance with the Polymath experiments.
We are working in the model of random balanced dice with
faces uniform in $[0,1]$: The dice are drawn with
faces iid uniform in $[0,1]$ and conditioned on the face-sums
equal to $n/2$. To be precise, we think of drawing
balanced dice by rejection sampling: Let $(a_1,\ldots,a_{n-1})$
be iid in $[0,1]$, set $a_n:=n/2-\sum_{i=1}^{n-1} a_i$
and accept and output a sample
if $0\le a_n\le 1$, otherwise reject and repeat.

\begin{theorem}\label{thm:main-intro}
  Let $A,B,C,D$ be four random dice with $n$ faces iid uniform in
  $[0,1]$ conditioned on all face-sums equal to $n/2$. Then,
  there exists $\eps>0$ such that,
  for $n$ large enough,
  the probability that the ``beats'' relation 
  on $A,B,C,D$ is transitive is 
  larger than $3/8+\eps$.
\end{theorem}
Theorem~\ref{thm:main-intro} implies that the dice tournament
is not quasirandom, since there are $2^{\binom{4}{2}}=64$ tournaments on four
elements and $4!=24$ of them, i.e., $3/8$ fraction, are transitive.
We note that our main motivation for the choice of the continuous distribution on $[0,1]$ was the ease of presentation (in particular we do not have to worry about ties). While we have not verified the details, we believe that our proof strategy can be adapted to the discrete case without conceptual difficulties.

The proof of Theorem~\ref{thm:main-intro} combines the main ideas present in two preceding 
works~\cite{Pol17} and~\cite{HMRZ20} and can roughly be
divided into two parts inspired by these papers.
However, there is a considerable number of items that need to be fleshed out and some care is required to get the details right.

The resulting $\eps$ in Theorem~\ref{thm:main-intro} seems tiny and we do not attempt to estimate it. 
We also note that in our model,
for typical balanced dice $A$ and $B$, 
the ``one-roll probability''
\begin{align*}
    P_1(n):=\Pr[\text{$A$ rolls higher than $B$}]
    =\frac{1}{n^2}\sum_{i,j=1}^n \Big(\mathbbm{1}(a_i>b_j)-\mathbbm{1}(a_i<b_j)\Big)
\end{align*}
is rather close to half, satisfying
$1/2-c/n\le P_1(n)\le 1/2+C/n$ for some universal
constants $c,C>0$.
Therefore, in the game we described before
it takes the second player order of $n$ dice rolls
to achieve one expected unit of payoff.

We leave open some intriguing
questions. For example, is there a more complete characterization
of the ``dice tournament'' except for the fact that it is not 
quasirandom?\footnote{We note that this question
was extensively addressed and answered in a subsequent
work by Sah and Sawney~\cite{SS24}.}
Is there another natural tweak to the model that makes the tournament
quasirandom or is there a reason to believe otherwise?

To explain how Theorem~\ref{thm:main-intro} comes about, it is worthwhile
to state its equivalent version that is a close analogue of 
Theorem~\ref{thm:polymath-reduced}. Indeed, in Section~\ref{sec:four-three-reduction}
we show that the following statement
implies Theorem~\ref{thm:main-intro}:
\begin{theorem}\label{thm:reduced-intro}
  Let $A,B,C$ be three random dice with $n$ faces iid uniform in
  $[0,1]$ conditioned on all face-sums equal to $n/2$. Then, there exists
  $\eps>0$ such that, for large enough $n$, with probability at least $\eps$
  we have $\big|\Pr[A\beats B,C\cond B,C]-1/4\big|>\eps$.
\end{theorem}

In the rest of this section we are going to give a 
sketch of the strategy that we use to prove
Theorem~\ref{thm:reduced-intro}.

\paragraph{Proof strategy} 
For the sake of clarity let us assume
that we want to demonstrate a slightly stronger statement
$\Pr[A\beats B,C\cond B,C]\allowbreak>1/4+\eps$ with constant probability. 
Intuitively, we would like to show that, with constant probability,
two random balanced dice $B,C$ are, in some sense, close to each other
and therefore there is positive correlation in the ``beats'' relation with another random die $A$. 
The equivalence between Theorems~\ref{thm:main-intro}
and~\ref{thm:reduced-intro} has already been spotted by Polymath
and~\cite{Pol17c} discusses some hypotheses on how the correct notion
of closeness could look like.

To continue, we need to say a few more words about the ideas used by Polymath to prove
Theorem~\ref{thm:polymath-reduced}. For a die $A$ with faces in $[0,1]$
and $0\le x\le 1$, let
\begin{align}\label{eq:34}
    f_A(x):=
    \sum_{i=1}^n \1\left[a_i\le x\right]=
    \big|\left\{i\in[n]: a_i\le x\right\}\big|\;,\qquad\qquad
    g_A(x):=f_A(x)- xn\;.
\end{align}
One checks that, as long as there are no ties between the faces,
a die $A$ beats $B$ if and only if
$\sum_{i=1}^n f_A(b_i)<n^2/2$. Furthermore, for balanced $A$ and $B$ it holds that $\sum_{i=1}^n g_A(b_i) =\sum_{i=1}^n f_A(b_i) - n^2/2  $, thus $A$ beats $B$ if and only if $\sum_{i=1}^n g_A(b_i)<0$. For a fixed
balanced die $A$, let $V$ be a uniform random variable in $[0,1]$ and
$U_A:=g_A(V)$. Now the event ``$A$ beats $B$'' for a random balanced die $B$
can be written as
\begin{align*}
    U_A^{*n}<0\quad\text{ conditioned on }\quad V^{*n}=n/2\;,
\end{align*}
where $(*n)$ denotes the $n$-wise convolution of random variables,
i.e.,
$(U_A^{*n},V^{*n}) = (\sum_{i=1}^n U_{A,i},
\allowbreak\sum_{i=1}^n V_i)$, and $(U_{A,i},V_i)$ are $n$
iid copies of the joint distribution of $(U_A, V)$.
\cite{Pol17} shows that, in their setting, a precise version of central limit
theorem holds for $(U_A, V)$ for a typical balanced $A$. Since $\EE U_A=0$
and we condition on $V^{*n}$ being equal to its expectation, after
the conditioning $U_A^{*n}$ is still close to a centered Gaussian and
Theorem~\ref{thm:polymath-reduced} follows.

In order to prove Theorem~\ref{thm:reduced-intro}, we follow suggestions
from the Polymath discussion and extend this limit theorem to a
triple of random variables $(U_B, U_C, V)$ for typical random balanced
dice $B$ and $C$. Now a notion of closeness suggests itself:
Since we are interested in finding $B$, $C$ for which
$\Pr[U_B^{*n},U_C^{*n}>0]$ is bounded away from $1/4$,
it might be helpful to prove that, with constant probability, the
covariance $\Cov[U_B, U_C]$ is bounded away from zero. Since
for balanced $B$ and $C$ it holds that
\begin{align*}
    \int_0^1\big(g_B(x)-g_C(x)\big)^2\,\mathrm{d}x
    = \Var U_B + \Var U_C - 2\Cov[U_B,U_C]\;,
\end{align*}
a somewhat large value of the covariance is 
roughly equivalent to somewhat small $L^2$ distance
between functions $g_B$ and $g_C$. While we follow this
intuitive picture, in the proofs we will be always
dealing directly with the covariance.

As a matter of fact, the preceding discussion is not quite correct. Since we are conditioning
on $V^{*n}=n/2$, we need to look at the \emph{conditional covariance}, that
is exclude the possibility that all the covariance between $U_B$ and
$U_C$ is ``mediated'' through $V$. An elementary calculation shows that
the correct expression is $CV_{\mathrm{cond}}:=\Cov[U_B,U_C]-12\Cov[U_B,V]\Cov[U_C,V]$.  
The event $|\Pr[A\beats B,C\cond B,C]-1/4|>\eps$ 
is essentially equivalent
to $|CV_{\mathrm{cond}}|$ being 
not too small compared to the variances of
$U_B$ and $U_C$.

We show the bound on $|CV_{\mathrm{cond}}|$
using the second moment method. Thinking of $CV_{\mathrm{cond}}$
as a random variable over the choice of balanced $B$ and $C$, we bound
its second and fourth moments and apply the Paley-Zygmund inequality to conclude
that indeed sometimes it is bounded away from zero. In the proof we employ a 
modification of the argument used in~\cite{HMRZ20}:
If $B$ and $C$ are not balanced, estimating the moments is an elementary
calculation. To account for the conditioning on balanced $B$ and $C$, we
employ yet another precise local central limit theorem to estimate
changes in the relevant moments. This requires some care, but in the end we
conclude that the conditioning does not fundamentally alter the behavior
of $CV_{\mathrm{cond}}$.

\paragraph{Organization} The rest of the paper contains the proof
of Theorem~\ref{thm:main-intro}. Section~\ref{sec:proof-outline}
includes a more detailed outline, while Sections~\ref{sec:clt-proof}
and~\ref{sec:moments} the details, of, respectively,
our central limit theorem and the moment calculations.

\paragraph{Notation} All statements of theorems, lemmas etc.~concerning
dice are meant to hold for number of faces $n$ large enough, 
even if not explicitly mentioned.
We also assume for ease of presentation that the number of
faces $n$ is odd, but the proof can be adapted to the case of even $n$. All logarithms are natural. We use $\cO(\cdot)$ for the
Big-oh notation, as well as
$f=\Omega(g)$ for $g=\cO(f)$ and
$f=\Theta(g)$ for $f=\cO(g)$ and $f=\Omega(g)$.
Furthermore, $\Otilde(f(n))$ denotes
$\cO(f(n)\cdot\log^c n)$ for some fixed $c$.
For $S\subseteq\mathbb{R}$, we let 
$d_S(x):=\inf\left\{|x-s|:s\in S\right\}$.
For a random variable $U$, we use the notation
$\|U\|_\infty$ for the essential supremum of $|U|$.

As we mentioned, by a ``random balanced die'' we mean a
die with iid uniform faces constrained to have the expected
face-sum by rejection sampling. Equivalently, if
we consider the uniform measure $\mu$ over 
$a_1,\ldots,a_{n-1}\in [0,1]^{n-1}$ and adopt
the convention that given $a_1,\ldots,a_{n-1}$ we let 
$a_n := n/2-\sum_{i=1}^{n-1} a_i$, then we decree that for an
integrable $f:\mathbb{R}^{n}\to\mathbb{R}$ we have
\begin{align}\label{eq:44}
\EE\left[f(a_1,\ldots,a_n)\mcond \sum_{i=1}^n a_i=n/2\right]
&:=\frac{\int
 f\left(a_1,\ldots,a_n\right)\cdot
 \mathbbm{1}\left(0\le a_n\le 1\right)
 \,\mathrm{d}\mu
}
{\int \mathbbm{1}
\left(0\le a_n\le 1\right)\,\mathrm{d}\mu}\;.
\end{align}

\paragraph{Acknowledgments} We are grateful to Ido Nachum for his
comments on the manuscript and to the referees for multiple corrections
and suggestions.

\section{Proof outline}
\label{sec:proof-outline}

In this section we present the main steps in the proof
of Theorem~\ref{thm:main-intro},
leaving more technical details for the following
sections. We present the proof divided into a number of stages,
starting with reducing Theorem~\ref{thm:main-intro}
to Theorem~\ref{thm:reduced-intro}.

\subsection{Reduction from four to three dice}
\label{sec:four-three-reduction}

In order to make a connection between 
Theorems~\ref{thm:main-intro}
and~\ref{thm:reduced-intro}, we also need a version
of Theorem~\ref{thm:polymath-reduced} adapted to our setting.
\begin{theorem}\label{thm:one-fixed}
  Let $A,B$ be two random balanced dice with faces
  uniform in $[0,1]$. Then, 
  except with probability $1/n$ over the 
  choice of $A$, we have
  \begin{align*}
      \left|\Pr\Big[
      A\beats B\cond A
      \Big]-\frac{1}{2}\right|=
      \Otilde\left(\frac{1}{\sqrt{n}}\right)\;.
  \end{align*}
\end{theorem}
While one could prove Theorem~\ref{thm:one-fixed} directly by adapting the argument from~\cite{Pol17}, we will obtain it as a byproduct
in the proof of Theorem~\ref{thm:reduced-intro}.

We will now show that Theorem~\ref{thm:main-intro}
can be deduced
from Theorems~\ref{thm:reduced-intro}
and~\ref{thm:one-fixed}. We start with some preliminary 
observations. Note that in our model almost surely all
faces of the sampled dice will have different values.
In particular, assuming $n$ odd,
the draws between
dice (i.e., situations where neither $A\beats B$ nor $B\beats A$) occur with
probability zero.
Therefore, almost surely, the
``beats'' relation on three random balanced
dice forms one of the eight possible tournaments.
As discussed, six of these tournaments are transitive and 
the other two are cycles, and by symmetry
the tournaments 
inside each class occur with the same probability.
Denoting the overall probabilities of the two classes as
$P_{\text{3-line}}$ and $P_\triangle$, we have, e.g.,
\begin{align*}
    P_{\text{3-line}}
    &=6\Pr[A\beats B,C \text{ and } B\beats C]\;,\\
    P_{\triangle}
    &=2\Pr[A\beats B\beats C\beats A]\;.
\end{align*}

Similarly, a random tournament on four balanced dice is
characterized by three probabilities of different types
of tournaments on four elements, as seen
in Figure~\ref{fig:tournaments}. These are 24 transitive
tournaments, 
16 tournaments with a unanimous winner (or loser)
above (respectively below) a 3-cycle, and 24 4-cycles. 
Again, tournaments inside
each class are equally likely and we denote the respective
overall probabilities as
$P_{\text{4-line}}$, $P_{1,\triangle}$ and $P_\square$.

We will need a simple application of Markov's inequality:
\begin{claim}\label{cl:inverse-chebyshev}
Let a random variable $X$ satisfy $|X-\EE X|\le 1$ almost surely.

If $\Pr[|X-\EE X|\ge \eps_1]\le\eps_2$, then 
$\Var X\le \eps_1^2+\eps_2$.
On the other hand, if $\Pr[|X-\EE X|\ge c_1]\ge c_2$, then
$\Var X\ge c_1^2c_2$.
\end{claim} 

We can now state and prove the lemma establishing the reduction:

\begin{lemma}\label{lem:four-three-reduction}
Let $A,B,C,D$ be random balanced dice and
let $X:=\Pr[A\beats B\cond A]$ and $Y:=\Pr[A\beats B,C\cond B,C]$.
Assume that for some constants $\eps_1,\eps_2,c_1,c_2>0$ such
that $c_1-\eps_1^2-\eps_2\ge 0$ the following
two conditions
hold:
\begin{align}
    \Pr\left[\left|X-\frac{1}{2}\right|\ge\eps_1\right]
    &\le\eps_2\;,
    \label{eq:32}\\
    \Pr\left[\left|Y-\frac{1}{4}\right|\ge c_1\right]
    &\ge c_2\;.
    \label{eq:31}
\end{align}
Then,
\begin{align*}
    P_{\mathrm{4\text{-}line}}\ge\frac{3}{8}+
    6c_2(c_1-\eps_1^2-\eps_2)^2\;.
\end{align*}
\end{lemma}
\begin{proof}
Note that if $A,B\beats C,D$, then $A,B,C,D$ form a transitive
tournament.
Hence, using symmetry considerations,
\begin{align*}
    \EE Y^2 = \EE \Pr[A,B\beats C,D \cond C,D]
    =\Pr[A,B\beats C,D]=\frac{1}{6}P_{\mathrm{4\text{-}line}}\;.
\end{align*}
Accordingly, we have $P_{\text{4-line}}=6\EE Y^2$.
At the same time, by symmetry we have $\EE X=1/2$ and, applying
Claim~\ref{cl:inverse-chebyshev} to~\eqref{eq:32},
$\Var X\le\eps_1^2+\eps_2$. 

Noting that 
$\EE Y=\Pr[A\beats B,C]=\EE X^2=\EEs^2 X+\Var X
\le 1/4+\eps_1^2+\eps_2$, we apply~\eqref{eq:31} and
Claim~\ref{cl:inverse-chebyshev} to conclude
\begin{align*}
P_{\text{4-line}}
&=6\EE Y^2
=6\EEs^2 Y+6\Var Y\ge6\EEs^4 X+6\Var Y
\ge \frac{3}{8}+6c_2(c_1-\eps_1^2-\eps_2)^2\;.\qedhere
\end{align*}
\end{proof}

The fact that Theorem~\ref{thm:main-intro} follows from
Theorems~\ref{thm:reduced-intro} and~\ref{thm:one-fixed} is now
a straightforward application of Lemma~\ref{lem:four-three-reduction}.

\begin{remark}
For a somewhat broader picture of the dice tournament,
note that a similar argument as in the proof of Lemma~\ref{lem:four-three-reduction} gives
\begin{align*}
    3\EE X^2 = 3\Pr[A\beats B,C]=P_{\text{3-line}}\;.
\end{align*}
Hence, as a result of Theorem~\ref{thm:one-fixed} we have estimates 
\begin{align*}
 P_{\text{3-line}}
 &=3\EE X^2
 =3\EEs^2 X + 3\Var X
 =\frac{3}{4}+\Otilde(1/n)
\end{align*}
and similarly $P_{\triangle}=1/4-\Otilde(1/n)$, meaning
that in the limit all eight tournaments are equally likely. 

With four dice, observe that choosing three dice is equivalent
to choosing four (unordered) dice and then randomly discarding
one of those four and permuting the remaining three.
Analyzing this procedure, we see that
\begin{align*}
    P_{\text{3-line}}
    &=P_{\text{4-line}}+\frac{1}{2}P_{\square}
    +\frac{3}{4}P_{1,\triangle}\;,\\
    P_{\triangle}
    &=\frac{1}{2}P_{\square}+\frac{1}{4}P_{1,\triangle}\;,
\end{align*}
from which it follows that in our case $P_{\text{4-line}}=P_{\square}+\Otilde(1/n)$.

Therefore, in light of Theorem~\ref{thm:main-intro}, the following
picture emerges. 
Asymptotically, all transitive tournaments
and all 4-cycles have the same 
probability\footnote{Technically we only
prove that this probability is $1/64+c_n$ for
some $c_n > c > 0$ without excluding the possibility
that $c_n$ does not converge.}
$1/64+c$ for some constant
$c>0$, while all ``winner/loser + 3-cycle'' tournaments have probability
$1/64-3c$.
\end{remark}

\subsection{Function \texorpdfstring{$g_A$}{g\textunderscore A}}
\label{sec:g_a-function}

In order to prove Theorems~\ref{thm:reduced-intro} and~\ref{thm:one-fixed}, a crucial
quantity is the $g_A$ function given by~\eqref{eq:34} 
(and already defined by Polymath).
For reasons of presentation we will consider dice with
faces sampled from different uniform distributions on intervals,
not necessarily $[0,1]$. In that general case of an interval
$[z_1,z_2]$, a die is called balanced if its face-sum is
equal to $n(z_1+z_2)/2$. We still consider sampling uniform balanced
dice.

Given a die $A=(a_1,\ldots,a_n)$ with faces
sampled from such a uniform distribution on an interval
with cdf $F$, we let for $z_1\le x\le z_2$
\begin{align*}
    f_A(x):=
    \sum_{i=1}^n \1\left[a_i\le x\right]\;,
    \qquad
    g_A(x):=f_A(x)-F(x)\cdot n\;.
\end{align*}
Indeed, in our current setting of
faces in $[0,1]$ this formula matches~\eqref{eq:34}.
In general, the value of $g_A(x)$ indicates
the difference between the number of faces of $A$ not exceeding
$x$ and the expected number of faces not exceeding $x$ if
drawn iid from the distribution.

We now state a number of straightforward claims
about the function $g_A$. The proofs are calculations
using that $F(x)=\frac{x-z_1}{z_2-z_1}$ for $z_1\le x\le z_2$.
\begin{claim}\label{cl:simple-g}
In the following let $A$ and $B$ be fixed dice with
$n$ faces from $[z_1,z_2]$ and $V$ a random variable uniform
in $[z_1,z_2]$ independent of everything else:
  \begin{enumerate}
  \item $\EE F(V)=1/2$.
  Furthermore, if $A$ is balanced,
  then $\sum_{i=1}^n F(a_i)=n/2$.
  \item If $A$ is balanced,
  then $\EE g_A(V)=0$.
  \item If $B$ is balanced
  and $a_i\ne b_j$ for every
    pair of faces, then $B$ beats $A$ if and only if
    $\sum_{i=1}^n g_A(b_i)>0$.
  \item If $B$ has face-sum $b$, then
  $\sum_{i=1}^n g_A(b_i)\in\mathbb{Z}-n\frac{b-nz_1}{z_2-z_1}$.
  In particular,
  if $B$ is balanced and $n$ is odd, then
  $\sum_{i=1}^n g_A(b_i)\in\mathbb{Z}+1/2$.
  \end{enumerate}
\end{claim}

Let $V$ be a random variable uniform in $[0, 1]$ and
for a die $A$ let $U_A:=g_A(V)$. By the above, and using the notation
$(U_A^{*n},U_B^{*n},U_C^{*n},V^{*n})$ for the $n$-fold
convolution of the random tuple $(U_A,U_B,U_C,V)$,
for balanced dice $A$, $B$, $C$ with faces uniform in $[0,1]$ we have
\begin{align*}
    \Pr[A\beats B,C\cond B,C]
    &=
    \Pr\left[U_B^{*n},U_C^{*n}>0\mcond B,C,V^{*n}=\frac{n}{2}\right]\;,\\
    \Pr\left[A\beats B\mcond A\right]
    &=
    \Pr\left[U_A^{*n}<0\mcond A,V^{*n}=\frac{n}{2}\right]\;.
\end{align*}
Therefore, from now on we can focus on proving the following equivalent
versions of Theorems~\ref{thm:reduced-intro} 
and~\ref{thm:one-fixed}. For clarity we omit the conditioning
on dice $A$ and $B$ from the probabilities:
\begin{theorem}\label{thm:two-fixed-equivalent}
There exists $\eps>0$ such that with probability at least $\eps$
over the choice of two random balanced dice $A$ and $B$ 
with faces in $[0,1]$ we have
\begin{align*}
    \left|\Pr\left[U_A^{*n},U_B^{*n}>0\mcond V^{*n}=\frac{n}{2}\right]
    -\frac{1}{4}\right|>\eps\;.
\end{align*}
\end{theorem}
\begin{theorem}\label{thm:one-fixed-equivalent}
  Except with probability
  $1/n$ over the choice of a random balanced die $A$
  with faces in $[0,1]$,
  \begin{align*}
      \left|\Pr\left[U_A^{*n}>0\mcond V^{*n}=\frac{n}{2}\right]
      -\frac{1}{2}\right|=\Otilde\left(\frac{1}{\sqrt{n}}\right)\;.
  \end{align*}
\end{theorem}
In the following sections we explain how to prove 
Theorems~\ref{thm:two-fixed-equivalent} and~\ref{thm:one-fixed-equivalent}.

\subsection{Central limit theorem for 
 \texorpdfstring{$\boldsymbol{g_A}$}{g\textunderscore A}}
\label{sec:clt}

One method for understanding the distribution of the random
variables $(U_A^{*n},U_B^{*n},V^{*n})$ is to establish a limit
theorem for them. This is indeed the approach taken
by~\cite{Pol17} and the one we adopt. 
To that end, let $(G_A, G_B, H)$ be joint Gaussians with
the same first and second moments as $(U_A,U_B,V)$. Recall that
by Claim~\ref{cl:simple-g} we have $\EE U_A =\EE U_B =0$
and $\EE V=1/2$. In addition, we introduce the notation
for the respective variances
\begin{align*}
    &\Var_A:=\Var U_A\;,\qquad\qquad 
    \Var_B:=\Var U_B\;,\\
    &CV_{AB}:=\Cov[U_A, U_B]\;,\quad\;
    CV_A:=\Cov[U_A, V]\;,\qquad
    CV_B:=\Cov[U_B, V]\;.
\end{align*}

In the end the main
result we prove is
\begin{theorem}\label{thm:conditional-clt}
  Fix $\eps>0$.
  Except with probability $1/n$ over the choice of two
  balanced dice $A$ and $B$, we have the following:
  If the dice satisfy $\Var_A-(CV_A^2/\Var V)$,
  $\Var_B-(CV_B^2/\Var V)\ge \eps n$, then
  \begin{align*}
      \Bigg|
      \Pr\Big[U_A^{*n},U_B^{*n}>0
      \cond V^{*n}=\EE V^{*n}\Big]
      -\Pr\Big[G_A^{*n},G_B^{*n}>0
      \cond H^{*n}=\EE H^{*n}\Big]
      \Bigg|=\Otilde\left(\frac{1}{\sqrt{n}}\right)\;.
  \end{align*}
\end{theorem}

In Section~\ref{sec:clt-proof} we also show
that Theorem~\ref{thm:one-fixed-equivalent} is 
a direct consequence of intermediate results 
established on the way to proving
Theorem~\ref{thm:conditional-clt}.
We note that the ``error probability'' $1/n$
in the statement of Theorem~\ref{thm:conditional-clt}
can be decreased,
at least to $n^{-k}$ for arbitrary $k$.
Before proceeding with the proof outline, let us explain
the $\Var_A-(CV_A^2/\Var V)$ factors from the statement
of the theorem. The reason for them is the following
claim, which is proved by a standard calculation
in Section~\ref{sec:facts-cf}:
\begin{claim}\label{cl:conditional-variance}
Assume $A$ and $B$ are balanced dice with uniform faces in
an interval $[z_1,z_2]$.
Conditioned on $H^{*n}=\EE H^{*n}$, random variables
$G_A^{*n}$ and $G_B^{*n}$ are joint centered Gaussians
and furthermore:
\begin{align}
    \Var\big[G_A^{*n}\cond H^{*n}=\EE H^{*n}\big]
    &=n\left(\Var_A-\frac{CV_A^2}{\Var V}\right)\;,
    \label{eq:26}\\
    \Var\big[G_B^{*n}\cond H^{*n}=\EE H^{*n}\big]
    &=n\left(\Var_B-\frac{CV_A^2}{\Var V}\right)\;,
    \label{eq:27}\\
    \Cov\big[G_A^{*n},G_B^{*n}\cond H^{*n}=\EE H^{*n}\big]
    &=n\left(CV_{AB}-\frac{CV_ACV_B}{\Var V}\right)\;.
    \label{eq:28}
\end{align}
\end{claim}

\paragraph{Proof strategy for Theorem~\ref{thm:conditional-clt}}
As a first step, for consistency with the Polymath
preprint~\cite{Pol17} we find it convenient to consider 
balanced dice with faces uniform in $[0,n]$.
That is, we assume that $A$ and $B$ are random dice
with faces iid uniform in $[0,n]$ and conditioned on 
face-sums equal to $\EE V^{*n}=\EE H^{*n}=n^2/2$.
Furthermore, we have $g_A(x)=f_A(x)-x$ and $V$ is uniform
in $[0,n]$. 

Using the natural measure-preserving bijection between
balanced dice $A\leftrightarrow nA$, it is easy to check
that the distribution of random variable 
$\Var_{A}$ in the $[0,n]$ setting is
the same as the distribution of $\Var_A$ in the $[0,1]$ setting. Furthermore, since the covariance $CV_A$ is multiplied
by a factor $n$ in the $[0,n]$ setting and the variance
$\Var V$ is $n^2$ times larger, also the random
variable
$\Var_{A}-(CV_{A}^2/\Var V)$ has the same distribution.
The upshot is that in the proof of Theorem~\ref{thm:conditional-clt}
we can assume that the faces of $A$ and $B$
are drawn from $[0,n]$ and indeed
this is what we will do from now on and throughout
Section~\ref{sec:clt-proof}.

The proof strategy follows Polymath in most important respects.
That is, we proceed by directly bounding the characteristic function
of the triple of random variables $(U_A,U_B,V-n/2)$:
\[
    \fhat(\alpha,\beta,\gamma):=
    \EE e\big(\alpha U_A+\beta U_B+\gamma(V-n/2)\big)\;,
\]
where $e(x)$ is a shorthand for $\exp(2\pi i x)$.
The main bound we achieve is given in
the following lemma:
\begin{lemma}\label{lem:fhat-bounded}
Except with probability $n^{-2}$ over
the choice of balanced dice $A$ and $B$, the following
holds:
$\|U_A\|_\infty,\|U_B\|_\infty\le 5\sqrt{n\log n}$
and furthermore,
for every $\alpha,\beta,\gamma\in\mathbb{R}$ such that
$|\alpha|,|\beta|\le 1/2$ and 
either $|\alpha|$ or $|\beta|$ exceeds
$10^{10}\log n/n$ or $|\gamma|$ exceeds 
$6\log^2 n/n^{3/2}$,
\[
\left|\fhat(\alpha,\beta,\gamma)\right|\le 1-\frac{10\log n}{n}\;,
\]
and consequently
\[
\left|\fhat(\alpha,\beta,\gamma)\right|^n\le n^{-10}\;.
\]
\end{lemma}
Recall that $(G_A, G_B, H)$ are joint Gaussians with
the first and second moments matching those of
$(U_A,U_B,V)$. We let
\begin{align}
\ghat(\alpha,\beta,\gamma)
&:=\EE e\big(\alpha G_A^{*n}+\beta
G_B^{*n}+\gamma(H^{*n}-n^2/2)\big)\;,
\label{eq:41}\\
\hhat(\alpha,\beta,\gamma)
&:=\begin{cases}
\fhat(\alpha,\beta,\gamma)^n&
\text{if $|\alpha|,|\beta|\le 1/2$,}\\
0&\text{otherwise.}
\end{cases}
\nonumber
\end{align}
We then use Lemma~\ref{lem:fhat-bounded} and a couple
of other bounds to establish 
\begin{lemma}\label{lem:characteristic-close}
Whenever the thesis of Lemma~\ref{lem:fhat-bounded}
holds, in particular except with probability at most 
$n^{-2}$, we have
\begin{align}\label{eq:23}
\left\|\ghat-\hhat\right\|_1
=\int_{\mathbb{R}^3}\left|
\ghat(\alpha,\beta,\gamma)-\hhat(\alpha,\beta,\gamma)
\right|\,\mathrm{d}\alpha\beta\gamma
\le \frac{\log^{16} n}{n^4}\;.
\end{align}
\end{lemma}

This in turn is used in the proof of
Theorem~\ref{thm:conditional-clt} by applying
Fourier inversion formulas. Similar arguments
allow to directly prove Theorem~\ref{thm:one-fixed-equivalent}.
Full proofs of Theorems~\ref{thm:conditional-clt}
and~\ref{thm:one-fixed-equivalent} and Lemmas~\ref{lem:fhat-bounded}
and~\ref{lem:characteristic-close} are given in
Section~\ref{sec:clt-proof}. While we presented a rough outline
here, there are multiple technical details to take care of.
Most (but not all) of them are handled in a way which is 
identical to or inspired by the CLT proof in~\cite{Pol17}.
Except for adaptations of arguments to the continuous
setting, among the main differences to Polymath are: using a 
grid of points and interpolating by Lemma~\ref{lem:interpolation}; applying Poisson limit
theorem in addition to Berry-Esseen in the proof
of Lemma~\ref{lem:intermediate}; and a more careful
estimate required in the final proof of Theorem~\ref{thm:conditional-clt}.

\subsection{Bounding the covariance}
\label{sec:covariance-bound}

If the random variables $(U_A^{*n},U_B^{*n},V^{*n})$ 
behave like Gaussians, the way to show that
$\Pr[U_A^{*n},U_B^{*n}>0\mid V^{*n}=\EE V^{*n}]$ is significantly
different from $1/4$ is to bound away from zero
the relevant (conditional) covariance 
between $U_A$ and $U_B$. More precisely,
by~\eqref{eq:28} we are interested in the expression
$CV_{AB}-(CV_A CV_B/\Var V)$. Let us make more concise notation
\begin{align}\label{eq:39}
    \CVABcond:=CV_{AB}-\frac{CV_A CV_B}{\Var V}\;,
    \qquad
    \VarAcond:=\Var_A-\frac{CV_A^2}{\Var V}\;,
    \qquad
    \VarBcond:=\Var_B-\frac{CV_B^2}{\Var V}\;.
\end{align}
We now give two main lemmas that we use in the proof
of Theorem~\ref{thm:two-fixed-equivalent}.
In the following the probability is over the choice of random balanced dice $A,B$ with faces in $[0,1]$.

\begin{lemma} \label{lem:var-upper-bound}
For every $\eps>0$ there exists $K>0$ such that
\[
    \Pr\big[\VarAcond >Kn\big]<\eps\;.
\]
\end{lemma}

\begin{lemma} \label{lem:cv-lower-bound}
There exists $\eps>0$ such that
\[
    \Pr\big[\left|\CVABcond\right|>\eps n\big]>\eps\;. 
\]
\end{lemma}

\paragraph{Proof strategy for Lemmas~\ref{lem:var-upper-bound}
and~\ref{lem:cv-lower-bound}}
To start with, as in Section~\ref{sec:clt} 
we will argue that throughout the whole proof
we can assume faces 
come from a uniform distribution on an interval different from $[0,1]$. This time, for consistency
with~\cite{HMRZ20} and to simplify some calculations
we take the faces to be uniform in $[-\sqrt{3},\sqrt{3}]$.
That is, the single-face distribution is centered
with variance $\Var V=\Var H=1$. As in Section~\ref{sec:clt},
this does not change the joint distribution
of the random variables
$\VarAcond$, $\VarBcond$ and $\CVABcond$.
Therefore, from now on we assume the faces 
and random variable $V$ are
uniform in $[-\sqrt{3},\sqrt{3}]$ and that a balanced
die has face-sum zero.

Lemmas~\ref{lem:var-upper-bound}
and~\ref{lem:cv-lower-bound} are proved by the first 
and second-moment methods. More precisely, in 
Section~\ref{sec:moments} we establish the following
moment bounds:

\begin{lemma}\label{lem:var-2nd-moment}
$\EE \Var_A^2=\cO(n^2)$.
\end{lemma}
\begin{lemma}\label{lem:cv-2nd-moment}
$\EE \CVABcond^2=\Omega(n^2)$.
\end{lemma}

These lemmas are proved through a careful calculation.
If $A$ and $B$ are random (not balanced) dice,
the expressions $\EE\Var_A^2$ and $\EE \CVABcond^2$ are polynomials in $n$ that
we can explicitly compute. Adding conditioning on 
balanced dice changes the formulas, but if we use a 
precise version of the local central limit theorem for 
densities  from~\cite{Pet75} we can 
sufficiently control the error terms and show
that qualitatively the moment bounds remain
the same. The technique we use is very similar
to the one employed in~\cite{HMRZ20}
(in particular in the proof of Proposition~1
therein), 
but we need to use more error terms than them.
On the other hand, some expressions simplify since we consider
a special case of the uniform distribution.
The proofs of both lemmas are left to Section~\ref{sec:moments}.

Given Lemmas~\ref{lem:var-2nd-moment} 
and~\ref{lem:cv-2nd-moment} it is not difficult
to establish Lemmas~\ref{lem:var-upper-bound}
and~\ref{lem:cv-lower-bound}:

\begin{proof}[Proof of Lemma~\ref{lem:var-upper-bound}]
Fix $\eps>0$. By Lemma~\ref{lem:var-2nd-moment}
we have $\EE \Var_A^2=\cO(n^2)$.
But now by Markov
\begin{align*}
    \Pr\left[\VarAcond>Kn\right]
    \le\Pr\left[\Var_A>Kn\right]
    =\Pr\left[\Var^2_A>K^2n^2\right]
    =\cO\left(\frac{1}{K^2}\right)<\eps\;,
\end{align*}
where we can certainly satisfy the last inequality
by choosing $K$ large enough.
\end{proof}

\begin{proof}[Proof of Lemma~\ref{lem:cv-lower-bound}]
By Lemma~\ref{lem:cv-2nd-moment} we have
$\EE\CVABcond^2=\Omega(n^2)$. On the other hand,
applying Lemma~\ref{lem:var-2nd-moment}
and Cauchy-Schwarz multiple times, we also see that
\begin{align*}
    \EE CV_{AB}^4
    &\le
    \EE\left[\Var_A^2\Var_B^2\right]
    =\EEs^2\Var_A^2=\cO(n^4)\;,\\
    \EE\,(CV_ACV_B)^4
    &=\EEs^2 CV_A^4
    \le \EEs^2\Var_A^2=\cO(n^4)\;,
\end{align*}
and therefore also
\begin{align*}
    \EE\CVABcond^4
    =\EE\,(CV_{AB}-CV_ACV_B)^4
    \le 16\left(\EE CV_{AB}^4+\EE\, (CV_A CV_B)^4\right)
    =\cO(n^4)\;.
\end{align*}
Recalling the Paley-Zygmund inequality
$\Pr[Z>c]\ge(1-c/\EE Z)^2\frac{(\EE Z)^2}{\EE Z^2}$
for $Z\ge 0$ and $0\le c\le \EE Z$, we apply it
to $Z=\CVABcond^2$ and get
\begin{align*}
    \Pr\big[\left|\CVABcond\right|>\eps n\big]
    =\Pr\big[\CVABcond^2>\eps^2n^2\big]
    \ge \left(1-\frac{\eps^2n^2}{\EE\CVABcond^2}\right)^2
    \cdot\frac{\EEs^2\CVABcond^2}{\EE\CVABcond^4}
    >\eps>0\;,
\end{align*}
where the last inequality is true if $\eps$ is chosen to be
a small enough absolute constant.
\end{proof}

\subsection{Proof of Theorem~\ref{thm:two-fixed-equivalent}}

We present the remaining steps in the
proof of Theorem~\ref{thm:two-fixed-equivalent}.
Recall the notation we defined in~\eqref{eq:39}.
From Lemmas~\ref{lem:var-upper-bound} and~\ref{lem:cv-lower-bound}
we know that for some constants $\eps, K>0$, with probability
at least $\eps$, random choice of balanced $A$ and $B$ gives
the dice such that
$\VarAcond, \VarBcond\le Kn$ and $|\CVABcond|>\eps n$ all hold.
In fact, this means that 
at least one of $\CVABcond>\eps n$ or $\CVABcond<-\eps n$ holds
with probability at least $\eps/2$.
Let us consider the former case $\CVABcond>\eps n$, the latter being 
symmetric.

Note that by Claim~\ref{cl:conditional-variance}
random variables $G_A^{*n}$ and $G_B^{*n}$ conditioned
on $H^{*n}=n/2$ are joint centered Gaussians with variances
respectively $n\VarAcond$ and $n\VarBcond$ and
covariance $n\CVABcond$. Let
\begin{align*}
    \Gamma_\rho(a,b):=\Pr\big[\cG_1< a\text{ and }\cG_2< b\big]
\end{align*}
where $\cG_1$ and $\cG_2$ are joint, centered, unit variance Gaussians
with covariance $\rho$. We will need a couple of standard
properties of $\Gamma_\rho$ summarized in the following:
\begin{claim}\label{cl:gaussian-gamma}
Using the notation above, we have
\begin{align*}
    \Gamma_\rho(0,0)=\Pr[\cG_1,\cG_2<0]=\Pr[\cG_1,\cG_2>0]\;.
\end{align*}
Furthermore, for $-1\le\rho\le 1$ the value of
$\Gamma_\rho(0,0)$ is a strictly increasing
continuous function of $\rho$, such
that $\Gamma_{-1}(0,0)=0$, $\Gamma_0(0,0)=1/4$ and
$\Gamma_1(0,0)=1/2$.
\end{claim}

By the discussion above, letting
\begin{align*}
    \rho:=\frac{n\CVABcond}{\sqrt{n^2\VarAcond\VarBcond}}
    >\frac{\eps}{K}>0\;,
\end{align*}
we have
\begin{align*}
    \Pr\left[G_A^{*n},G_B^{*n}>0\mid H^{*n}=\frac{n}{2}\right]
    =\Gamma_\rho(0,0)>\frac{1}{4}+\delta
\end{align*}
for some absolute constant $\delta>0$. Furthermore, by
Cauchy-Schwarz,
\begin{align*}
    \VarAcond\ge\frac{\CVABcond^2}{\VarBcond}
    \ge\frac{\eps^2}{K}\cdot n
\end{align*}
and similarly $\VarBcond\ge\eps^2n/K$. Consequently,
except with probability $1/n$, the conclusion
of Theorem~\ref{thm:conditional-clt} holds. Overall,
with probability at least $\eps/2-1/n>\eps/4$
we have
\begin{align*}
\Pr\left[U_A^{*n},U_B^{*n}>0\mcond V^{*n}=\frac{n}{2}\right]
&> \frac{1}{4}+\delta-\Otilde\left(\frac{1}{\sqrt{n}}\right)
>\frac{1}{4}+\delta/2\;.\pushQED{\qed}\qedhere\popQED
\end{align*}

\section{CLT for two dice}\label{sec:clt-proof}

In this section we prove the results concerned with our central limit
theorem for dice, that is Theorems~\ref{thm:one-fixed-equivalent}
and~\ref{thm:conditional-clt}, Claim~\ref{cl:conditional-variance} and Lemmas~\ref{lem:fhat-bounded} and~\ref{lem:characteristic-close}.
We remind that in this section we assume that the dice have
faces which are iid uniform in $[0,n]$, a balanced die has face-sum
$n^2/2$ and $g_A(x)=f_A(x)-x$.
The method for the most part follows the Polymath draft~\cite{Pol17}
with adaptations to the continuous setting.

The proof takes some space and we separate it into several modules.
We start with some basic machinery of 
characteristic functions and Gaussians,
as well as some useful tools and concentration bounds.
Then, we give concentration bounds on $U_A$ and $U_B$
and related Lipschitz bounds on the characteristic
function $\fhat$. Subsequently, we address
the decay of $\fhat$ in different regimes, e.g., small
$|\alpha|$ and $|\beta|$ and/or large $|\gamma|$.
Finally, we put everything together in the proofs of the aforementioned
lemmas and theorems.

\subsection{Characteristic functions and joint Gaussians}
\label{sec:facts-cf}

In this section we gather some basic claims
about characteristic functions and Gaussian random
variables.
Before we proceed, let us make a more systematic exposition of
our notation. Recall from~\eqref{eq:41}
that the characteristic function of the 
convoluted Gaussians $(G_A^{*n}, G_B^{*n}, H^{*n}-n^2/2)$ is denoted by $\ghat(\alpha,\beta,\gamma)$.
Accordingly, we let $g(x,y,z)$ be the joint density
function of these random variables. Furthermore,
we let $g(z)$ to be the marginal density of
$H^{*n}-n^2/2$.
Analogously, recall that $\uhat(\alpha,\beta,\gamma)$
is the characteristic function of convoluted random variables $(U_A^{*n},U_B^{*n},V^{*n}-n^2/2)$
(as we will discuss shortly, since the first two random variables lie on a lattice, the relevant domain
is $|\alpha|,|\beta|\le 1/2$). Only the third of these
random variables is fully continuous, so there is
no joint density, but we denote by $u(z)$ the marginal
density of $V^{*n}-n^2/2$.

First, we state a claim that follows from the definition
of $\ghat$ and the formula for multivariate Gaussian
characteristic function
$\EE\exp\left(i \Vec{t}\cdot\Vec{G}\right)=
\exp\left(-\frac{1}{2}\Vec{t}^T\Sigma\Vec{t}\right)$,
where $\Sigma$ is the covariance matrix:
\begin{claim}\label{cl:ghat-formula}
$\ghat(\alpha,\beta,\gamma)=\exp(-nQ)$, where
\begin{align}\label{eq:38}
Q:=Q(\alpha,\beta,\gamma)=2\pi^2\big(
\alpha^2\Var_A+\beta^2\Var_B+\gamma^2\Var V+
2\alpha\beta CV_{AB}+2\alpha\gamma CV_A
+2\beta\gamma CV_B\big)\;.
\end{align}
\end{claim}

Turning to the other random variables,
the distribution of $(U_A^{*n},U_B^{*n},V^{*n})$ is a particular
kind of a discrete-continuous mix. More concretely, recall
from Claim~\ref{cl:simple-g} that $V^{*n}$
is continuous, and $U_A^{*n}, U_B^{*n}$ are both supported
on $\mathbb{Z}-V^{*n}$. 
In particular, since we
assumed $n$ is odd, if $V^{*n}=n^2/2$, then
$U^{*n}_A,U^{*n}_B\in\mathbb{Z}+1/2$.
Therefore, their inverse Fourier transform
can be checked to
take a particular form that we state below:
\begin{claim}[Inverse Fourier transform]\label{cl:inverse-fourier}
Let $u(z)$ denote the density function of $V^{*n}-n^2/2$.
For fixed dice $A$ and $B$, for every $a,b\in\mathbb{Z}+1/2$
we have the formula
\begin{align}
    \Pr\left[U_A^{*{n}}=a,U_B^{*n}=b\;\middle\vert\;V^{*n}=\frac{n^2}{2}
    \right]\cdot u(0)&=
    \int_{-1/2}^{1/2}\int_{-1/2}^{1/2}\int_{-\infty}^\infty
    \fhat(\alpha,\beta,\gamma)^n
    e(-\alpha a-\beta b)\,\mathrm{d}\gamma\mathrm{d}\beta
    \mathrm{d}\alpha
    \nonumber\\
    &=
    \int_{\mathbb{R}^3}\hhat(\alpha,\beta,\gamma)
    e(-\alpha a-\beta b)\,\mathrm{d}\alpha\beta\gamma
    \;.
    \label{eq:60}
\end{align}
At the same time, if $g(x,y,z)$ denotes the joint density
of $(G_A^{*n},G_B^{*n},H^{*n}-n^2/2)$, then we have
\begin{align}\label{eq:61}
    g(a,b,0)=\int_{\mathbb{R}^3}
    \ghat(\alpha,\beta,\gamma)
    e(-\alpha a-\beta b)\,\mathrm{d}\alpha\beta\gamma
    \;.
\end{align}
\end{claim}
\begin{proof}
The formula in~\eqref{eq:61} is just standard Fourier inversion for an integrable
function $\ghat$. As for~\eqref{eq:60}, later on in Lemma~\ref{lem:large-gamma}
we will prove that
$|\fhat(\alpha,\beta,\gamma)|\le 1/|\gamma-\alpha-\beta|$, in particular
for $-1/2\le\alpha,\beta\le 1/2$ we have
$|\fhat(\alpha,\beta,\gamma)|\le\frac{2}{|\gamma|}$
whenever $|\gamma|\ge 2$.
Therefore, we have
$\left|\uhat(\alpha,\beta,\gamma)\right|\le\left(2/|\gamma|\right)^n$
and the characteristic function
$\uhat$ is integrable for $n>1$. Accordingly, the integral 
in~\eqref{eq:60} is well-defined and we are free to change the
order of integration and apply Fourier inversion in what follows.

The rest of the justification for~\eqref{eq:60} is a calculation that we include 
for completeness.
Since $U_A^{*n},U_B^{*n}\in\mathbb{Z}-V^{*n}$,
let $W_A:=U_A^{*n}+V^{*n}-n^2/2$ and
$W_B:=U_B^{*n}+V^{*n}-n^2/2$. Since we assumed
$n$ odd, that gives
$W_A,W_B\in\mathbb{Z}+1/2$ and
also
\begin{align*}
    \uhat(\alpha,\beta,\gamma)
    &=\EE e\big(\alpha U_A^{*n}+\beta U_B^{*n}
    +\gamma (V^{*n}-n^2/2)\big)\\
    &=\EE e\big(\alpha W_A+\beta W_B
    +(\gamma-\alpha-\beta)(V^{*n}-n^2/2)\big)
    =:\uhat'(\alpha,\beta,\gamma-\alpha-\beta)\;,
\end{align*}
where $\uhat'$ is the characteristic function
of random tuple $(W_A,W_B,V^{*n}-n^2/2)\in
(\mathbb{Z}+1/2)\times(\mathbb{Z}+1/2)\times\mathbb{R}$.
Let $w(x,y,z):=\Pr[W_A=x\text{ and } W_B=y\cond V^{*n}-n^2/2=z]\cdot u(z)$.
Now we calculate, starting from the right-hand side of~\eqref{eq:60},
for $a,b\in\mathbb{Z}+1/2$,
\begin{align*}
    &\int_{[-1/2,1/2]^2\times \mathbb{R}}\uhat(\alpha,\beta,\gamma)e(-\alpha a-\beta b)
    \,\mathrm{d}\alpha\beta\gamma
    =
    \int_{[-1/2,1/2]^2\times \mathbb{R}}\uhat'(\alpha,\beta,\gamma-\alpha-\beta)e(-\alpha a-\beta b)
    \,\mathrm{d}\alpha\beta\gamma\\
    &\qquad\qquad=
    \int_{[-1/2,1/2]^2\times \mathbb{R}}\sum_{x,y\in\mathbb{Z}+1/2}\int_{\mathbb{R}}
    w(x,y,z)
    e\big(-\alpha(a-x)-\beta(b-y)+(\gamma-\alpha-\beta)z\big)
    \,\mathrm{d}z\,\mathrm{d}\alpha\beta\gamma\\
    &\qquad\qquad=
    \sum_{x,y\in\mathbb{Z}+1/2}
    \int_{[-1/2,1/2]^2}
    e\big(-\alpha(a-x)-\beta(b-y)\big)
    \int_{\mathbb{R}}
    \int_{\mathbb{R}}
    w(x,y,z)e\big((\gamma-\alpha-\beta)z\big)
    \,\mathrm{d}\gamma\,\mathrm{d}z\,\mathrm{d}{\alpha\beta}\\
    &\qquad\qquad=
    \sum_{x,y\in\mathbb{Z}+1/2}
    \int_{[-1/2,1/2]^2}
    e\big(-\alpha(a-x)-\beta(b-y)\big)
    \int_{\mathbb{R}}
    \int_{\mathbb{R}}
    w(x,y,z)e(\gamma z)
    \,\mathrm{d}\gamma\,\mathrm{d}z\,\mathrm{d}{\alpha\beta}\;.
\end{align*}
For fixed $x,y$, by the Fourier inversion formula we have
\begin{align*}
    \int_{\mathbb{R}^2}w(x,y,z)e(\gamma z)
    \,\mathrm{d}\gamma z=w(x,y,0)\;.
\end{align*}
Substituting and continuing,
\begin{align*}
    \int_{[-1/2,1/2]^2\times\mathbb{R}}
    \uhat(\alpha,\beta,\gamma)e(-\alpha a-\beta b)
    \,\mathrm{d}\alpha\beta\gamma
    &=
    \sum_{x,y\in\mathbb{Z}+1/2}w(x,y,0)
    \int_{[-1/2,1/2]^2}e\big(-\alpha(a-x)-\beta(b-y)\big)
    \,\mathrm{d}\alpha\beta\\
    &=w(a,b,0)\;,
\end{align*}
where in the end we used the Fourier inversion for integers since
$a-x,b-y\in\mathbb{Z}$. But now we are done, since
\begin{align*}
    w(a,b,0)
    &=\Pr[W_A=a\text{ and } W_B=b\cond V^{*n}=n^2/2]u(0)
    =\Pr[U_A^{*n}=a\text{ and } U_B^{*n}=b\cond V^{*n}=n^2/2]u(0)\;.
    \qedhere
\end{align*}
\end{proof}

We also state the standard connection between 
characteristic function
$\fhat$ and respective moments, which follows from an
application of Taylor's theorem to $e(t)$ giving
$|e(t)-(1+2\pi it-2\pi^2t^2)|\le 4\pi^3|t|^3/3$: 
\begin{claim}\label{cl:fhat-moments}
We have 
$\fhat(\alpha,\beta,\gamma)=1-Q+R$,
where $Q$ is defined in~\eqref{eq:38}
and $R=R(\alpha,\beta,\gamma)$ is such that
\begin{align*}
    |R|&\le\frac{4\pi^3}{3}\big(
    |\alpha|\|U_A\|_\infty
    +|\beta|\|U_B\|_\infty
    +|\gamma|\|V-n/2\|_\infty
    \big)^3\\
    &\le1200\left(|\alpha|^3\|U_A\|_\infty^3
    +|\beta|^3\|U_B\|_\infty^3
    +|\gamma|^3\frac{n^3}{8}\right)\;.
\end{align*}
\end{claim}


We will also need some elementary bounds on the densities of $H^{*n}$
and $V^{*n}$ around their means.
While the anticoncentration
bound in the third point in Claim \ref{cl:h-lower-bound} is not optimal, it will be enough for our purposes.
\begin{claim}\label{cl:h-lower-bound}
Let $g(z)$ be the density function of the random variable
$H^{*n}-n^2/2$ and $u(z)$ the density of $V^{*n}-n^2/2$.
Then, we have:
\begin{enumerate}
    \item $g(0),u(0)\ge \frac{1}{4n\sqrt{n}}$.
    \item For $0<\eps\le n\sqrt{n}$, it holds that
    $\Pr\big[|V^{*n}-n^2/2|\le\eps\big]\ge \eps/4n\sqrt{n}$.
    \item For any $\eps_1<\eps_2$, it holds that
    $\Pr\big[\eps_1\le V^{*n}\le\eps_2\big]\le(\eps_2-\eps_1)/n$.
\end{enumerate}
\end{claim}
\begin{proof}
Since $H^{*n}$ is Gaussian, clearly the density $g(z)$ achieves
maximum at $z=0$. Similarly, since the density 
of $V^{*n}$ is a convolution of $n$ symmetric,
unimodal densities, $u(z)$ achieves maximum at $z=0$;
in fact, $u(z)$ is unimodal and symmetric around zero.
On the other hand, by standard concentration
inequalities
(cf.~Claim~\ref{cl:chernoff}, recall that $V\in[0,n]$ and $\Var H^{*n}=n^3/12$)
\begin{align}\label{eq:35}
\Pr\left[\left|H^{*n}-\frac{n^2}{2}\right|\ge n\sqrt{n}\right],
\Pr\left[\left|V^{*n}-\frac{n^2}{2}\right|\ge n\sqrt{n}\right]
\le2\exp\left(-2\right)\le\frac{1}{2}
\end{align}
and hence $\int_{-n\sqrt{n}}^{n\sqrt{n}}g(z)\,\mathrm{d}z\ge 1/2$ and similarly for $u(z)$.
Consequently, $g(0),u(0)\ge 1/4n\sqrt{n}$, establishing the first point.

As for the second point, we proceed by contradiction. If
$\Pr\big[|V^{*n}-n^2/2|\le\eps\big]<\eps/4n\sqrt{n}$, by unimodality
and symmetry of $u$, $\int_{-n\sqrt{n}}^{n\sqrt{n}}u(z)\,\mathrm{d}z\ge 1/2$ implies
\begin{align*}
u(\eps)=u(-\eps)\ge \frac{1/2-\eps/4n\sqrt{n}}{2n\sqrt{n}}
\ge\frac{1}{8n\sqrt{n}}\;,
\end{align*}
which by the unimodality of $u(z)$ gives 
$\Pr\big[|V^{*n}-n^2/2|\le\eps\big]=\int_{-\eps}^\eps u(z)\,\mathrm{d}z
\ge\eps/4n\sqrt{n}$ anyway.

Finally, for the third point the bound follows from the fact that the density of $V$
is uniformly bounded by $1/n$ and the convolution cannot increase this bound,
hence $\int_{\eps_1}^{\eps_2}u(z)\,\mathrm{d}z\le(\eps_2-\eps_1)/n$.
\end{proof}

We will also use a standard anti-concentration
estimate for the Gaussians that follows from the
formula for Gaussian density:
\begin{claim}\label{cl:gaussian-anti}
Let $G$ be a Gaussian random variable with variance
at least $\sigma^2$. Then, for all
$a,b$, we have
\begin{align*}
    \Pr[a\le G\le b]\le\frac{b-a}{\sqrt{2\pi}\sigma}\;.
\end{align*}
\end{claim}

Finally, we derive the formulas for conditional variances
of $(G_A^{*n},G_B^{*n},H^{*n})$ given in Claim~\ref{cl:conditional-variance}:
\begin{proof}[Proof of Claim~\ref{cl:conditional-variance}]
Note that by considering the relevant variances and covariances
and using the Hilbert space structure of joint Gaussians,
random variables $G_A$ and $G_B$ can be written as
\begin{align*}
    G_A&=\sqrt{\Var_A-\frac{CV_A^2}{\Var V}}\cdot\mathcal{G}_A+
    \frac{CV_A}{\Var V}\cdot \left(H-\frac{n}{2}\right)\;,\\
    G_B&=\sqrt{\Var_B-\frac{CV_B^2}{\Var V}}\cdot\mathcal{G}_B+
    \frac{CV_B}{\Var V}\cdot \left(H-\frac{n}{2}\right)\;,
\end{align*}
where $\cG_A$ and $\cG_B$ are joint centered 
Gaussians independent of $H$, each with variance one.
Let $\alpha_A:=\sqrt{\Var_A-(CV_A^2/\Var V})$ and
$\alpha_B:=\sqrt{\Var_B-(CV_B^2/\Var V)}$
Accordingly, we also have
\begin{align}
    G_A^{*n}&=\alpha_A\cdot\mathcal{G}_A^{*n}+
    \frac{CV_A}{\Var V}\cdot \left(H^{*n}-\frac{n^2}{2}\right)\;,
    \label{eq:29}\\
    G_B^{*n}&=\alpha_B\cdot\mathcal{G}_B^{*n}+
    \frac{CV_B}{\Var V}\cdot \left(H^{*n}-\frac{n^2}{2}\right)\;.
    \label{eq:30}
\end{align}
The fact that after conditioning $G_A^{*n}$ and $G_B^{*n}$
are centered joint Gaussians, as well as~\eqref{eq:26} and~\eqref{eq:27}
all follow (note that $\Var\cG_A^{*n}=\Var\cG_B^{*n}=n$).

As for the covariance in~\eqref{eq:28}, first
by rearranging~\eqref{eq:29}
and~\eqref{eq:30} we observe that
\begin{align*}
\Cov[\cG_A^{*n},\cG_B^{*n}]
&=\frac{n}{\alpha_A\alpha_B}\left(
CV_{AB}-\frac{CV_ACV_B}{\Var V}
\right)\;.
\end{align*}
Finally, we invoke~\eqref{eq:29} 
and~\eqref{eq:30} again to see that
\begin{align*}
    \Cov\left[G_A^{*n},G_B^{*n}\;\middle\vert\;H^{*n}=\frac{n^2}{2}\right]
    &=\alpha_A\alpha_B
    \cdot\Cov\left[\mathcal{G}_A^{*n},\mathcal{G}_B^{*n}
    \right]
    =n\left(CV_{AB}-\frac{CV_ACV_B}{\Var V}\right)\;.
    \qedhere
\end{align*}
\end{proof}

\subsection{Tools}
\label{sec:tools}

In this section we present a few tools, mostly imported
from~\cite{Pol17}, that we will use in various proofs in
this section.
Many of them will be used to prove Lemma~\ref{lem:fhat-bounded}.
We start with referencing some standard concentration
bounds:
\begin{claim}\label{cl:chernoff}
If $X_1,\ldots,X_n$ are independent random variables
in $[0,1]$ and $t>0$, then
\begin{align*}
    \Pr\left[\left|\sum_{i=1}^nX_i-\sum_{i=1}^n\EE X_i\right|
    \ge t\right]\le 2\exp(-2t^2/n)\;.
\end{align*}
On the other hand, if $G$ is a centered Gaussian of variance
$\sigma^2$, then, for $t>0$,
\begin{align*}
    \Pr\big[|G|\ge t\big]\le 2\exp(-t^2/2\sigma^2)\;.
\end{align*}
\end{claim}

We also need other estimates of sums of iid random variables. 
In particular, we will use explicit
error bounds for both the central limit theorem and the Poisson limit theorem:
\begin{lemma}[Berry-Esseen theorem]
\label{lem:berry-esseen}
Let $X_1,\ldots,X_n$ be iid random variables with $\EE X_i=0$, 
$\EE X_i^2=\sigma^2$ and $\EE|X_i|^3=\rho$. 
Let $X:=\sum_{i=1}^n X_i$ and $N$ be a standard normal. Then,
for every real $a,b$,
\[
\left|\Pr[a\le X\le b]-\Pr\left[\frac{a}{\sigma\sqrt{n}}\le 
N\le\frac{b}{\sigma\sqrt{n}}\right]
\right|\le\frac{\rho}{\sigma^3\sqrt{n}}\;.
\]
\end{lemma}

\begin{lemma}[Poisson limit theorem convergence rate,
see~(1.1) in~\cite{BH84}]
\label{lem:poisson}
Let $X_1,\ldots,X_n$ be iid Bernoulli random variables with 
$\EE X_i=\lambda/n$. Let $X:=\sum_{i=1}^n X_i$ and let
$Z$ be a Poisson random variable with mean $\lambda$. Then,
for every $S\subseteq \mathbb{N}$,
\[
\Big|\Pr[X\in S]-\Pr[Z\in S]\Big|\le\frac{\lambda^2}{n}\;.
\]
\end{lemma}

We use a technical lemma from the Polymath paper connecting the
distance to integers to the values of $e(\cdot)$:
\begin{lemma}[Lemma 5.9 in~\cite{Pol17}]
\label{lem:e-to-mod1}
For real numbers $\theta_1$ and $\theta_2$, we have
\[
\frac{1}{2}\left|e(\theta_1)+e(\theta_2)\right|
\le 1-d_{\mathbb{Z}}(\theta_1-\theta_2)^2\;.
\]
Similarly, for any $\theta_1,\theta_2,\theta_3$ and $\theta_4$
it holds that
\[
\frac{1}{4}\left|e(\theta_1)+e(\theta_2)+e(\theta_3)+e(\theta_4)
\right|\le 1-d_{\mathbb{Z}}(
\theta_1-\theta_2+\theta_3-\theta_4)^2/4\;.
\]
\end{lemma}

We also employ a concentration bound analogous to Lemmas~5.4 and~5.8 in~\cite{Pol17}. Rather than prove it directly as Polymath does,
we resort to the Azuma's inequality:
\begin{lemma}[Azuma's inequality]
\label{lem:azuma}
Let random variables $Y_0,\ldots,Y_n$ form a submartingale with
respect to a filtration $\mathcal{F}_0,\ldots,\mathcal{F}_n$
and with bounded differences, i.e., the
random variable $Y_i$ is measurable with respect
to the $\sigma$-algebra $\mathcal{F}_i$ and
\[
\EE\left[Y_{i+1}-Y_i\mid \mathcal{F}_i\right]\ge 0\;,\qquad\qquad
|Y_{i+1}-Y_i|\le c
\] 
both hold almost surely. Then, for all $n$ and for every $\eps>0$,
\[
\Pr[Y_n\le Y_0-\eps]\le\exp\left(
-\frac{\eps^2}{2cn}
\right)\;.
\]
\end{lemma}

\begin{lemma}[Azuma except with small probability]
\label{lem:modified-azuma}
As in Lemma~\ref{lem:azuma}, let
$Y_0,\ldots,Y_n$ be a sequence of random variables 
adapted to a filtration $\mathcal{F}_0,\ldots,\mathcal{F}_n$
and with bounded differences $|Y_{i+1}-Y_i|\le 1$.
Furthermore, let $D_{i+1}$ denote the event
\[
\EE[Y_{i+1}-Y_i\mid\mathcal{F}_i]< 0\;.
\] 
Then, for every $\eps>0$,
\begin{align}\label{eq:07}
\Pr[Y_n\le Y_0-\eps]\le
\exp\left(-\frac{\eps^2}{2n}\right)+
\Pr\left[\bigcup_{i=1}^{n} D_i\right]\;.
\end{align}
\end{lemma}

\begin{proof}
The only difference to Lemma~\ref{lem:azuma} is that random
variables $Y_0,\ldots,Y_n$ do not always satisfy the
submartingale property.
To fix this, let $Y'_0,\ldots,Y'_n$ be a sequence of random variables given by
\[
Y'_0:=Y_0\;,
\qquad\qquad
Y'_{i+1}:=\begin{cases}
Y_{i+1}&\text{if none of events $D_1,\ldots,D_{i+1}$ happened,}\\
Y'_{i}&\text{otherwise.}
\end{cases}
\]
Keeping in mind the definitions of $D_i$ and $Y'_i$, by induction we see that $Y'_0,\ldots,Y'_n$ form a bounded difference submartingale with
respect to the filtration $\mathcal{F}_0,\ldots,\mathcal{F}_n$.
Consequently, by Lemma~\ref{lem:azuma},
\[
\Pr[Y'_n\le Y'_0-\eps]\le\exp\left(-\frac{\eps^2}{2n}\right)\;. 
\]
On the other hand, it also holds that if
none of $D_1,\ldots,D_{n}$ happened, then $Y'_n=Y_n$.
Hence,
\[
\Pr[Y_n\le Y_0-\eps]\le\Pr[Y'_n\le Y'_0-\eps]+\Pr[Y'_n\ne Y_n]
\le\exp\left(-\frac{\eps^2}{2n}\right)+
\Pr\left[\bigcup_{i=1}^{n} D_i\right]\;.
\qedhere\]
\end{proof}

Finally, in the proof of Lemma~\ref{lem:characteristic-close}
we are going to employ Lemma~6.2 from~\cite{Pol17}
in order to approximate $(1-Q+R)^n$ with $\exp(-nQ)$
for some $R\ll Q\ll 1$:
\begin{lemma}[Lemma~6.2 in \cite{Pol17}]
\label{lem:exp-nq-approx-real}
Let $Q,R\in\mathbb{R}$ such
that $Q^2,|R|\le 1/4n$. Then,
\[
\left|1-\frac{(1-Q+R)^n}{\exp(-nQ)}\right|,
\left|1-\frac{\exp(-nQ)}{(1-Q+R)^n}\right|
\le 4n(Q^2+|R|)\;.
\]
\end{lemma}
Actually, we need a complex version of 
Lemma~\ref{lem:exp-nq-approx-real} which uses the real
version in the proof:
\begin{lemma}
\label{lem:exp-nq-approx}
Let $Q\in\mathbb{R}$, $R\in\mathbb{C}$ such
that $Q^2,|R|\le 1/100n$. Then,
\[
\left|1-\frac{(1-Q+R)^n}{\exp(-nQ)}\right|,
\left|1-\frac{\exp(-nQ)}{(1-Q+R)^n}\right|
\le 40n(Q^2+|R|)\;.
\]
\end{lemma}
\begin{proof}
Let us write $R=\alpha+\beta i$ and denote
$L:=(1-Q+R)^n$, $z:=1-L/\exp(-nQ)$. We proceed to estimate
the imaginary and real part of $z$:
\begin{align*}
\left|\Im L\right|
&=\left|\Im\sum_{k=0}^n\binom{n}{k}(1-Q)^{n-k}
\sum_{j=0}^k\binom{k}{j}\alpha^{k-j}(\beta i)^j\right|
\le\sum_{k=1}^n\binom{n}{k}(1-Q)^{n-k}
\sum_{j=0}^k\binom{k}{j}|\alpha|^{k-j}|\beta|^j\\
&=(1-Q+|\alpha|+|\beta|)^n-(1-Q)^n\;,
\end{align*}
and consequently, applying Lemma~\ref{lem:exp-nq-approx-real}
to $(1-Q+|\alpha|+|\beta|)^n$ and $(1-Q)^n$,
\begin{align}
|\Im z|
&=\frac{|\Im L|}{\exp(-nQ)}
\le\frac{(1-Q+|\alpha|+|\beta|)^n}{\exp(-nQ)}
-\frac{(1-Q)^n}{\exp(-nQ)}\nonumber\\
&\le 1+4n(Q^2+2|R|)-1+4nQ^2=8n(Q^2+|R|)\;.
\label{eq:15}
\end{align}
By a similar argument, we also establish
\begin{align*}
    \left|\Re L-(1-Q)^n\right|
    \le(1-Q+|\alpha|+|\beta|)^n-(1-Q)^n
\end{align*}
and
\begin{align}
    \left|\Re z\right|
    &=\left|
    1-\frac{(1-Q)^n}{\exp(-nQ)}
    + \frac{(1-Q)^n-\Re L}{\exp(-nQ)}
    \right|
    \nonumber\\
    &\le\left|1-\frac{(1-Q)^n}{\exp(-nQ)}
    \right|+\frac{(1-Q+|\alpha|+|\beta|)^n}{\exp(-nQ)}
    -\frac{(1-Q)^n}{\exp(-nQ)}\le 12n(Q^2+|R|)\;.
    \label{eq:16}
\end{align}
Equations \eqref{eq:15} and~\eqref{eq:16} together imply
\begin{align*}
    \left|1-\frac{(1-Q+R)^n}{\exp(-nQ)}\right|
    \le 20n(Q^2+|R|)\;,
\end{align*}
as we wanted. Finally, we obtain the other inequality with
\begin{align*}
    \left|1-\frac{\exp(-nQ)}{(1-Q+R)^n}\right|
    &=|z|\frac{\exp(-nQ)}{|1-Q+R|^n}
    \le 40n(Q^2+|R|)\;.\qedhere
\end{align*}
\end{proof}

\subsection{Interpolating 
\texorpdfstring{$\fhat$}{f hat} 
and concentration
of \texorpdfstring{$U$}{U}}

We continue to introduce some more technical tools. We
start with some bounds on the characteristic function $\fhat$
that we will later use to interpolate values of $\fhat$
by adjacent values on a discrete grid:

\begin{claim}\label{cl:interpolation-a}
Let $A, A'$ and $B$ be three dice such that $A$ and $A'$
are equal except for
one face $i$ with $|a_i-a'_i|\le\eps$. Denote the characteristic
functions of $(U_A,U_B,V-n/2)$ and $(U_{A'},U_B,V-n/2)$ as,
respectively, $\fhat$ and $\fhat'$.
Then, for
every $\alpha,\beta,\gamma\in\mathbb{R}$,
\begin{align*}
    \Big|
    \fhat(\alpha,\beta,\gamma)-\fhat'(\alpha,\beta,\gamma)
    \Big|\le2\eps/n\;.
\end{align*}
\end{claim}
\begin{proof}
Let $z_A(v):=e(\alpha g_A(v)+\beta g_B(v)+\gamma (v-n^2/2))$
and $z_{A'}(v)=e(\alpha g_{A'}(v)+\beta g_B(v)+\gamma (v-n^2/2))$
that is $\fhat(\alpha,\beta,\gamma)=\EE z_A(V)$
and $\fhat'(\alpha,\beta,\gamma)=\EE z_{A'}(V)$.
Since $|a_i-a'_i|\le\eps$, we have $z_A(v)=z_{A'}(v)$ everywhere
except on an interval of length at most $\eps$, i.e., on
a set of measure at most $\eps/n$.  
Since $|z_A(v)-z_{A'}(v)|\le 2$ always holds,
the result follows by the triangle inequality.
\end{proof}

\begin{lemma}\label{lem:interpolation}
For every $\alpha_0,\alpha,\beta_0,\beta,\gamma\in\mathbb{R}$,
we have
\[
\left|\fhat(\alpha,\beta,\gamma)-\fhat(\alpha_0,\beta_0,\gamma)\right|
\le2\pi\Big(|\alpha-\alpha_0|\left\|U_A\right\|_{\infty}+
|\beta-\beta_0|\left\|U_B\right\|_\infty\Big)\;.
\]
\end{lemma}

\begin{proof}
Since we have
\[
|e(x)-1|^2=\left(\cos(2\pi x)-1\right)^2+\sin^2(2\pi x)
=2\left(1-\cos(2\pi x)\right)\le 4\pi^2 x^2\;,
\]
where in the last step we used $\cos x\ge 1-x^2/2$ for $x\in\mathbb{R}$,
we can use it to write
\begin{align*}
\left|\fhat(\alpha,\beta,\gamma)-\fhat(\alpha_0,\beta_0,\gamma)\right|
&=\Big|\EE\Big[
e\left(\alpha_0 U_A+\beta_0 U_B+\gamma(V-n/2)\right)
\big(e\left((\alpha-\alpha_0)U_A+(\beta-\beta_0)U_B\right)-1\big)
\Big]\Big|\\
&\le\EE\Big[\Big|
e((\alpha-\alpha_0)U_A+(\beta-\beta_0)U_B)-1
\Big|\Big]\\
&\le2\pi\Big(|\alpha-\alpha_0|\left\|U_A \right\|_\infty
+|\beta-\beta_0|\left\|U_B \right\|_\infty\Big)\;.\qedhere
\end{align*}
\end{proof}

We also show that, with high probability
over the choice of balanced $A$, the random variable
$U_A$ is uniformly bounded by $\Otilde\left(\sqrt{n}\right)$.
As in other places, 
this is an adaptation of a similar Polymath argument
to the continuous setting.
\begin{lemma}\label{lem:supremum-norm}
Let $A$ be a random balanced die. Then, except with probability
$n^{-10}$, we have that
\[
\left|g_A(t)\right|<5\sqrt{n\log n}\;.
\]
for every $0\le t\le n$. In other words,
$\|U_A\|_\infty\le5\sqrt{n\log n}$.
\end{lemma}
\begin{proof}
For the purposes of this proof, let $A$ be a die
with faces iid uniform in $[0,n]$ and let $\cE$
denote the event that $A$ is balanced.
If not for the balancing, we would be
done by an application of Claim~\ref{cl:chernoff}.

To deal with the balanced case, let
$\alpha\ge 2$ be a threshold that we will choose later.
Recall that
one way of sampling $A$ from the conditional
distribution is to sample $a_1,\ldots,a_{n-1}$ iid from $[0,n]$
and reject unless 
$\frac{n^2}{2}-n\le\sum_{i=1}^{n-1} a_i\le\frac{n^2}{2}$,
in which case we set $a_n:=n^2/2-\sum_{i=1}^{n-1} a_i$. Because
of that, it might be helpful to reduce the analysis
of the ``bad event'' $|g_A(t)|\ge\alpha$ to a related event
depending only on the initial $n-1$ coordinates. More specifically, let
$W:=|\{i\in[n-1]:a_i\le t\}|$ and observe
that $W\in\{f_A(t),f_A(t)-1\}$.
By an elementary analysis of
cases $f_A(t)\ge t+\alpha$ and
$f_A(t)\le t-\alpha$, one checks
that
\begin{align}\label{eq:57}
    \mathbbm{1}\Big[|g_A(t)|\ge\alpha\Big]
    &=
    \mathbbm{1}\Big[
    \left|f_A(t)-t\right|\ge
    \alpha\Big]
    \le
    \mathbbm{1}\Bigg[
    \left|\frac{W}{n-1}-\frac{t}{n}\right|
    \ge\frac{\alpha-1}{n-1}
    \Bigg]\;.
\end{align}
Let $\mathcal{F}$ denote the event
$\left|\frac{W}{n-1}-\frac{t}{n}\right|
\ge\frac{\alpha-1}{n-1}$.
Since $W$ is a sum
of $n-1$ iid binary random variables
with expectation $t/n$ each, by
Claim~\ref{cl:chernoff} we have
\begin{align}\label{eq:58}
\Pr[\mathcal{F}]
    \le2\exp\left(-\frac{2(\alpha-1)^2}{n-1}\right)
    \le2\exp\left(-\frac{\alpha^2}{2n}\right)\;,
\end{align}
where we used that $\alpha\ge 2$ implies
$(\alpha-1)^2\ge\alpha^2/4$.
Let $D:=\{(a_1,\ldots,a_{n-1})\in[0,n]^{n-1}:
\frac{n^2}{2}-n\le\sum_{i=1}^{n-1}a_i\le\frac{n^2}{2}\}$
and $a:=(a_1,\ldots,a_{n-1})$.
Putting~\eqref{eq:57} and~\eqref{eq:58} together,
we have, for random balanced $A$,
\begin{align*}
\Pr\Big[|g_A(t)|\ge\alpha\cond\cE\Big]
&
\le\frac
{n^{-(n-1)}\int
\mathbbm{1}\left[\mathcal{F}\right]
\cdot\mathbbm{1}[a\in D]
\,\mathrm{d}a}
{n^{-(n-1)}\int
\mathbbm{1}[a\in D]\,\mathrm{d}a}
\le
\frac
{\Pr[\mathcal{F}]}
{\Pr[a\in D]}
\le 5\sqrt{n} \cdot 
2\exp\left(-\frac{\alpha^2}{2n}\right)\;,
\end{align*}
where in the end we made the estimate
$\Pr[a\in D]\ge1/5\sqrt{n}$ by using 
Lemma~\ref{lem:berry-esseen} (for that purpose, one
checks $\EE(a_i-n/2)^2=n^2/12$ and
$\EE|a_i-n/2|^3=n^3/32$).

Letting $\alpha:=\sqrt{24 n\log n}$ we get that
$\Pr[|g_A(t)|\ge\alpha\cond\cE]\le 10\sqrt{n}\cdot n^{-12}$ for any fixed
$t$. Applying this for $t=1,2,\ldots,n-1$
(note that $g_A(0)=g_A(n)=0$ almost surely), by union bound
we get that, except with probability $n^{-10}$, 
$|g_A(i)|<\alpha$ for every
$i=1,\ldots,n$.
Finally, if $|g_A(i)|<\alpha$ for each integer $i$,
then also for any $i<t<i+1$ we have
\[
t-\alpha-1< i-\alpha<f_A(i)\le f_A(t)
\le f_A(i+1)<i+1+\alpha<t+\alpha+1
\]
implying
$|g_A(t)|< \alpha+1<5\sqrt{n\log n}$ for
every $0\le t\le n$.
\end{proof}

Finally, we make a similar conditional concentration bound for the 
convolutions $U_A^{*n}$ and $G_A^{*n}$.
\begin{lemma}\label{lem:u_a-concentration}
Let $A$ be a balanced die such that
$\|U_A\|_{\infty}\le5\sqrt{n\log n}$. Then,
\begin{align}\label{eq:24}
    \Pr\left[\left|U_A^{*n}\right|\ge 50n\log n
    \;\middle\vert\; V^{*n}=\frac{n^2}{2}\right]
    \le n^{-10}\;.
\end{align}
Similarly, for such a die it holds that
\begin{align}\label{eq:25}
    \Pr\left[\left|G_A^{*n}\right|\ge 30n\log n
    \;\middle\vert\; H^{*n}=\frac{n^2}{2}\right]
    \le n^{-10}\;.
\end{align}
\end{lemma}
\begin{proof}
For~\eqref{eq:24}
the argument is similar as in the
proof of Lemma~\ref{lem:supremum-norm}.
Specifically, we consider sampling $v:=(v_1,\ldots,v_{n-1})$
iid from the uniform distribution on $[0,n]$, rejecting
if $\sum_{i=1}^{n-1} v_i>n^2/2$ or $\sum_{i=1}^{n-1} v_i<n^{2}/2-n$
and otherwise setting 
$D:=\{(v_1,\ldots,v_{n-1}):n^2/2-n\le\sum_{i=1}^{n-1}v_i\le n^2/2\}$,
$v_n:=n^2/2-\sum_{i=1}^{n-1} v_i$ and
$u_i:=g_A(v_i)$. Then, we have
\begin{align*}
    \Pr\left[\left|U_A^{*n}\right|\ge 50n\log n\;\middle\vert\;
    V^{*n}=\frac{n^2}{2}\right]
    &=\frac{n^{-(n-1)}\int \1\left[\left|\sum_{i=1}^n u_i\right|\ge 50n\log n\right]
    \cdot\1\left[v\in D\right]\,\mathrm{d}v}
    {n^{-(n-1)}\int \1\left[v\in D\right]\,\mathrm{d}v}\\
    &\le\frac{n^{-(n-1)}\int \1\left[\left|\sum_{i=1}^{n-1} u_i\right|\ge 40(n-1)\log n\right]
    \,\mathrm{d}v}
    {n^{-(n-1)}\int \1\left[v\in D\right]\,\mathrm{d}v}\\
    &=\frac{\Pr\left[\left|\sum_{i=1}^{n-1} u_i\right|/(n-1)
    \ge 40\log n
    \right]}
    {\Pr[v\in D]}\\
    &\le5\sqrt{n}\cdot
    2\exp\left(-2\frac{(40\log n)^2(n-1)}{100 n\log n}\right)\le n^{-10}\;,
\end{align*}
where to arrive in the last line we used $\Pr[v\in D]\ge1/5\sqrt{n}$
by the same argument as in the proof of Lemma~\ref{lem:supremum-norm}
and to bound the numerator we used the fact that
$u_1,\ldots,u_{n-1}$ are iid centered random variables
in $[-5\sqrt{n\log n},5\sqrt{n\log n}]$ and standard Hoeffding inequality 
$\Pr[\left|\sum_{i=1}^n x_i/n\right|\ge C]
\le2\exp(-2(C/2M)^2n)$ for iid centered 
random variables such that $|x_i|\le M$
(cf.~Claim~\ref{cl:chernoff}).

In case of~\eqref{eq:25}, the argument uses properties of
joint Gaussians $G_A^{*n}$ and $H^{*n}$. First, 
by Claim~\ref{cl:conditional-variance}, random variable
$G_A^{*n}$ conditioned on $H^{*n}=n^2/2$ is a centered
Gaussian such that
\begin{align*}
    \Var\left[G_A^{*n}\;\middle\vert\;H^{*n}=\frac{n^2}{2}\right]
    \le\Var\left[G_A^{*n}\right]
    =n\Var[U_A]\le n\|U_A\|_{\infty}^2
    \le 25n^2\log n\;.
\end{align*}
Consequently, using Claim~\ref{cl:chernoff},
\begin{align*}
    \Pr\left[\left|G_A^{*n}\right|\ge30n\log n\;\middle\vert\;
    H^{*n}=\frac{n^2}{2}\right]
    &\le2\exp\left(-\frac{30^2n^2\log^2 n}{50n^2\log n}\right)
    \le n^{-10}\;.
    \qedhere
\end{align*}
\end{proof}

\subsection{Decay of 
\texorpdfstring{$\fhat$}{f hat} 
with respect to 
\texorpdfstring{$\alpha$}{alpha}}

Clearly, $\fhat(0,0,0)=1$. Our whole argument
relies on controlling how fast $\fhat$ decays from $1$
in the neighborhood of the origin.
We now state the main lemma concerning a rate of decay of 
$\fhat$ in the most demanding case:
as the first argument $\alpha$ moves away from the origin. We will
spend the following couple of sections establishing:
\begin{lemma}\label{lem:fhat-decay}
Let $\alpha,\beta$ be such that $1/n^3\le|\alpha|\le 1/2$ and
$\delta:=10^{-14}\cdot\min\left(1,\frac{\alpha^2 n}{\log n}\right)$.
Then, except with
probability at most $40n^{-7}$ over the choice of balanced $A$ and $B$,
for every $\gamma\in\mathbb{R}$ it holds that
\begin{align*}
\left|\fhat(\alpha,\beta,\gamma)\right|\le 1-\delta/2\;.
\end{align*}
\end{lemma}

We prove Lemma~\ref{lem:fhat-decay} via two intermediate
lemmas. First, we prove its version where $A$ is a random die
\emph{without} conditioning on $A$ being balanced:
\begin{lemma}\label{lem:fhat-decay-unbalanced}
Let $\alpha,\beta$ be such that $|\alpha|\le 1/2$ and
$\delta:=10^{-14}\cdot\min\left(1,\frac{\alpha^2 n}{\log n}\right)$
and $B$ be a fixed die.
Then, except with
probability at most $2n^{-15}$ over the choice of random die $A$
with faces iid in $[0,n]$,
for every $\gamma\in\mathbb{R}$ it holds that
\begin{align}\label{eq:12}
\left|\fhat(\alpha,\beta,\gamma)\right|\le 1-\delta\;.
\end{align}
\end{lemma}

The setting we use to prove Lemma~\ref{lem:fhat-decay-unbalanced} is as follows: Consider fixed $\alpha,\beta\in\mathbb{R}$ and a die
$B$. Ultimately, we want to show that,
except with probability $2n^{-15}$ over
random choice of unbalanced die $A$, the bound
$|\fhat(\alpha,\beta,\gamma)|\le 1-\delta$ holds for every $\gamma\in\mathbb{R}$.

To that end, we choose some $m>0$ and $k\in\mathbb{N}$. We will
give specific values later, but we will always have
$k\ge 3\cdot 10^5\log n$ and $n/2-4m<4mk\le n/2$.
We partition the interval $[0,4mk)$ into
$k$ consecutive intervals $S_1,\ldots,S_k$ of length $4m$
each.
In other words, we let
\[
S_i:=[s_i,s_{i+1}):=[4m(i-1), 4mi)\;.
\]
We specify
\begin{align}\label{eq:06}
    \eps:=10^{-5}\cdot\min\left(1, |\alpha| 
    \sqrt{\frac{n}{\log n}}\right)
\end{align}
and
\begin{align*}
\theta(t)&:=\theta_{\alpha,\beta}(t):=
\alpha f_A(t)+\beta f_B(t)\;,\\
z(t)&:=z_{\alpha,\beta,\gamma}(t) :=
e\big(\alpha g_A(t)+\beta g_B(t)+\gamma(t-n/2)\big)\;,
\end{align*}
so that $\fhat(\alpha,\beta,\gamma)=\EE z(V)$ holds.
The strategy is to show that for every $i$, the bound $\EE\big[|z(V)|\;\vert\; V\in S_i\big]<1-\Omega(\eps^2)$ holds with
probability at least $1/20$ over the choice of $A$.
Furthermore, we show that this is true
even if conditioned on all the faces of $A$ that lie
in the preceding intervals $S_1,\ldots,S_{i-1}$. Then, we
use the tools from Section~\ref{sec:tools}
to conclude that with high probability, the event 
holds simultaneously for a constant
fraction of $S_i$ intervals, resulting in the overall
inequality $|\fhat|<1-\delta$.

With all that in mind, we state the second intermediate lemma
giving the bound for fixed $i\in[k]$:
\begin{lemma}
\label{lem:intermediate}
In the setting above, i.e., with fixed $\alpha,\beta$ and $B$,
as well as random $A$ with faces iid in $[0,n]$,
there exists a choice of $k$ and $m$ such
that $k\ge 3\cdot 10^5\log n$, $n/2-4m<4mk\le n/2$ and the following holds:

Let $i\in[k]$. Define the event 
$D_i:\equiv f_A(s_i)> 3n/5$. Furthermore, let
\[
A_{\le x}:=(a'_1,\ldots,a'_n)\;,
\qquad\qquad
a'_j:=\begin{cases}
a_j&\text{if $a_j\le x$,}\\
?&\text{otherwise.}
\end{cases}
\]
Then, for every $s_i\le v<s_i+m$, it always holds that
\begin{align}\label{eq:05}
\Pr_{A}\Big[
d_{\mathbb{Z}}\big(\theta(v)-\theta(v+m)-\theta(v+2m)+\theta(v+3m)\big)\ge\eps
\;\Big\vert\; A_{\le s_i}, \lnot D_i
\Big] \ge \frac{1}{20}\;.
\end{align}
\end{lemma}
Before proving Lemma~\ref{lem:intermediate}, let us see how it
implies Lemma~\ref{lem:fhat-decay-unbalanced}. In short,
we will use the triangle inequality to divide contributions
to $|\fhat|$ into intervals $S_i$,
Lemma~\ref{lem:e-to-mod1} to connect
distance to integers from $\theta(t)$ to $z(t)$ values,
and Lemma~\ref{lem:modified-azuma} to amplify
$1/20$ probability from Lemma~\ref{lem:intermediate}
into high probability.

\begin{proof}[Proof of Lemma~\ref{lem:fhat-decay-unbalanced} assuming Lemma~\ref{lem:intermediate}]
Define random variables $Z_1,\ldots,Z_k$ as
\[
Z_i:=\Pr\Big[
d_{\mathbb{Z}}\big(
\theta(V)-\theta(V+m)-\theta(V+2m)+\theta(V+3m)
\big)\ge\eps\;\Big\vert\;A,s_i\le V<s_i+m\Big]\;.
\]
Note that the value of $Z_i$ is fully determined by the die $A$. We have
$0\le Z_i\le 1$ and, by applying~\eqref{eq:05} pointwise, it almost surely holds
that
\[
\lnot D_i\implies
\EE\left[Z_i\;\middle\vert\;A_{\le s_i}\right]
\ge\frac{1}{20}\;.
\]
By the first moment method applied to $\EE[Z_i\cond A_{\le s_i}]$ pointwise
for all conditionings on $A_{\le s_i}$, we get,
again almost surely, that
\[
\lnot D_i\implies
\Pr\left[Z_i\ge\frac{1}{40}\;\middle\vert\;
A_{\le s_i}\right]\ge\frac{1}{40}\;.
\]
Now let $X_i:=\mathbbm{1}[Z_i\ge 1/40]$ 
be the indicator of the event $Z_i\ge 1/40$.
It should now be clear that $X_i$ is measurable
with respect to the $\sigma$-algebra $\mathcal{F}_i$
generated by $A_{\le s_{i+1}}$. Furthermore, 
if the event $D_i$ does not happen, then
$\EE[X_i\;\vert\;\mathcal{F}_{i-1}]\ge 1/40$.
Therefore, defining $Y_0:=0$ and 
$Y_{i+1}:=Y_i+X_{i+1}-1/40$, we get a sequence
of random variables $Y_0,\ldots,Y_k$ adapted
to $\mathcal{F}_0,\ldots,\mathcal{F}_k$
with $|Y_{i+1}-Y_i|\le 1$ and
$\EE[Y_{i+1}-Y_i\;\vert\;\mathcal{F}_i]\ge 0$ whenever
the event $D_{i+1}$ does not happen. Therefore, we can apply
Lemma~\ref{lem:modified-azuma} to obtain
\begin{align*}
\Pr\left[\sum_{i=1}^kX_i\le\frac{k}{80}\right]
&=\Pr\left[Y_k\le-\frac{k}{80}\right]
\le\exp\left(-\frac{k}{2\cdot 10^4}\right)+\Pr\left[\bigcup_{i=1}^k D_i\right]\\
&\le n^{-15}+\Pr\big[f_A(n/2)>3n/5\big]
\le 2n^{-15}\;,
\end{align*}
where in the end we used the bound $k\ge 3\cdot 10^5\log n$ and
Claim~\ref{cl:chernoff} for the deviation of $f_A(n/2)$.

Therefore, except with probability $2n^{-15}$, we get that
for at least $k/80$ indices $i\in[k]$, for at least $1/40$
fraction of $v\in[s_i,s_i+m]$, we have
\begin{align}\label{eq:10}
d_{\mathbb{Z}}
\big(\theta(v)-\theta(v+m)-\theta(v+2m)+\theta(v+3m)\big)
\ge\eps\;.
\end{align}
Recall that $\theta(t)=\alpha f_A(t)+\beta f_B(t)$ and
let $\theta'(t):=\alpha g_A(t)+\beta g_B(t)+\gamma(t-n/2)$
and again recall that $z(t)=e(\theta'(t))$.
Note that
\begin{align}\label{eq:11}
\theta(v)-\theta(v+m)-\theta(v+2m)+\theta(v+3m)
=
\theta'(v)-\theta'(v+m)-\theta'(v+2m)+\theta'(v+3m)\;,
\end{align}
and from~\eqref{eq:10} and~\eqref{eq:11}, applying
Lemma~\ref{lem:e-to-mod1}, we get that for all $v$
such that~\eqref{eq:10} holds we have
\[
\big|
z(v)+z(v+m)+z(v+2m)+z(v+3m)
\big|\le 4-\eps^2\;.
\]
Consequently, we partition the interval $[0,n]$ into $S_1,\ldots,S_k$
and the leftover interval of length $n-4km$.
Using the trivial bound of one for the modulus
of the characteristic function in the leftover interval, we can write
\begin{align*}
\left|\fhat(\alpha,\beta,\gamma)\right|
&=
\frac{1}{n}\left|\int_{0}^n z(t)\,\mathrm{d}t\right|
\le 1-\frac{4km}{n}+\frac{1}{n}\sum_{i=1}^k
\left|\int_{S_i}z(t)\,\mathrm{d}t\right|\\
&\le 1-\frac{4km}{n}+\frac{1}{n}\sum_{i=1}^k
\int_{s_i}^{s_i+m} \big|z(t)+z(t+m)+z(t+2m)+z(t+3m)\big|
\,\mathrm{d}t\\
&\le 1-\frac{4km}{n}+\frac{4km}{n}-\frac{1}{n}\cdot\frac{k}{80}\cdot\frac{m}{40}\cdot\eps^2
\le1-\frac{\eps^2}{10^4}.
\end{align*}
Recalling~\eqref{eq:12} and~\eqref{eq:06} for the definitions
of $\delta$ and $\eps$, we notice that~\eqref{eq:12} is
now established and the proof is concluded.
\end{proof}

We still need to show that Lemma~\ref{lem:fhat-decay-unbalanced} implies
Lemma~\ref{lem:fhat-decay}. The
Polymath draft in the discrete setting can just use the union bound, but we are conditioning on an event of measure zero.
Furthermore, a basic approach like in Lemmas~\ref{lem:supremum-norm}
and~\ref{lem:u_a-concentration} does not seem to work, since,
for larger values of $|\alpha|$,
in principle changing just one face in a die can significantly
change the value of $\fhat$.
Instead, we use a little more
precise coupling argument:
\begin{proof}[Proof of Lemma~\ref{lem:fhat-decay} assuming
Lemma~\ref{lem:fhat-decay-unbalanced}]
Fix $\alpha$, $\beta$ and a die $B$ and let $s:=n^{-13/2}$.
Consider random choice of an unbalanced die $A$ with faces
in $[0,n]$. We define two events: 
let $\cE$ denote that there exists $\gamma$ such that
$\big|\fhat(\alpha,\beta,\gamma)\big|>1-\delta$, 
and $\cS$ that $|\sum_{i=1}^n a_i-n^2/2|\le s$. 

By
Claim~\ref{cl:h-lower-bound}.2, we have
$\Pr[\cS]\ge n^{-8}/4$. Applying Lemma~\ref{lem:fhat-decay-unbalanced}, we get that
\begin{align*}
\Pr[\cE\cond\cS]
&\le\frac{\Pr[\cE]}{\Pr[\cS]}
\le\frac{2n^{-15}}{n^{-8}/4}=8n^{-7}\;.
\end{align*}
Fix $s'$ with $|s'|\le s$. Recall that one way to choose
$A$ conditioned on face-sum equal to $n^2/2$ is by
rejection sampling. That is, if the event
$\mathcal{T}_1$ denoting $n^2/2-n\le \sum_{i=1}^{n-1} a_i\le n^2/2$ happens,
then set $a_n$ to make the die balanced, otherwise
reject and repeat. Similarly, to sample $A$ conditioned
on face-sum equal to $n^2/2+s'$, one can utilize the rejection sampling using the
event $\mathcal{T}_2$ standing for $n^2/2+s'-n\le\sum_{i=1}^{n-1} a_i\le n^2/2+s'$.

Note that the symmetric difference 
$\mathcal{T}_1\triangle\mathcal{T}_2$
consists of $\sum_{i=1}^{n-1}a_i$ being in one of two intervals
of width $|s'|$ and by Claim~\ref{cl:h-lower-bound}.3
it has probability at most 
$\Pr[\mathcal{T}_1\triangle\mathcal{T}_2]\le 
2s/(n-1)\le 3s/n=3n^{-15/2}$. On the other hand,
again by Claim~\ref{cl:h-lower-bound}.2, we have
\begin{align*}
\Pr[\mathcal{T}_1]
&=\Pr\left[\left|\sum_{i=1}^{n-1} a_i-(n-1)n/2\right|\le \frac{n}{2}\right]
\ge \frac{n}{8(n-1)\sqrt{n-1}}\ge \frac{1}{8\sqrt{n}}\;.   
\end{align*}

Consider the procedure of rejection sampling of two dice using
joint randomness.
One die is balanced, using the event $\mathcal{T}_1$
and the other die has face-sum $n^2/2+s'$ using the
event $\mathcal{T}_2$. Consider the round when
one of the dice is successfully output, i.e,
the first round when
$\mathcal{T}_1\cup\mathcal{T}_2$ occurs. The probability
that the other die is not output in that round
is at most
\begin{align*}
    \frac{\Pr[\mathcal{T}_1\triangle\mathcal{T}_2]}{
    \Pr[\mathcal{T}_1\cup\mathcal{T}_2]}
    \le \frac{3n^{-15/2}}{n^{-1/2}/8}=24n^{-7}\;.
\end{align*}
At the same time, if both dice are output in the same
round of this joint rejection sampling, they will agree
everywhere except for the last face.
Consequently, there exists a coupling between a
random balanced $A$ on the one hand,
and random $A$ conditioned on face-sum equal to $n^2/2+s'$
on the other hand,
such that their first $n-1$ faces are equal except with probability
$24n^{-7}$. Averaging over $s'$, there also exists such
a coupling between $A$ conditioned on being balanced and
$A$ conditioned on $\cS$. 

Let $A'$ be a die with
face-sum $n^2/2+s'$ such that $|s'|\le s$ and let $A$ be $A'$
with the last face modified so that $A$ is balanced.
By Claim~\ref{cl:interpolation-a}, if the event $\cE$ does
not hold for $A'$, then for $A$ it holds that for all $\gamma\in\mathbb{R}$ we
have
\begin{align}\label{eq:36}
    \big|\fhat(\alpha,\beta,\gamma)\big|\le 1-\delta+2s/n
    =1-\delta+2n^{-15/2}
    \le 1-\delta/2\;,
\end{align}
where we used that $|\alpha|\ge 1/n^3$ implies 
$\delta\ge 1/n^6$. By union bound it follows that~\eqref{eq:36}
holds for a random balanced $A$ except with probability
$24n^{-7}+8n^{-7}<40n^{-7}$.
\end{proof}

\subsection{Proof of Lemma~\ref{lem:intermediate}}

We delay choosing a specific $m$ and for now
assume that the unique integer $k$
satisfying $n/2-4m<4mk\le n/2$ has value at least
$3\cdot 10^5\log n$.
Following the statement of the lemma, 
we fix $v$ in the interval
$s_i\le v<s_i+m$ and a die $B$
and condition on $A_{\le s_i}$, i.e., on all faces of $A$ that
have values not exceeding $s_i$. Note that this determines
the values of $f_A(x)$ for all $x\le s_i$ and that we assume that
the conditioning satisfies $f_A(s_i)\le 3n/5$.

First, we observe that, after all the foregoing conditioning,
letting $t:=f_A(v+2m)$, the difference
$t-f_A(s_i)$ is a sum of $n-f_A(s_i)$ iid binary random
variables, each with expectation $(v+2m-s_i)/(n-s_i)$.
Therefore, we have
\[
\EE[t\;\vert\; A_{\le s_i}]=f_A(s_i)+\big(n-f_A(s_i)\big)\frac{v+2m-s_i}{n-s_i}
\le 3n/5 + 6m\;,
\]
where we used that $n-s_i\ge n/2$. 
Since the conditional expectation of $t$ can be written as a sum of a constant
which is at most $3n/5$
and a nonnegative random variable with expectation at most $6m\le n/100\log n$ (due to bounds on $k$ and $m$),
by Markov's inequality, we have that
$t\le 2n/3$ with probability at least $0.99$.

Consequently, let us further condition on 
$A_{\le v+2m}$ such that $t\le 2n/3$
and consider the distribution of
$\theta(v)-\theta(v+m)-\theta(v+2m)+\theta(v+3m)$.
After conditioning on $A_{\le v+2m}$,
the only part of this expression that remains random
is $f_A(v+3m)$ featuring in $\theta(v+3m)=\alpha f_A(v+3m)+\beta f_B(v+3m)$.
What is more, this can be written as
\[
    f_A(v+3m) = f_A(v+2m)+\Xtilde\;,
\]
where $\Xtilde=\sum_{i=1}^{n-t} X_i$ is a sum
of iid binary random variables with
$m/n\le\Pr[X_i=1]=\frac{m}{n-(v+2m)}\le 2m/n$.

All in all, omitting the conditioning from the notation,
we can write the probability from~\eqref{eq:05}
as
\begin{align}\label{eq:08}
\Pr\left[
d_\mathbb{Z/|\alpha|}
\big(\Xtilde+C\big)
\ge \frac{\eps}{|\alpha|}
\right]
\end{align}
for some constant $C\in\mathbb{R}$. In order
to lower bound~\eqref{eq:08}, we proceed with two
cases, using different limit theorems.

\paragraph{Case 1: $|\alpha|\le 10^{-4}$ and Berry-Esseen theorem}
In this case we choose
\begin{align}\label{eq:59}
m:=\max\left(\frac{4\eps^2}{\alpha^2},10^5\right)\;.   
\end{align}
Considering two cases in~\eqref{eq:06}
$|\alpha|<\sqrt{\log n/n}$ and
$|\alpha|\ge\sqrt{\log n/n}$, we see that
$m\le 4\cdot 10^{-10}n/\log n$.
Accordingly, we verify that we have
$k=\lfloor n/8m\rfloor\ge 3\cdot 10^5\log n$.

Recall that random variables $X_i$ are 
binary iid
with $m/n\le\EE X_i\le 2m/n$. Consequently, also
$0.99\cdot m/n\le\Var X_i=:\sigma^2\le 2m/n$ and
$\EE|X_i-\EE X_i|^3\le 2m/n$.
Applying Lemma~\ref{lem:berry-esseen}, we get
that for any $a,b\in\mathbb{R}$,
\[
\Pr\left[
a\le\Xtilde-\EE\Xtilde\le b
\right]
\ge\Pr\left[
\frac{a}{\sigma\sqrt{n-t}}\le N
\le\frac{b}{\sigma\sqrt{n-t}}
\right]
- \sqrt{13/m}\;,
\]
where $N$ is a standard normal. Choosing
$a:=\sqrt{m}/2$, 
$b:=2\sqrt{m}$, 
$a':=-2\sqrt{m}$,
$b':=-\sqrt{m}/2$, we get
\begin{align*}
    \Pr\left[a\le\Xtilde\le b\right],\
    \Pr\left[a'\le\Xtilde\le b'\right]
    &\ge\Pr\left[\frac{\sqrt{m}}{2\sigma\sqrt{n-t}}
    \le N\le \frac{2\sqrt{m}}{\sigma\sqrt{n-t}}\right]
    -\sqrt{13/m}\\
    &\ge \Pr\left[1.01\cdot \sqrt{3}/2\le N\le \sqrt{2}\right]-\sqrt{13/m}
    \ge 1/10\;,
\end{align*}
where in the last step we used 
$\Pr[1.01\cdot\sqrt{3}/2\le N\le \sqrt{2}]\ge 0.112$ and
$m\ge 10^5$. Hence, we have found two disjoint intervals
such that $\Xtilde$ lies within
each of them with probability at least $1/10$.
To conclude the analysis of this case it is enough
that we prove the following:
\begin{claim}\label{cl:one-event}
At least one of the events $a'\le\Xtilde\le b'$
and $a\le\Xtilde\le b$ implies 
$d_{\mathbb{Z}/|\alpha|}(\Xtilde+C)\ge\sqrt{m}/2
\ge\eps/|\alpha|$.
\end{claim}
\begin{proof}[Proof (cf.~Figure~\ref{fig:one-event})]
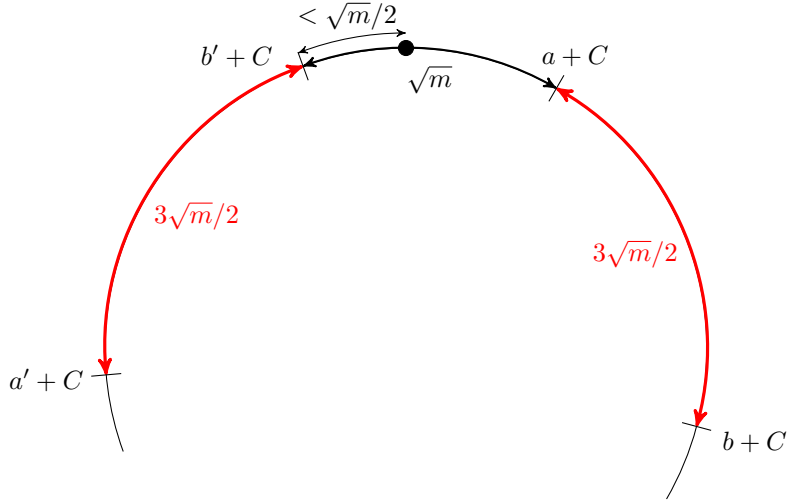
\begin{figure}[!ht]\centering\begin{tikzpicture}
\draw (4, 0) arc (0:200:4);
\draw (4, 0) arc (0:-30:4);
\filldraw (0, 4) [black] circle [radius=0.1];

\draw [thick, <->] ({4*cos(60)}, {4*sin(60)}) arc (60:110:4);
\draw ({3.6*cos(85)}, {3.6*sin(85)}) node {$\sqrt{m}$};

\draw [very thick, red, <->] ({4*cos(110)}, {4*sin(110)}) arc (110:185:4);
\draw ({3.8*cos(185)}, {3.8*sin(185)}) --
({4.2*cos(185)}, {4.2*sin(185)});
\draw ({3.8*cos(110)}, {3.8*sin(110)}) --
({4.2*cos(110)}, {4.2*sin(110)});
\draw ({3.3*cos(147.5)}, {3.3*sin(147.5)}) node[text=red] {$3\sqrt{m}/2$};
\draw ({4.8*cos(185)}, {4.8*sin(185)}) node {$a'+C$};
\draw ({4.5*cos(120)}, {4.5*sin(120)}) node {$b'+C$};

\draw [very thick, red, <->] ({4*cos(-15)}, {4*sin(-15)}) arc (-15:60:4);
\draw ({3.8*cos(60)}, {3.8*sin(60)}) --
({4.2*cos(60)}, {4.2*sin(60)});
\draw ({3.8*cos(-15)}, {3.8*sin(-15)}) --
({4.2*cos(-15)}, {4.2*sin(-15)});
\draw ({3.3*cos(22.5)}, {3.3*sin(22.5)}) node[text=red] {$3\sqrt{m}/2$};
\draw ({4.8*cos(-15)}, {4.8*sin(-15)}) node {$b+C$};
\draw ({4.5*cos(60)}, {4.5*sin(60)}) node {$a+C$};

\draw [<->] ({4.2*cos(90)}, {4.2*sin(90)}) arc (90:111:4);
\draw ({4.5*cos(100)}, {4.5*sin(100)}) node {$<\sqrt{m}/2$};
\end{tikzpicture}    
\caption{Illustration of the proof
of Claim~\ref{cl:one-event}. The circle represents
the set $\faktor{\mathbb{R}}{\mathbb{Z}/|\alpha|}$ and its circumference
is at least $6\sqrt{m}$. The black dot represents
zero.}
\label{fig:one-event}
\end{figure}
The fact that $\sqrt{m}/2\ge\eps/|\alpha|$
follows directly from~\eqref{eq:59}
and therefore it remains to show that the distance
to $\mathbb{Z}/|\alpha|$ exceeds $\sqrt{m}/2$.
Assume that the first event does not imply
the conclusion, that is that
there exists $a'\le x'\le b'$ such that
$d_{\mathbb{Z}/|\alpha|}(x'+C)<\sqrt{m}/2$.
In other words, there exists $k\in\mathbb{Z}$ such that
\[
\left|x'+C-\frac{k}{|\alpha|}\right|<\sqrt{m}/2\;.
\]
We will be done if we show that for every
$a\le x\le b$ we have
$k/|\alpha|+\sqrt{m}/2\le x+C\le (k+1)/|\alpha|-\sqrt{m}/2$.
Equivalently, we want to show $k/|\alpha|+\sqrt{m}/2\le a+C$
and $b+C\le (k+1)/|\alpha|-\sqrt{m}/2$. But indeed,
\begin{align*}
a+C-\frac{k}{|\alpha|}&
=a-x'+\left(x'+C-\frac{k}{|\alpha|}\right)
>a-b'-\sqrt{m}/2=\sqrt{m}/2\;,\\
\frac{k+1}{|\alpha|}-(b+C)
&=\frac{1}{|\alpha|}-(b-x')-
\left(x'+C-\frac{k}{|\alpha|}\right)
>\frac{1}{|\alpha|}-4\sqrt{m}-\sqrt{m}/2
\ge\sqrt{m}\;,
\end{align*}
where in the last calculation one checks from~\eqref{eq:59}
and~\eqref{eq:06} that
$1/|\alpha|\ge 6\sqrt{m}$.
\end{proof}

\paragraph{Case 2: $10^{-4}<|\alpha|\le 1/2$ and
Poisson limit theorem}
In the case of larger $|\alpha|$ we will use
$m$ of the order of small constant. Here
the Berry-Esseen theorem might be less helpful, since
$\Xtilde$ converges in distribution to a Poisson random
variable. Note that also conceptually this case
seems to be slightly different. For example, to establish
$|\fhat(1/2,0,1/2)|=|\EE e(f_A(t)/2)|<1$ we need to exclude the possibility that function $f_A$ always takes even (or odd) values.
In order to do that, we need to be able to home in
on specific values of $\Pr[\Xtilde=k]$.

More concretely, in this case we fix
$m:=1$. Again, it is easily checked that
$k\ge 3\cdot 10^5\log n$. 
Let $\lambda:=\EE\Xtilde=(n-t)\Pr[X_i=1]$
and note that we have
$1/3\le\lambda\le 2$.
By Lemma~\ref{lem:poisson},
\begin{align*}
    \Pr[\Xtilde=0]
    &\ge\Pr[Z_\lambda=0]-4/(n-t)\ge
    \exp(-\lambda)-12/n
    \ge 1/10\;,\\
    \Pr[\Xtilde=1]
    &\ge\Pr[Z_\lambda=1]-4/(n-t)
    \ge\lambda\exp(-\lambda)-12/n
    \ge 1/10\;.
\end{align*}
As in case 1, to finish the proof it is enough
to argue that if $\big|C-k/|\alpha|\big|< \eps/|\alpha|$,
then $k/|\alpha|+\eps/|\alpha|\le 1+C\le
(k+1)/|\alpha|-\eps/|\alpha|$ and therefore
at least one of $\Xtilde=0$ and $\Xtilde=1$
must imply the event from~\eqref{eq:08}.
But this is again a straightforward verification
using bounds on $|\alpha|$ and $\eps\le 10^{-5}$:
\begin{align*}
    1+C-\frac{k}{|\alpha|}
    &>1-\frac{\eps}{|\alpha|}
    \ge \frac{\eps}{|\alpha|}\;,\\
    \frac{k+1}{|\alpha|}-(1+C)
    &=\frac{1}{|\alpha|}-1+
    \left(\frac{k}{|\alpha|}-C\right)
    >\frac{1}{|\alpha|}-1-\frac{\eps}{|\alpha|}
    \ge 1-\frac{\eps}{|\alpha|}\ge\frac{\eps}{|\alpha|}\;,
\end{align*}
concluding the analysis of case 2.
\medskip

To sum up, in both cases we obtained that, conditioned
on $t\le 2n/3$, which happens with probability at least
0.99, with probability at least $1/10$,
\begin{align}\label{eq:09}
d_\mathbb{Z}\big(
\theta(v)-\theta(v+m)-\theta(v+2m)+\theta(v+3m)
\big)\ge \eps\;.
\end{align}
Therefore,~\eqref{eq:09} certainly holds with overall
probability at least $1/20$, as claimed in~\eqref{eq:05}.
\hfill\qedsymbol

\subsection{Bounding 
\texorpdfstring{$\fhat$}{f hat} 
with respect to 
\texorpdfstring{$\gamma$}{gamma}}

For the central limit theorem, we will need that $|\fhat|$
is small everywhere outside a small box around the
origin. We will use Lemma~\ref{lem:fhat-decay} to deal with the
case of larger $|\alpha|$ and $|\beta|$. In this section we
prove some similar, but simpler results to handle
small $|\alpha|$ and $|\beta|$ and larger $|\gamma|$.

\begin{lemma}\label{lem:large-gamma}
For any choice of dice $A$ and $B$, we always have $|\fhat(\alpha,\beta,\gamma)|
\le\frac{1}{|\gamma-\alpha-\beta|}$. 
\end{lemma}
\begin{proof}
We calculate, using the fact that the interval
$[0,n]$ can be partitioned into $2n+1$ intervals
$[c_i,c_{i+1}]$ such that on each of the intervals the functions
$f_A$ and $f_B$ are constant with $f_A(t)=d_i$ and
$f_B(t)=e_i$:
\begin{align*}
\fhat(\alpha,\beta,\gamma)
&=\EE e\big(\alpha g_A(V)+\beta g_B(V)+\gamma (V-n/2)\big)
=\frac{1}{n}\int_0^n e\big(
\alpha g_A(t)+\beta g_B(t)+\gamma (t-n/2)
\big)\,\mathrm{d}t\\
&=\frac{1}{n}\int_0^n e\big(
\alpha f_A(t)+\beta f_B(t)+(\gamma-\alpha-\beta)t +\gamma n/2)
\big)\,\mathrm{d}t\\
&=\frac{1}{n}\sum_{i=0}^{2n}e\big(\alpha d_i+\beta e_i+\gamma n/2\big)
\int_{c_i}^{c_{i+1}} e\big((\gamma-\alpha-\beta)t\big)\,\mathrm{d}t\\
&=\frac{1}{n}\sum_{i=0}^{2n} 
\frac{F_i\cdot\Big[e\big((\gamma-\alpha-\beta)c_{i+1}\big)
-e\big((\gamma-\alpha-\beta)c_i\big)\Big]}
{2\pi(\gamma-\alpha-\beta)}\;,
\end{align*}
where in the last step we denote the constant factor
on the $i$-th interval by
$F_i$ with $|F_i|=1$. Consequently, and using the triangle
inequality and
$|e(a)-e(b)|\le 2$,
\begin{align*}
\left|\fhat(\alpha,\beta,\gamma)\right|
&\le \frac{2(2n+1)}{2\pi n|\gamma-\alpha-\beta|}
\le\frac{1}{|\gamma-\alpha-\beta|}\;.\qedhere
\end{align*}
\end{proof}

Lemma~\ref{lem:large-gamma} handles the case where
$|\gamma|$ is quite large. The remaining case is
very small $|\alpha|,|\beta|$ and somewhat larger 
$|\gamma|$:

\begin{lemma}\label{lem:gamma-box-decay}
Let $|\alpha|,|\beta|\le\frac{10^{10}\log n}{n}$. Then,
provided that dice $A$ and $B$ satisfy
$\|U_A\|_\infty, \|U_B\|_\infty\le 5\sqrt{n\log n}$,
for every $\gamma$
with $|\gamma|\ge 6\log^2 n/n^{3/2}$ it holds that
\[
    \left|\fhat(\alpha,\beta,\gamma)\right|
    \le 1-\frac{\log^4 n}{2n}\;.
\]
\end{lemma}

\begin{proof}
Recall that we are bounding 
$\fhat(\alpha,\beta,\gamma)=
\EE e(\alpha g_A(V)+\beta g_B(V)+\gamma(V-n/2))$.
By our assumption we have 
$|g_A(t)|,|g_B(t)|\le 5\sqrt{n\log n}$ for every $t$.
Therefore, by the assumption on $\alpha$ and $\beta$,
we can write
\[
\fhat(\alpha,\beta,\gamma)
=e(-\gamma n/2)\EE e(\gamma V + R(V))\;,
\]
where it always holds that
$|R(V)|\le (\log^2 n/\sqrt{n})$.
Following~\cite{Pol17}, we adopt a simpler version of arguments
from the proof of
Lemma~\ref{lem:fhat-decay-unbalanced}. 
Specifically, we 
will (later) 
choose some $0<m\le n/2$ and, letting
$k:=\lfloor n/2m\rfloor$,
partition the interval $[0,n]$
into $2k$ intervals $S_1,T_1\ldots,S_k,T_k$
of length $m$ each and a possible ``leftover'' interval
$S_0$ of length at most $n/2$.

We focus on values of $z(t)$ and $z(t+m)$ inside the intervals.
More specifically, we will choose $m$ so
that it always holds
\begin{align}\label{eq:14}
\frac{3\log^2 n}{\sqrt{n}}\le |\gamma|m\le \frac{1}{4}\;,
\end{align}
consequently giving
\begin{align*}
    \frac{\log^2 n}{\sqrt{n}}
    \le|\gamma|m-\frac{2\log^2 n}{\sqrt{n}}
    \le\big|\gamma m+R(t+m)-R(t)\big|
    \le|\gamma|m+\frac{2\log^2 n}{\sqrt{n}}
    \le\frac{1}{2}\;,
\end{align*}
and
\begin{align*}
    d_{\mathbb{Z}}\big(\gamma m+R(t+m)-R(t)\big)
    \ge\frac{\log^2 n}{\sqrt{n}}\;.
\end{align*}
Applying triangle inequality and Lemma~\ref{lem:e-to-mod1}, 
we then get
\begin{align*}
\left|\fhat(\alpha,\beta,\gamma)\right|
&\le\frac{1}{n}\left(
\int_{S_0}\,\mathrm{d}t+
\sum_{i=1}^k\int_{S_i}
\big|e(\gamma t+R(t))+e(\gamma(t+m)+R(t+m))\big|\,\mathrm{d}t
\right)\\
&\le\frac{1}{n}\left(
\int_{S_0}\,\mathrm{d}t+
\sum_{i=1}^k\int_{S_i}
2-2\frac{\log^4 n}{n}
\,\mathrm{d}t
\right)
=1-\frac{2km}{n}\cdot\frac{\log^4 n}{n}\\
&\le 1-\frac{\log^4 n}{2n}\;,
\end{align*}
as claimed. It remains to show that a value
of $0<m\le n/2$ satisfying~\eqref{eq:14} can be chosen.
This is done by considering two cases.

\paragraph{Case 1: $6\log^2 n/n^{3/2}\le |\gamma|\le 1/2n$} 
In this case we just pick $m:=n/2$. Indeed, we clearly have
\[
\frac{3\log^2 n}{\sqrt{n}}\le
|\gamma|m=\frac{|\gamma|n}{2}\le\frac{1}{4}\;.
\]
\paragraph{Case 2: $|\gamma|>1/2n$}
Here we choose $m:=1/4|\gamma|$ and easily check
$\frac{3\log^2 n}{\sqrt{n}}\le m|\gamma|=1/4$.
\end{proof}

\subsection{Proof of Lemma~\ref{lem:fhat-bounded}}

By Lemma~\ref{lem:supremum-norm},
$\|U_A\|_{\infty},\|U_B\|_{\infty}<5\sqrt{n\log n}$ holds
except with probability $2n^{-10}$. Assume that it is so.
If both $|\alpha|$ and $|\beta|$ are at most
$10^{10}\log n/n$ and $|\gamma|$ exceeds
$6\log^2 n/n^{3/2}$,
then we get $|\fhat(\alpha,\beta,\gamma)|\le1-10\log n/n$
directly by
Lemma~\ref{lem:gamma-box-decay}.

As for the case where either $|\alpha|$ or
$|\beta|$ exceed $10^{10}\log n/n$,
we cover the square 
$[-1/2,1/2]^2$ with a grid
of points $(\alpha_i,\beta_i)$ such that
\begin{align*}
\alpha_i,\beta_i\in\left\{
-\frac{1}{2},-\frac{1}{2}+\frac{1}{n^2},
-\frac{1}{2}+\frac{2}{n^2},\ldots
-\frac{1}{2n^2},\frac{1}{2n^2},\ldots,\frac{1}{2}
\right\}\;.
\end{align*}
This is a discrete grid of $(n^2+1)^2\le 2n^4$ points. We
will proceed, for each applicable $(\alpha,\beta,\gamma)$,
to find a grid point $(\alpha_i,\beta_i)$ such that it
holds both that
\begin{align}\label{eq:37}
|\alpha-\alpha_i|+|\beta-\beta_i|\le \frac{2}{n^2}
\qquad\text{ and }\qquad
\left|\fhat(\alpha_i,\beta_i,\gamma)\right|\le 1-12\log n/n\;.
\end{align}
Then we will be finished, since, applying
Lemma~\ref{lem:interpolation}, we will have
\begin{align*}
\left|\fhat(\alpha,\beta,\gamma)\right|
&\le\left|\fhat(\alpha_i,\beta_i,\gamma)\right|
+\left|\fhat(\alpha_i,\beta_i,\gamma)-
\fhat(\alpha,\beta,\gamma)\right|\\
&\le1-\frac{12\log n}{n}+\frac{20\pi\sqrt{n\log n}}{n^2}
<1-\frac{10\log n}{n}\;.
\end{align*}
To achieve that we apply Lemma~\ref{lem:fhat-decay}. Given
$\alpha$ and $\beta$,
we take $(\alpha_i,\beta_i)$ to be the closest possible 
to
$(\alpha,\beta)$ with $|\alpha_i|\ge|\alpha|$
and $|\beta_i|\ge|\beta|$. Clearly, the first
condition in~\eqref{eq:37} holds and by
Lemma~\ref{lem:fhat-decay} we get
that $|\fhat(\alpha_i,\beta_i,\gamma)|\le 1-12\log n/n$ for
every $\gamma\in\mathbb{R}$ except with probability
$40n^{-7}$. 

Taking union bound
over the whole grid and the events from Lemma~\ref{lem:supremum-norm},
we see that $|\fhat(\alpha,\beta,\gamma)|\le 1-10\log n/n$
for all applicable points, except with probability
at most $2n^{-10}+2n^4\cdot 40n^{-7}<n^{-2}$.
\hfill\qedsymbol

\subsection{Proof of Lemma~\ref{lem:characteristic-close}}

Assume that the thesis of Lemma~\ref{lem:fhat-bounded}
holds, in particular that 
$\|U_A\|_\infty,\|U_B\|_\infty\le 5\sqrt{n\log n}$.
By Claim~\ref{cl:fhat-moments}, we have
$\fhat(\alpha,\beta,\gamma)=1-Q+R$, where 
$Q$ is given in~\eqref{eq:38} and $R$ is such that
\begin{align}
    |R|
    &\le2\cdot 10^5(n\log n)^{3/2}
    \big(|\alpha|^3+|\beta|^3\big)
    +150n^3|\gamma|^3\;.\label{eq:17}
\end{align}
Our task is to estimate the integral
\begin{align*}
    \int_{\bbR^3}\left|\ghat(\alpha,\beta,\gamma)
    -\hhat(\alpha,\beta,\gamma)
    \right|\,\mathrm{d}\alpha\beta\gamma\;.
\end{align*}
To that end, we are going to divide $\bbR^3$ into various 
parts. More precisely, let 
$B:=\big\{(\alpha,\beta,\gamma):\allowbreak
|\alpha|,|\beta|\le 10^{10}\log n/n,
|\gamma|\le 6\log^2 n/n^{3/2}
\big\}$. We make an estimate
\begin{align*}
    \int_{\bbR^3}\left|\ghat(\alpha,\beta,\gamma)
    -\hhat(\alpha,\beta,\gamma)
    \right|\,\mathrm{d}\alpha\beta\gamma
    &\le\int_{B}\left|\ghat(\alpha,\beta,\gamma)
    -\hhat(\alpha,\beta,\gamma)
    \right|\,\mathrm{d}\alpha\beta\gamma\\
    &\qquad\qquad+\int_{\Bcompl}
    \left|\ghat(\alpha,\beta,\gamma)\right|\,\mathrm{d}\alpha\beta\gamma
    +\int_{\Bcompl}
    \left|\hhat(\alpha,\beta,\gamma)\right|\,\mathrm{d}\alpha\beta\gamma
\end{align*}
and proceed to bounding each of the three resulting terms.

\paragraph{Term 1: $\int_B|\ghat-\hhat|$}
For this term, note that for every $(\alpha,\beta,\gamma)\in B$
we have
\begin{align}
    |Q|&\le 6\pi^2\big(\alpha^2\Var_A+\beta^2\Var_B+\gamma^2\Var V\big)
    \le6\pi^2\left(2\cdot\frac{10^{20}\log^2n}{n^2}\cdot
    25n\log n+\frac{36\log^4 n}{n^3}\cdot \frac{n^2}{4}
    \right)\nonumber\\
    &\le \frac{\log^5 n}{n}\;,\label{eq:18}
\end{align}
and accordingly $Q^2\le 1/100n$. Similarly, continuing from~\eqref{eq:17}, we have
\begin{align}
    |R|&\le
    2\cdot 10^5n^{3/2}\log^{3/2}n\cdot 2 \cdot
    \frac{10^{30}\log^3 n}{n^3}
    +150n^3\cdot\frac{6^3\log^6n}{n^{9/2}}\le\frac{\log^7 n}{n^{3/2}}\;,\label{eq:19}
\end{align}
in particular again $|R|\le 1/100n$. 
Therefore, Lemma~\ref{lem:exp-nq-approx} applies inside
$B$ and we can use it, together with definitions
of $\ghat$ and $\hhat$ and Claim~\ref{cl:ghat-formula}
to estimate
\begin{align*}
    \int_B\left|\ghat-\hhat\right|
    &=\int_B\left|\exp(-nQ)-(1-Q+R)^n\right|
    \le\int_B\left|1-\frac{(1-Q+R)^n}{\exp(-nQ)}\right|\\
    &\le\left(\frac{2\cdot 10^{10}\log n}{n}\right)^2
    \cdot \frac{12\log^2 n}{n^{3/2}}
    \cdot 40n\cdot\left(
    \frac{\log^{10} n}{n^2}+\frac{\log^7 n}{n^{3/2}}
    \right)
    \le \frac{\log^{15} n}{n^4}\;.
\end{align*}

\paragraph{Term 2: $\int_{\Bcompl} |\ghat|$}
By Lemma~\ref{lem:fhat-bounded}, we have
$|\hhat(\alpha,\beta,\gamma)|\le n^{-10}$ for all
the points on the boundary of $B$. Furthermore,
applying Lemma~\ref{lem:exp-nq-approx} as well
as~\eqref{eq:18} and~\eqref{eq:19}, we also get
\begin{align}\label{eq:20}
    |\ghat|\le|\hhat|+|\hhat-\ghat|
    =|\hhat|\left(1+\left|1-\frac{\ghat}{\hhat}\right|\right)
    \le|\hhat|\left(1+40n(Q^2+|R|)\right)
    \le 2n^{-10}
\end{align}
everywhere on the boundary of $B$. We now proceed to
estimating $\int_{\Bcompl} |\ghat|$ by a standard substitution
of spherical coordinates
$\alpha=t\sin\phi\cos\theta$, $\beta=t\sin\phi\sin\theta$,
$\gamma=t\cos\phi$, where $\phi$ ranges from $0$ to $\pi$ and $\theta$
from $0$ to $2\pi$:
\begin{align}\label{eq:21}
    \int_{\Bcompl}|\ghat(\alpha,\beta,\gamma)|\,
    \mathrm{d\alpha\beta\gamma}
    =\int_{\Bcompl}t^2\sin\phi\cdot
    \ghat(t\sin\phi\cos\theta,t\sin\phi\sin\theta,
    t\cos\phi)\,\mathrm{dt\phi\theta}\;,
\end{align}
For fixed $\phi$ and $\theta$,
we let $\gtilde(t):=\ghat(t\sin\phi\cos\theta,\allowbreak t\sin\phi\sin\theta,\allowbreak t\cos\phi)$ and note that for the corresponding
iterated integral we have, for some appropriate
$t_0=t_0(\phi,\theta)$,
\begin{align*}
    \int_{\Bcompl} t^2\sin\phi\cdot\gtilde(t)\,\mathrm{d}t
    =\int_{t_0}^\infty t^2\sin\phi\cdot\gtilde(t)\,\mathrm{d}t
    \le\int_{t_0}^\infty t^2\cdot\gtilde(t)\,\mathrm{d}t\;.
\end{align*}
In order to bound this last integral, note that by definition
of $\ghat=\exp(-nQ)$ it must be that
$\gtilde(t)=\exp(-kt^2)$ for some $k\ge 0$.
Furthermore, by~\eqref{eq:20} we have $\exp(-k t_0^2)\le 2n^{-10}$ and 
by definition of $B$ also $1/n^2\le|t_0|\le\log^2 n/n$. 
Letting $\delta:=t_0$ and $\eps:=2n^{-10}$, we estimate
\begin{align}
    \int_{t_0}^\infty t^2\cdot\gtilde(t)\,\mathrm{d}t
    &=\int_{t_0}^\infty t^2\exp(-kt^2)\,\mathrm{d}t
    \le\delta\sum_{j=0}^\infty (t_0+(j+1)\delta)^2
    \exp(-k(t_0+\delta j)^2)\nonumber\\
    &\le t_0^3\exp(-kt_0^2)\sum_{j=0}^\infty (j+2)^2\exp(-kt_0^2 j)
    \le t_0^3\eps\sum_{j=0}^\infty\Big( (j+2)(j+1) + (j+1) + 1\Big)
    \eps^j\nonumber\\
    &=t_0^3\eps\left(\frac{1}{(1-\eps)^3}+\frac{1}{(1-\eps)^2}+\frac{1}{1-\eps}\right)
    \le n^{-12}\;.\label{eq:22}
\end{align}
Chaining together~\eqref{eq:22} and~\eqref{eq:21} and 
applying Fubini's theorem, we finally get
$\int_{\Bcompl}|\ghat|\le 2\pi^2n^{-12}\le n^{-11}$, better than what we needed
for~\eqref{eq:23}.

\paragraph{Term 3: $\int_{\Bcompl} |\hhat|$}
Here we divide the area of $\Bcompl$ in two more subcases. First, for values
$|\gamma|\le 4$, by Lemma~\ref{lem:fhat-bounded} we have that the total contribution
to the integral is at most $8\cdot n^{-10}$ (recall that $\hhat$ is zero outside
of $|\alpha|,|\beta|\le 1/2$). For $|\alpha|,|\beta|\le 1/2$ and $|\gamma|>4$,
by Lemma~\ref{lem:large-gamma} we have
$|\fhat(\alpha,\beta,\gamma)|\le 1/|\gamma-\alpha-\beta|\le 2/|\gamma|$, and
consequently, for fixed $\alpha,\beta$,
\begin{align*}
    \int_{|\gamma|>4}\left|\hhat(\alpha,\beta,\gamma)\right|\,\mathrm{d}\gamma
    \le 2^{n+1}\int_4^\infty \gamma^{-n}\,\mathrm{d}\gamma
    =\frac{4}{(n-1)2^{n-1}}\le 2^{-n}\;.
\end{align*}

Therefore, we have $\int_{\Bcompl} |\hhat|\le 8n^{-10}+2^{-n}\le n^{-9}$ 
and putting all the cases together
$\int_{\mathbb{R}^3}\left|\ghat-\hhat\right|\le \log^{16}n/n^4$.\qed

\subsection{Proof of Theorem~\ref{thm:conditional-clt}}
Finally we are ready to prove the main theorem of this section. As a preliminary point, by 
Claim~\ref{cl:conditional-variance} the condition
$\Var_A-(CV_{A}^2/\Var V),\Var_B-(CV_{B}^2/\Var V)\ge\eps n$
implies
$\Var[G_A^{*n}\cond H^{*n}=n^2/2],
\Var[G_B^{*n}\cond H^{*n}=n^2/2]\ge\eps n^2$.
This will be used when invoking
Claim~\ref{cl:gaussian-anti} later on.
Throughout we use the notation $\cE_V$ for the event $V^{*n}=n^2/2$ and $\cE_H$
for $H^{*n}=n^2/2$. 

Below we show a detailed calculation establishing 
$\Pr[U_A^{*n},U_B^{*n}\cond\cE_V]\ge
\Pr[G_A^{*n},G_B^{*n}\cond\cE_H]-\Otilde(1/\sqrt{n})$.
The justification for the reverse inequality is very similar
and we skip it.
We proceed
by a string of inequalities applying the lemmas we proved
before. Using union bound, we check that all events
used in those lemmas happen
except with probability $1/n$.
Whenever we are summing over $a$ and $b$, the sum goes
over the set $\mathbb{Z}+1/2$: 
\begin{IEEEeqnarray*}{rCl}
  \Pr\big[U_A^{*n},U_B^{*n}>0\cond \cE_V\big]
  &\stackrel{\text{Lem~\ref{lem:u_a-concentration}}}{\ge}
    &\frac
    {
    \Pr\big[50n\log n\ge U_A^{*n},U_B^{*n}>0\cond \cE_V\big]
    }
    {\Pr\big[|U_A^{*n}|,|U_B^{*n}|\le 50n\log n\cond\cE_V\big]+2n^{-10}
    }\\
  &\stackrel{\text{Cl \ref{cl:inverse-fourier}}}{=}
    &\frac{
    \sum_{50n\log n\ge a,b>0}
    \int_{\mathbb{R}^3}\hhat(\alpha,\beta,\gamma)
    e\left(-\alpha a-\beta b\right)\,\mathrm{d\alpha\beta\gamma}
    }
    {
    \sum_{|a|,|b|\le 50n\log n}
    \int_{\mathbb{R}^3}\hhat(\alpha,\beta,\gamma)
    e\left(-\alpha a-\beta b\right)\,\mathrm{d\alpha\beta\gamma}
    +2n^{-10}
    }\\
  &\stackrel{\text{Lem \ref{lem:characteristic-close}}}{\ge}
    &\frac{
    \sum_{50n\log n\ge a,b>0}
    \int_{\mathbb{R}^3}\ghat(\alpha,\beta,\gamma)
    e\left(-\alpha a-\beta b\right)\,\mathrm{d\alpha\beta\gamma}
    -\log^{19} n/n^2
    }
    {
    \sum_{|a|,|b|\le 50n\log n}
    \int_{\mathbb{R}^3}\ghat(\alpha,\beta,\gamma)
    e\left(-\alpha a-\beta b\right)\,\mathrm{d\alpha\beta\gamma}
    +\log^{19}n/n^2
    }\\
  &\stackrel{\text{Cl \ref{cl:inverse-fourier}}}{=}
    &\frac{
    \sum_{50n\log n\ge a,b>0}
    g(a,b,0)
    -\log^{19} n/n^2
    }
    {
    \sum_{|a|,|b|\le 50n\log n}g(a,b,0)
    +\log^{19}n/n^2
    }\\
  &\stackrel{*}{\ge}
    &\frac
    {
    g(0)\Pr\big[50n\log n> G_A^{*n},G_B^{*n}>1/2\cond \cE_H\big] -\log^{19}n/n^2
    }
    {
    g(0)+g(0)8\Pr\big[|G_A^{*n}|\le 1/2\cond\cE_H\big]+\log^{19}n/n^2
    }\\
    &\stackrel{\text{Cl \ref{cl:gaussian-anti}}}{\ge}
    &\frac
    {
    g(0)\Big(
    \Pr\big[50n\log n> G_A^{*n},G_B^{*n}>0\cond \cE_H\big]-2/\sqrt{\eps}n\Big)-\log^{19}n/n^2
    }
    {
    g(0)\Big(1+8/\sqrt{\eps}n\Big)+\log^{19}n/n^2
    }\\
    &\stackrel{\text{Lem~\ref{lem:u_a-concentration}}}{\ge}
    &\frac
    {
    g(0)\Big(
    \Pr\big[G_A^{*n},G_B^{*n}>0\cond\cE_H\big]
    -2/\sqrt{\eps}n-2n^{-10}
    \Big)-\log^{19}n/n^2
    }
    {
    g(0)\Big(
    1+8/\sqrt{\eps}n
    \Big)+\log^{19}n/n^2
    }\\
    &\stackrel{\text{Cl~\ref{cl:h-lower-bound}}}{\ge}
    &\frac{
    \Pr\big[G_A^{*n},G_B^{*n}>0\cond\cE_H\big]
    -\Otilde(1/\sqrt{n})
    }{
    1+\Otilde(1/\sqrt{n})
    }
    \ge
    \Pr\big[G_A^{*n},G_B^{*n}>0\cond\cE_H\big]
    -\Otilde(1/\sqrt{n})\;.
\end{IEEEeqnarray*}

We need an additional explanation for the inequality marked with the $*$ sign. There, in the numerator we used the fact
that $g(a,b,0)\ge\int_{B}g(a',b',0)\,\mathrm{d}a'b'$, where $B$ is the
$1\times 1$ box
\begin{align*}
  B = \Big\{(a',b'): |a|+1>|a'|>|a|,\sgn(a')=\sgn(a)\text{ and }
  |b|+1>|b'|>|b|,\sgn(b')=\sgn(b)\Big\}\;.
\end{align*}
Similarly, in the denominator we used that
whenever $|a|,|b|>1$, then
$g(a,b,0)\le \int_B g(a',b',0)\,\mathrm{d}a'b'$
for
\begin{align*}
    B = \Big\{(a',b'):
    |a|>|a'|>|a|-1,\sgn(a')=\sgn(a)\text{ and }
  |b|>|b'|>|b|-1,\sgn(b')=\sgn(b)\Big\}\;,
\end{align*}
as well as $g(a,b,0)\le4\int_Bg(a',b',0)\,
\mathrm{d}a'b'$ if $|a|=1/2$ or $|b|=1/2$ with
\begin{align*}
    B = \Big\{(a',b'):
    |a|>|a'|>|a|-1/2,\sgn(a')=\sgn(a)\text{ and }
  |b|>|b'|>|b|-1/2,\sgn(b')=\sgn(b)\Big\}\;.
\end{align*}
An illustration of this argument is provided in Figure~\ref{fig:g-decreasing}.\qed

\begin{figure}[!ht]\centering\begin{tikzpicture}
    \draw [-latex] (0, -1) -- (0, 3);
    \draw [-latex] (-1, 0) -- (3, 0);
    \draw [thick] (1, -0.1) -- (1, 0.1);
    \node at (1, -0.3) {$1$};
    \draw [thick] (-0.1, 1) -- (0.1, 1);
    \node at (-0.3, 1) {$1$};
    \filldraw [pink!30] (0.5,0.5) rectangle (3,3);
    \filldraw (0.5, 0.5) [red] circle [radius=0.1];
    \filldraw (1.5, 0.5) [red] circle [radius=0.1];
    \filldraw (2.5, 0.5) [red] circle [radius=0.1];
    \filldraw (0.5, 1.5) [red] circle [radius=0.1];
    \filldraw (1.5, 1.5) [red] circle [radius=0.1];
    \filldraw (2.5, 1.5) [red] circle [radius=0.1];
    \filldraw (0.5, 2.5) [red] circle [radius=0.1];
    \filldraw (1.5, 2.5) [red] circle [radius=0.1];
    \filldraw (2.5, 2.5) [red] circle [radius=0.1];

    \filldraw [pink!30] (5.5, -1) rectangle (9.5, 3);
    \filldraw [pink!60] (7, -1) rectangle (8, 3);
    \filldraw [pink!60] (5.5, 0.5) rectangle (9.5, 1.5);
    \filldraw [pink] (7, 0.5) rectangle (8, 1.5);
    \draw [-latex] (7.5, -1) -- (7.5, 3);
    \draw [-latex] (5.5, 1) -- (9.5, 1);
    \draw [thick] (8.5, 0.9) -- (8.5, 1.1);
    \draw [thick] (7.4, 2) -- (7.6, 2);
    \filldraw (8, 1.5) [red] circle [radius=0.1];
    \filldraw (9, 1.5) [red] circle [radius=0.1];
    \filldraw (8, 2.5) [red] circle [radius=0.1];
    \filldraw (9, 2.5) [red] circle [radius=0.1];
    \filldraw (8, 0.5) [red] circle [radius=0.1];
    \filldraw (9, 0.5) [red] circle [radius=0.1];
    \filldraw (8, -0.5) [red] circle [radius=0.1];
    \filldraw (9, -0.5) [red] circle [radius=0.1];
    \filldraw (7, 1.5) [red] circle [radius=0.1];
    \filldraw (6, 1.5) [red] circle [radius=0.1];
    \filldraw (7, 2.5) [red] circle [radius=0.1];
    \filldraw (6, 2.5) [red] circle [radius=0.1];
    \filldraw (7, 0.5) [red] circle [radius=0.1];
    \filldraw (6, 0.5) [red] circle [radius=0.1];
    \filldraw (7, -0.5) [red] circle [radius=0.1];
    \filldraw (6, -0.5) [red] circle [radius=0.1];
\end{tikzpicture}
\caption{A graphical illustration of claims
in the proof of Theorem~\ref{thm:conditional-clt}.
We are using the fact that
$g(ta_0,tb_0,0)$ is decreasing in $|t|$ for
every direction $(a_0,b_0)$.}
\label{fig:g-decreasing}
\end{figure}
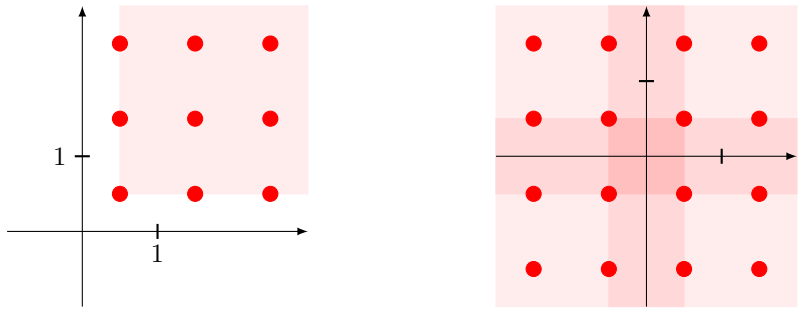

\subsection{Proof of Theorem~\ref{thm:one-fixed-equivalent}}
\label{sec:one-fixed-equivalent}

For the same reasons as those
given in Section~\ref{sec:clt}, let us assume,
as elsewhere in this section,
that the dice have faces in $[0,n]$ and that the conditioning
(denoted by $\cE_V$) is given by the event $V^{*n}=n^2/2$.
For purposes of this proof let 
$P(a,b):=\Pr[U_A^{*n}=a,U_B^{*n}=b\cond\cE_V]$.
To prove that $\Pr[U_A^{*n}>0\cond \cE_V]\approx 1/2$, i.e., that a random balanced
die $A$ is approximately ``fair'', we can employ a simplified argument using
symmetry. Specifically, by Claim~\ref{cl:inverse-fourier}
and Lemma~\ref{lem:characteristic-close}, except with probability
$n^{-2}$ we have that for every $a,b\in\mathbb{Z}+1/2$ it holds that
\begin{align*}
    \Big|
    P(a,b)\cdot h(0)-
    g(a, b, 0)
    \Big|
    \le
    \int_{\mathbb{R}^3}\Big|\hhat(\alpha,\beta,\gamma)
    -\ghat(\alpha,\beta,\gamma)\Big|\,\mathrm{d}\alpha\beta\gamma
    \le\frac{\log^{16} n}{n^4}\;.
\end{align*}
Since $g$ is a density function of a centered Gaussian, we have
$g(a,b,0)=g(-a,-b,0)$ which implies, using also
Claim~\ref{cl:h-lower-bound},
\begin{align}\label{eq:33}
\Big|P(a,b)-P(-a,-b)
\Big|\le\frac{8\log^{16}n}{n^{5/2}}\;.
\end{align}
Finally, we apply Lemma~\ref{lem:u_a-concentration}
(twice) and Lemma~\ref{lem:supremum-norm},
as well as~\eqref{eq:33} to get that,
except with probability $2n^{-10}+n^{-2}<1/n$,
\begin{align*}
    \left|\Pr\left[U_A^{*n}>0\mcond V^{*n}=\frac{n^2}{2}
    \right]-\frac{1}{2}\right|
    &\le\left|
    \sum_{\substack{0<a\le 50n\log n\\|b|\le 50n\log n}}
    P(a,b)
    -\frac{1}{2}\right|+2n^{-10}\\
    &\le\frac{1}{2}\left|
    \sum_{|a|,|b|\le 50n\log n} P(a,b)
    -1\right|+\Otilde\left(\frac{1}{\sqrt{n}}\right)
    \le\Otilde\left(\frac{1}{\sqrt{n}}\right)\;.
    \qed
\end{align*}

\section{Bounding moments} \label{sec:moments}
This section contains proofs of Lemmas~\ref{lem:var-2nd-moment}
and~\ref{lem:cv-2nd-moment}, following the method originally used in~\cite{HMRZ20}.
As a reminder, in this section we consider faces drawn uniformly in $[-\sqrt{3}, \sqrt{3}]$, i.e., $\EE V = 0$ and $\Var V = 1$
and for $x \in [-\sqrt{3}, \sqrt{3}]$ we have
\begin{align*}
    g_A(x) = \big|\{ i: a_i \leq x \}\big| - n \frac{x + \sqrt{3}}{2\sqrt{3}}.
\end{align*}
In this setting, a die $A= (a_i)_{i \in [n]}$ is said to be balanced if $\sum_{i=1}^n a_i = 0$
and the proper adaptation of the 
expectation formula~\eqref{eq:44}
for random balanced die becomes
\begin{align}\label{eq:49}
    \EE\left[f(a_1,\ldots,a_n)\right]
    &=
    \frac{
    \int_{\mathbb{R}^{n-1}}
    f(a_1,\ldots,a_n)\cdot\bigwedge_{i=1}^n\mathbbm{1}
    \left(-\sqrt{3}\le a_i\le\sqrt{3}\right)\,\mathrm{d}
    a_1\ldots a_{n-1}
    }{
    \int_{\mathbb{R}^{n-1}}
    \bigwedge_{i=1}^n\mathbbm{1}
    \left(-\sqrt{3}\le a_i\le\sqrt{3}\right)\,\mathrm{d}
    a_1\ldots a_{n-1}
    }\;,
\end{align}
where we let $a_n:=-\sum_{i=1}^{n-1}a_i$.

\subsection{Warm-up}
In this section we illustrate the methodology of the proofs of Lemmas~\ref{lem:var-2nd-moment} and~\ref{lem:cv-2nd-moment} on a simpler example. Namely, we show
\begin{lemma}\label{lem:warm-up}
$\EE[\Var_A] = \cO(n)$.
\end{lemma}
Note that strictly speaking Lemma~\ref{lem:warm-up}
is not necessary for our proof, since it follows
from Lemma~\ref{lem:var-2nd-moment} by Cauchy-Schwarz
inequality.

\begin{proof}
Let $A$ be a fixed die, not necessarily balanced, with faces
between $-\sqrt{3}$ and $\sqrt{3}$. Then, using
$\Pr[a\le V]=\Pr[a<V]=\frac{1}{2}-
\frac{a}{2\sqrt{3}}$
and
\begin{align}\label{eq:56}
\EE\big[\mathbbm{1}(a<V)V\big]
=\int_a^{\sqrt{3}} \frac{1}{2\sqrt{3}}
\cdot x\,\mathrm{d}x
=\frac{\sqrt{3}}{12}(3-a^2)\;,
\end{align}
we have
\begin{align}
     \EE[g_A(V)^2] &=  \sum_{i,j=1}^n \EE\left[\left(\mathbbm{1}(a_i<V) - \frac{V+\sqrt{3}}{2\sqrt{3}}\right) \left(\mathbbm{1}(a_j<V) - \frac{V+\sqrt{3}}{2\sqrt{3}}\right)\right] \nonumber\\
     & = \sum_{i,j=1}^n \Pr\left[\max\{a_i,a_j \} < V \right] - \frac{n}{\sqrt{3}} \sum_{i=1}^n \EE\left[ \mathbbm{1}(a_i <V) (V+ \sqrt{3}) \right]  + \frac{n^2}{12} \EE[(V + \sqrt{3})^2]
     \nonumber\\
    &=\frac{n^2}{2}-\frac{1}{2\sqrt{3}}\sum_{i,j=1}^n\max\{a_i,a_j\}
    -\frac{n^2}{4}+\frac{n}{12}\sum_{i=1}^n a_i^2
    -\frac{n^2}{2}+\frac{n}{2\sqrt{3}}\sum_{i=1}^n a_i
    +\frac{n^2}{3}
    \nonumber\\
   &=\frac{n^2}{12} 
   + \frac{n}{2 \sqrt{3}} \sum_{i=1}^n a_i 
   + \frac{n}{12} \sum_{i=1}^n a_i^2 
   - \frac{1}{2 \sqrt{3}}\sum_{i,j=1}^n \max \{ a_i,a_j \}. \label{eq:varA}
\end{align}
Let us start with a heuristic argument. To that
end, consider iid uniform $a_1,\ldots,a_n$
without balanced conditioning.
In that we case we can simply employ
$\EE[a_i]=0$, $\EE[a_i^2]=1$ and
\begin{align*}
    \EE[\max\{a_i,a_j\}]
    =2\int_{-\sqrt{3}}^{\sqrt{3}}\int_x^{\sqrt{3}}
    \frac{1}{12}y\,\mathrm{d}y\mathrm{d}x
    =\frac{1}{12}\int_{-\sqrt{3}}^{\sqrt{3}}3-x^2\,\mathrm{d}x
    =\frac{\sqrt{3}}{3}
\end{align*}
and substitute into~\eqref{eq:varA}, getting
\begin{align}\label{eq:40}
  \EE\Big[
  \EE[g_A(V)^2\mid A]\Big]
  & = \frac{n^2}{12}+\frac{n^2}{12} - \frac{n(n-1)}{2 \sqrt{3}}\cdot \frac{\sqrt{3}}{3} = \frac{n}{6}.
\end{align}
In other words, without conditioning, a typical value
of the second moment $\EE[g_A(V)^2]$ is $\Theta(n)$.
Now recall from Claim~\ref{cl:simple-g}
that if $A$ is a balanced die, then $\EE[g_A(V)] =0$
and consequently $\Var_A=\EE[g_A(V)^2]$. We are counting
on the estimate in~\eqref{eq:40} remaining of the same order of 
magnitude under the
conditioning. To verify that, we substitute into~\eqref{eq:varA}
in the balanced case and using $\sum_{i=1}^n a_i=0$, 
obtain
\begin{align} 
    \EE[\Var_A]&=\frac{n^2}{12}+\frac{n}{12} \EE\left[ \sum_{i=1}^n a_i^2\right] 
    - \frac{1}{2 \sqrt{3}} \EE\left[\sum_{i\neq j }\max \{a_i,a_j \}\right] \\
    & =\frac{n^2}{12}+\frac{n^2}{12} \EE[ a_1^2] - \frac{n(n-1)}{2 \sqrt{3}} \EE[\max \{a_1,a_2 \}]\;. \label{eq:expVA}
\end{align}
What remains is to compute the conditional expectations featured in (\ref{eq:expVA}):
For example, $\EE[a_1^2]$ (which is really a function of $n$) 
is the expectation of $a_1^2$ for a random balanced die.
In order to do that, we need to understand how the joint distribution of 
independent $(a_1,a_2)$ changes upon balanced conditioning. 

Let us start with $\EE[a_1^2]$.
Let $\tilde\phi_{n-1}$ be the density of
$(n-1)$-wise convolution of the uniform density
on $[-\sqrt{3},\sqrt{3}]$.
Recall the convolution formula
\begin{align*}
\tilde\phi_{n-1}(s)=(2\sqrt{3})^{-(n-1)}
\int_{\mathbb{R}^{n-1}}\bigwedge_{i=2}^n
\mathbbm{1}\left(-\sqrt{3}\le a_i\le\sqrt{3}\right)
\mathrm{d}a_2\ldots a_{n-1}\;,
\end{align*}
where $a_n=s-\sum_{i=2}^{n-1}a_i$.
Starting from~\eqref{eq:49}, and this time letting
$a_n:=-\sum_{i=1}^{n-1}a_i$, we see that
\begin{align*}
    \EE[a_1^2] 
    &=
    \frac{
    \int_{\mathbb{R}^{n-1}}
    a_1^2\cdot\bigwedge_{i=1}^n\mathbbm{1}
    \left(-\sqrt{3}\le a_i\le\sqrt{3}\right)\,\mathrm{d}
    a_1\ldots a_{n-1}
    }{
    \int_{\mathbb{R}^{n-1}}
    \bigwedge_{i=1}^n\mathbbm{1}
    \left(-\sqrt{3}\le a_i\le\sqrt{3}\right)\,\mathrm{d}
    a_1\ldots a_{n-1}
    }\\
    &=
    \frac{
    \int_{-\sqrt{3}}^{\sqrt{3}} a_1^2\cdot
    \int_{\mathbb{R}^{n-2}}\bigwedge_{i=2}^n
    \mathbbm{1}\left(-\sqrt{3}\le a_i\le\sqrt{3}\right)
    \,\mathrm{d}a_2\ldots a_{n-1}\,\mathrm{d} a_1
    }{
    \int_{-\sqrt{3}}^{\sqrt{3}}\,
    \int_{\mathbb{R}^{n-2}}\bigwedge_{i=2}^n
    \mathbbm{1}\left(-\sqrt{3}\le a_i\le\sqrt{3}\right)
    \,\mathrm{d}a_2\ldots a_{n-1}\,\mathrm{d} a_1
    }\\
    &= \frac{\int_{-\sqrt{3}}^{\sqrt{3}}  t^2 \tilde \phi_{n-1}(-t) d t}{\int_{-\sqrt{3}}^{\sqrt{3}} \tilde\phi_{n-1}(-t) dt}
    = \frac{\int_{-\sqrt{3}}^{\sqrt{3}}  t^2 \tilde \phi_{n-1}(t) d t}{\int_{-\sqrt{3}}^{\sqrt{3}} \tilde\phi_{n-1}(t) dt}
    = \frac{1}{2\sqrt{3}} \int_{-\sqrt{3}}^{\sqrt{3}} t^2  p_{n-1}(t) dt = \EE[V^2 p_{n-1}(V)]\;, 
\end{align*}
where $p_{n-1}(x) =\frac{\tilde\phi_{n-1}(x)}{\frac{1}{2\sqrt{3}}\int_{-\sqrt{3}}^{\sqrt{3}} \tilde\phi_{n-1}(s) ds} $ and $V \sim \mathcal{U}[-\sqrt{3},\sqrt{3}]$. This way, we expressed the expectation of $a_1^2$ in the conditional distribution through an unconditional expectation of another deterministic function.
Thus, our problem reduces to computing the density $\tilde \phi_{n-1}$ and the function $p_{n-1}$.

Similarly, for $\EE[\max\{a_1,a_2 \}] $ and letting $\tilde \phi_{n-2}$ be the density of $\sum_{i=3}^{n} a_i$, we have
\begin{align*}
    \EE[\max\{a_1,a_2 \}] 
    & = \frac{\int_{-\sqrt{3}}^{\sqrt{3}}\int_{-\sqrt{3}}^{\sqrt{3}} \max\{t_1,t_2 \}\tilde \phi_{n-2}(t_1+t_2) dt_1dt_2}{\int_{-\sqrt{3}}^{\sqrt{3}}\int_{-\sqrt{3}}^{\sqrt{3}} \tilde\phi_{n-2}(t_1+t_2) dt_1dt_2}\\
    & = \frac{1}{12}\int_{-\sqrt{3}}^{\sqrt{3}}\int_{-\sqrt{3}}^{\sqrt{3}} \max\{t_1,t_2 \}  p_{n-2}(t_1+t_2) dt_1 dt_2 = \EE[\max\{V_1,V_2\} p_{n-2}(V_1+V_2)], 
\end{align*}
where $p_{n-2}(x) =\frac{\tilde\phi_{n-2}(x)}{\frac{1}{12}\iint_{[-\sqrt{3},\sqrt{3}]^2} \tilde\phi_{n-2}(s_1+s_2) ds_1ds_2} $ and $V_1,V_2 \overset{iid}{\sim} \mathcal{U}[-\sqrt{3},\sqrt{3}]$. Again, the problem reduces to the computation of the 
function $p_{n-2}$.

We defer estimating $p_{n-1}$ and $p_{n-2}$ to the following section.
For now we finish the argument by applying~\eqref{eq:42} and~\eqref{eq:43}
from Lemma~\ref{lem:intlemma} to get
\begin{align*}
p_{n-1}(t) = 1 + \cO(n^{-1})\;,\qquad\qquad
p_{n-2}(t) = 1 + \cO(n^{-1})\;,
\end{align*}
and consequently
\begin{align*}
    \EE[a_1^2]&= \EE[V^2p_{n-1}(V)] = 1 + \cO(n^{-1})\;,\\
    \EE[\max \{ a_1,a_2 \}] &= 
    \EE\big[\max \{V_1,V_2\}p_{n-2}(V_1+V_2)\big]=
    \sqrt{3}/3 + \cO(n^{-1})\;,\\
    \EE[\Var_A]
    &=\frac{n^2}{12}+\frac{n^2}{12}\big(1+\cO(n^{-1})\big)
    -\frac{n(n-1)}{2\sqrt{3}}\left(\frac{\sqrt{3}}{3}+\cO(n^{-1})\right)
    =\cO(n)\;.\qedhere
\end{align*}
\end{proof}
\begin{remark}
As a matter of fact, Lemma~\ref{lem:intlemma} yields more precise estimates
\begin{align}\label{eq:50}
p_{n-1}(t)=1+\frac{1}{2n}-\frac{t^2}{2n}+\cO(n^{-2})\;,\qquad\qquad
p_{n-2}(t)=1+\frac{1}{n}-\frac{t^2}{2n}+\cO(n^{-2})\;,
\end{align}
which are then checked to lead to
\begin{align*}
    \EE[V^2p_{n-1}(V)]
    &=1+\frac{1}{2n}-\frac{\EE[V^4]}{2n}+\cO(n^{-2})
    =1-\frac{2}{5n}+\cO(n^{-2})\;,\\
    \EE\big[\max\{V_1,V_2\}p_{n-2}(V_1+V_2)\big]
    &=\frac{\sqrt{3}}{3}+\frac{\sqrt{3}}{3n}
    -\frac{\EE\big[\max\{V_1,V_2\}(V_1+V_2)^2\big]}{2n}
    +\cO(n^{-2})\\
    &=\frac{\sqrt{3}}{3}+\frac{2\sqrt{3}}{15n}+\cO(n^{-2})\;,\\
    \EE[\Var_A]
    &=\frac{n^2}{12}+\frac{n^2}{12}\left(1-\frac{2}{5n}\right)
    -\frac{n(n-1)}{2\sqrt{3}}\left(\frac{\sqrt{3}}{3}+\frac{2\sqrt{3}}{15n}\right)
    +\cO(1)\\
    &=\frac{1}{15}n+\cO(1)\;.
\end{align*}
Therefore, while conditioning on balanced $A$ decreased
the expected variance $\Var_A$, its order of magnitude
remained the same.
\end{remark}

In the following sections we will continue with justifying~\eqref{eq:50}.
This is obtained by estimating $\tilde\phi_{n-k}$ using a precise local central limit theorem for densities. Heuristically, for $x=\cO(1)$,
we have
\begin{align*}
   \tilde \phi_{n-k}(x) 
   &\propto \exp{\left(\frac{x^2}{2(n-k)}\right)} \approx 1-\frac{x^2}{2n}\;.
\end{align*}
However, there are several other error terms, including a correction
coming from the denominator in the definition of $p_{n-k}(x)$.
For our final objective, i.e., the proofs of Lemmas~\ref{lem:var-2nd-moment} and~\ref{lem:cv-2nd-moment}, we will need the approximations of $p_{n-1} (x)$ and $p_{n-2}(x)$ up to an $\cO(n^{-3})$ error term. 




\subsection{Estimating 
\texorpdfstring{$p_{n-k}$}{p\_\{n-k\}}
}

In this section we estimate the ``correction factors''
$p_{n-k}$ and derive general formulas for the balanced conditional
expectations. The precise statement we will use later on is:

\begin{lemma}
\label{lem:intlemma}
Let $A=(a_i)_{i \in [n]}$ and $B=(b_i)_{i \in [n]}$ be two random dice with faces drawn uniformly iid in $[-\sqrt{3}, \sqrt{3}] $, conditioned on both being balanced. Let $V_1,V_2,V_3,V_4 \overset{iid}{\sim} \mathcal{U}[-\sqrt{3},\sqrt{3}]$. Let $f$ denote a generic integrable function. Then,
\begin{align}
    \EE[f(a_1)]&= \EE[ f(V_1) p_{n-1}(V_1)]
    \nonumber\\
    \EE[f(a_1,a_2)]&= \EE[ f(V_1,V_2) p_{n-2}(V_1+V_2)]\nonumber\\
    \EE[f(a_1,a_2,a_3)] &= \EE[ f(V_1,V_2,V_3) p_{n-3}(V_1+V_2+V_3)]\nonumber\\
    \EE[f(a_1,a_2,a_3,a_4)] &= \EE[ f(V_1,V_2,V_3,V_4) p_{n-4}(V_1+V_2+V_3+V_4)]\nonumber\\
    \EE[f(a_1,b_1)] &= \EE[ f(V_1,V_2) p_{n-1}(V_1)p_{n-1}(V_2)]\nonumber\\
    \EE[f(a_1,a_2,b_1)] &= \EE[ f(V_1,V_2,V_3) p_{n-2}(V_1+V_2)p_{n-1}(V_3)]\nonumber\\
    \EE[f(a_1,a_2,b_1,b_2)] &= \EE[ f(V_1,V_2,V_3,V_4) p_{n-2}(V_1+V_2)p_{n-2}(V_3+V_4)]\;,
    \label{eq:55}
\end{align}
where
\begin{align}
    p_{n-1}(x)&=1 +\frac{1}{2n} - \frac{x^2}{2n}  + \frac{9}{40n^2} - \frac{ 9 x^2}{20n^2}+ \frac{x^4}{8 n^2} + \cO(n^{-3})\;,
    \label{eq:42}\\
     p_{n-2}(x)&= 1 + \frac{1}{n} - \frac{x^2}{2n} + \frac{6}{5n^2} - \frac{6 x^2}{5n^2} + \frac{x^4}{8 n^2}+ \cO(n^{-3})\;,
     \label{eq:43}
\end{align}
and
\begin{align}
     p_{n-3}(x)&=  1+ \frac{3}{2n} -\frac{x^2}{2n} + \cO(n^{-2}),
     \nonumber\\
     p_{n-4}(x)&=  1+ \frac{2}{n} -\frac{x^2}{2n} + \cO(n^{-2})\;.
     \label{eq:53}
\end{align}
\end{lemma}

The main tool in the proof of Lemma~\ref{lem:intlemma}
is a precise formula for $\tilde\phi_{n}(x)$.
We state it as an application to the uniform distribution 
of a general local limit theorem for densities from~\cite{Pet75}:

\begin{theorem}[\cite{Pet75}, Theorem 15, pp.~206-207] \label{thm:Pet75}
Let $(X_n)$ be a sequence of iid continuous random variables with bounded densities, zero mean, variance one, and finite absolute moment 
$\EE|X_1|^k$ for some $k \geq 3$. Moreover, for $\nu\in\mathbb{N}$, let
\begin{align}\label{eq:51}
    q_{\nu }(x):=\frac{1}{\sqrt{2 \pi}} e^{-x^{2} / 2} \sum H_{\nu+2 s}(x) \prod_{m=1}^{\nu} \frac{\Gamma_{m+2}^{k_m}}{k_{m} !},
\end{align}
where the summation is over all non-negative integer solutions $(k_1,k_2,...,k_\nu)$ of the equalities $k_1 + 2k_2+...+ \nu k_\nu = \nu $, $s = k_1+k_2+...+k_\nu$, $H_m$ is the $m^{th}$ (probabilists') 
Hermite polynomial and $\Gamma_k=\gamma_k/k!$ with
$\gamma_k$ being the cumulant of order $k$ of the random variable $X_1$. Then,
the density $\phi_n$ of the random variable $1/\sqrt{n}\sum_{i=1}^n X_i$
satisfies
\begin{align*}
    \phi_{n}(x)=\frac{1}{\sqrt{2 \pi}} e^{-x^{2} / 2}+\sum_{\nu=1}^{k-2} \frac{q_{\nu}(x)}{n^{\nu / 2}}+o\left(\frac{1}{n^{(k-2 )/ 2}}\right),
\end{align*}
uniformly in $x$.
\end{theorem}


\begin{corollary}\label{cor:clt}
Let $(X_n)$ be a sequence of iid random variables uniform in
$[-\sqrt{3},\sqrt{3}]$. Then, the density $\phi_n$ of the random variable
$1/\sqrt{n}\sum_{i=1}^n X_i$ satisfies
\begin{align*}
    \phi_n(x) 
    &= 
    \frac{e^{-\frac{x^2}{2}}}{\sqrt{2 \pi}} \left[
    1  + \frac{1}{n} \Gamma_4 H_4(x)+ \frac{1}{n^2} \left( \Gamma_6 H_6(x) + \frac{\Gamma_4^2}{2} H_8(x)\right)
    + \frac{P(x)}{n^3}
    \right] + o(n^{-3})\;,
\end{align*}
where $P(x)$ is a fixed polynomial,
$\Gamma_4=-1/20$, $\Gamma_6=1/105$ and
\begin{align*}
    &H_4(x)= 3 - 6x^2 + \cO(x^4)\;,\qquad
    H_6(x)= -15 + \cO(x^2)\;,\qquad
    H_8(x)= 105 +\cO(x^2)\;.
\end{align*}
\end{corollary}

\begin{proof}
It is a direct application of Theorem~\ref{thm:Pet75} to the uniform
distribution on $[-\sqrt{3},\sqrt{3}]$. For verification, it is useful
to note that since the uniform distribution is symmetric
around its mean zero, its odd order cumulants are zero. 
In particular,
all $q_{\nu}$ polynomials for odd $\nu$ are identically zero
(since in each term of the sum in~\eqref{eq:51} there is at least
one factor $\Gamma_k=0$ for odd $k$).
On the other hand, we recall the formulas for
Hermite polynomials and check that
$\gamma_4=-6/5$ and $\gamma_6=48/7$.
\end{proof}
We continue with a variable substitution to get:
\begin{corollary}\label{cor:phi-tilde}
Let $\tilde\phi_n(x)$ be the density of $\sum_{i=1}^n X_i$, where
$(X_i)_{i\in[n]}$ are iid uniform in $[-\sqrt{3},\sqrt{3}]$. Then,
for $|x|=\cO(1)$, we have
\begin{align*}
    \tilde\phi_n(x)
    &=
    \frac{1}{ \sqrt{2 \pi n}} \left[ 1 + \frac{A}{n} - \frac{x^2}{2 n} + \frac{C}{n^2} -\frac{A x^2}{2  n^2} + \frac{E x^2}{n^2} + \frac{x^4}{8  n^2} \right] + \cO(n^{-7/2}),
\end{align*}
where we denoted $A:=3\Gamma_4$, 
$C:=-15 \Gamma_6 +\frac{105}{2} \Gamma_4^2$
and $E:=-6 \Gamma_4$.
\end{corollary}
\begin{proof}
Recall that $\phi_n$ is the density of $\frac{1}{\sqrt{n}} \sum_{i=1}^n a_i$
and $\tilde\phi_n$ the density of
$\sum_{i=1}^n a_i$.
These two are related by
$\tilde \phi_n(x) = \frac{1}{\sqrt{n}} \phi_n\left (\frac{x}{\sqrt{n}}\right)$. Substituting
into Corollary~\ref{cor:clt},
\begin{align*}
    \tilde \phi_n(x) 
    &=\frac{e^{-\frac{x^2}{2n}}}{\sqrt{2\pi n}}
    \left[1+\frac{\Gamma_4}{n}H_4\left(\frac{x}{\sqrt{n}}\right)
    +\frac{\Gamma_6}{n^2}H_6\left(\frac{x}{\sqrt{n}}\right)
    +\frac{\Gamma_4^2}{2n^2}H_8\left(\frac{x}{\sqrt{n}}\right)
    \right]+O\left(n^{-7/2}\right)\;,
    \\
     & = \frac{e^{-\frac{x^2}{2n  }}}{ \sqrt{2 \pi n}} \left[ 1 + \frac{A}{n}  +\frac{C}{n^2} + \frac{E x^2}{n^2}  \right] + O \left( n^{-7/2} \right)\\
    & =   \frac{1}{ \sqrt{2 \pi n}}  \left( 1-\frac{x^2}{2 n} + \frac{x^4}{8  n^2}\right)\left[ 1 + \frac{A}{n}  +\frac{C}{n^2} + \frac{E x^2}{n^2} \right] + 
    O \left( n^{-7/2} \right)\\
    & =  \frac{1}{ \sqrt{2 \pi n}} \left[ 1 + \frac{A}{n} - \frac{x^2}{2 n} + \frac{C}{n^2} -\frac{A x^2}{2  n^2} + \frac{E x^2}{n^2} + \frac{x^4}{8  n^2} \right] + O \left( n^{-7/2} \right).\qedhere
\end{align*}
\end{proof}

We are now ready to complete the proof of Lemma~\ref{lem:intlemma}.

\begin{proof}[Proof of Lemma~\ref{lem:intlemma}]
Let us start with approximating $\tilde\phi_{n-k}(x)$
for fixed $k$ and $|x|=\cO(1)$. Applying 
Corollary~\ref{cor:phi-tilde} and
using $1/(n-k)=1/n+k/n^2+\cO(n^{-3})$ and
$1/(n-k)^2=1/n^2+\cO(n^{-3})$,
\begin{align}
    \tilde \phi_{n-k}(x) = &\frac{1}{\sqrt{2 \pi (n-k)}} \left[ 1 + \frac{A}{n-k} - \frac{x^2}{2 (n-k)} + \frac{C}{n^2}
    -\frac{A x^2}{2n^2} + \frac{E x^2}{n^2} + 
    \frac{x^4}{8n^2} \right] + 
    O \left( n^{-7/2}\right)
    \nonumber\\
    = & \frac{1}{ \sqrt{2 \pi (n-k)}} \left[ 
    1+ \frac{A}{n} +\frac{kA}{n^2} - \frac{x^2}{2 n} - \frac{kx^2}{2 n^2}+ \frac{C}{n^2} -\frac{A x^2}{2 n^2} + \frac{E x^2}{n^2} + \frac{x^4}{8 n^2} \right] 
    + O \left( n^{-7/2}  \right) 
    \nonumber\\
    = & \frac{1}{ \sqrt{2 \pi (n-k)}} \left[ 1+ \frac{A}{n} - \frac{x^2}{2 n} + \frac{kA+C}{n^2} 
    +\frac{(2E-A-k) x^2}{2  n^2} + \frac{x^4}{8  n^2} 
    \right] + O \left( n^{-7/2}  \right)\;.
    \label{eq:52}
\end{align}
As in the proof of Lemma~\ref{lem:warm-up}, we let
\begin{align*}
    p_{n-k}(x):=\frac
    {\tilde\phi_{n-k}(x)}
    {\frac{1}{(2\sqrt{3})^k}\int_{[-\sqrt{3},\sqrt{3}]^k}
    \tilde\phi_{n-k}\left(\sum_{i=1}^k s_i\right)\,
    \mathrm{d}s_1\ldots s_k
    }\;.
\end{align*}
Having already established~\eqref{eq:52}, we turn
to the denominator in the definition of $p_{n-k}$.
Let $V_1,\ldots,V_k$ be iid uniform in $[-\sqrt{3},\sqrt{3}]$. In the following we will use
$\EE\left(\sum_{i=1}^k V_i\right)^2=k$
and $\EE\left(\sum_{i=1}^k V_i\right)^4
=k\EE V_1^4+\binom{k}{2}\binom{4}{2}\EE[V_1^2V_2^2]
=9k/5+6\binom{k}{2}$. Applying~\eqref{eq:52},
\begin{align*}
    Z_{n-k} 
    :&= \frac{1}{(2 \sqrt{3})^k} 
    \int_{[-\sqrt{3},\sqrt{3}]^k} 
    \tilde \phi_{n-k}\left(\sum_{i=1}^k t_i\right) 
    \mathrm{d}t_1\ldots t_k
    =\EE\tilde\phi_{n-k}\left(\sum_{i=1}^k V_i\right)\\
    &= \frac{1}{\sqrt{2 \pi (n-k)}}  
    \left[ 1 + \frac{A}{n} - \frac{k}{2n} +\frac{kA+C}{n^2} +\frac{k(2E-A-k)}{2n^2} + \frac{9k/5+6\binom{k}{2}}{8n^2} \right] + \cO(n^{-7/2})\\
    & = \frac{1}{\sqrt{2 \pi (n-k)}}  
    \left[ 1 + \frac{A- \frac{k}{2}}{n} 
    +\frac{\frac{kA}{2}+C +kE - \frac{6k+5k^2}{40}}{n^2}  
    \right] + \cO(n^{-7/2})\;.
\end{align*}
Let us first show less precise formulas up to
$\cO(n^{-2})$ error. Those give
\begin{align*}
    p_{n-k}(x)
    &=\frac{\tilde\phi_{n-k}(x)}{Z_{n-k}}
    =\frac{1+A/n-x^2/2n+\cO(n^{-2})}
    {1+A/n-k/2n+\cO(n^{-2})}
    =1+\frac{k}{2n}-\frac{x^2}{2n}+\cO(n^{-2})\;,
\end{align*}
which indeed is consistent with~\eqref{eq:42}--\eqref{eq:53}.
To establish more detailed~\eqref{eq:42} and~\eqref{eq:43},
we apply the same method, also using
$\frac{1}{1+\alpha/n+\beta/n^2}
=1-\alpha/n-\beta/n^2+\alpha^2/n^2+\cO(n^{-3})$
along the way.
This time we find it easiest to treat $k=1$ and
$k=2$ separately, recalling
that $E=-6\Gamma_4=3/10$:
\begin{align*}
   p_{n-1}(x)&= \frac{\tilde \phi_{n-1}(x)}{Z_{n-1}}  = \sqrt{2\pi(n-1)}\tilde \phi_{n-1}(x) \left( 1- \frac{A- \frac{1}{2}}{n} -\frac{\frac{A}{2}+C +E - \frac{11}{40}}{n^2} + \frac{A^2 - A + \frac{1}{4}}{n^2}  + \cO(n^{-3}) \right) \\
     &= \left( 1+ \frac{A}{n} - \frac{x^2}{2 n} + \frac{A+C}{n^2} +\frac{(2E-A-1) x^2}{2 n^2} +\frac{x^4}{8 n^2}\right)\\
    &\qquad\cdot \left( 1- \frac{A- \frac{1}{2}}{n} -\frac{ \frac{3A}{2}- A^2+C +E - \frac{21}{40}}{n^2} \right)+\cO(n^{-3})\\
    &=  1 +\frac{1}{2n} - \frac{x^2}{2n}  + \frac{\frac{21}{40} -E}{n^2} + \frac{(E - \frac{3}{4} ) x^2}{n^2}+ \frac{x^4}{8 n^2} + \cO(n^{-3})\\
    &
    =1 +\frac{1}{2n} - \frac{x^2}{2n}  + \frac{9}{40n^2} - \frac{ 9 x^2}{20n^2}+ \frac{x^4}{8 n^2} + \cO(n^{-3})\;.\\
  p_{n-2}(x)
  &=  \frac{\tilde \phi_{n-2}(x)}{Z_{n-2}} 
    = \sqrt{2\pi(n-2)}\tilde \phi_{n-2}(x) \left(1 - \frac{A- 1}{n} -\frac{A+C +2E - \frac{4}{5}}{n^2} + \frac{A^2 -2A + 1}{n^2} + \cO(n^{-3})\right)\\
    &= \left( 1+ \frac{A}{n} - \frac{x^2}{2 n} + \frac{2A+C}{n^2} +\frac{(2E-A-2) x^2}{2  n^2} + \frac{x^4}{8 n^2} \right)\\
    &\qquad\cdot\left(1 - \frac{A- 1}{n} -\frac{A+C +2E - \frac{4}{5}}{n^2} + \frac{A^2 -2A +1}{n^2}\right)
    +\cO(n^{-3})\\
    &= 1 + \frac{1}{n} - \frac{x^2}{2n} + \frac{\frac{9}{5} -2E}{n^2} + \frac{(E -\frac{3}{2}) x^2}{n^2} + \frac{x^4}{8 n^2}+ \cO(n^{-3})\\
    &=
    1 + \frac{1}{n} - \frac{x^2}{2n} + \frac{6}{5n^2} - \frac{6 x^2}{5n^2} + \frac{x^4}{8 n^2}+ \cO(n^{-3})\;.
\end{align*}
Finally, each of the formulas in~\eqref{eq:55}
is justified in the same way as in the proof
of Lemma~\ref{lem:warm-up}.
\end{proof}

\subsection{Conditional expectations of terms}

As in the proof of Lemma~\ref{lem:warm-up},
the expectations considered in 
Lemmas~\ref{lem:var-2nd-moment}
and~\ref{lem:cv-2nd-moment} consist
of several terms. In this section
we apply Lemma~\ref{lem:intlemma}
to estimate each of those terms.
We start with a list of unconditional expectations that we need. Each
of those can be verified by checking
an elementary integral:

\begin{lemma}   \label{lem:simple-integrals}
For $V_1,V_2,V_3,V_4 \overset{iid}{\sim} \mathcal{U} [-\sqrt{3},\sqrt{3}]$, we have
\begin{IEEEeqnarray*}{rClrCl}
    \EE[V_1^k]=0 \text{ for every k odd}, 
    \qquad\EE[V_1^2]&=&1,
    &\EE[V_1^4]=\frac{9}{5},
    \qquad\qquad\EE[V_1^6]&=&\frac{27}{7},\\
    \EE[\max\{V_1,V_2\}]&=&\frac{\sqrt{3}}{3},
    &\EE[\max\{V_1,V_2\}^2 ]&=&1,\\
    \EE[\max\{V_1,V_2\} V_1]&=&\frac{1}{2},
    &\EE[\max\{V_1,V_2\}V_1^2]&=&\frac{2\sqrt{3}}{5},\\
    \EE[\max\{ V_1,V_2\} V_1^3]&=&\frac{9}{10},
    &\EE[\max\{ V_1,V_2\} V_1^4]&=&\frac{27\sqrt{3}}{35},\\
    \EE[\max\{V_1,V_2\}V_1V_2 ]&=&-\frac{\sqrt{3}}{5},
    &\EE[\max\{V_1,V_2\}V_1^2V_2]&=&\frac{1}{2},\\
    \EE[\max\{V_1,V_2\}V_1^2V_2^2]&=&\frac{3 \sqrt{3}}{7},
    &\EE[\max\{V_1,V_2\}\max\{V_1,V_3\} ]&=&\frac{3}{5}\;,\\
    \EE[\max\{V_1,V_2\}\max\{V_1,V_3\}V_1^2 ]&=&\frac{33}{35},
    &\qquad\qquad\EE[\max\{V_1,V_2\}\max\{V_1,V_3\}V_2^2 ]&=&\frac{23}{35},\\
    \EE[\max\{V_1,V_2\}\max\{V_1,V_3\}V_2V_3 ]&=&\frac{13}{35}\;.
\end{IEEEeqnarray*}
\end{lemma}

We now give the list of conditional expectations for each
term that we will need. Each of them is obtained
by the substitution of appropriate values from Lemma~\ref{lem:simple-integrals}
into one of the equations in~\eqref{eq:55}
from Lemma~\ref{lem:intlemma}:

\begin{lemma} \label{lem:calclemma}
Let $A=(a_i)_{i \in [n]}$ and $B=(b_i)_{i \in [n]}$ be two random
balanced dice with faces uniform in $[-\sqrt{3}, \sqrt{3}] $.
Then,
\begin{align*}
   \EE[a_1^2]&= 1 - \frac{2}{5n} -\frac{18}{175n^2} + \cO(n^{-3})\;,\\
   \EE[a_1^2a_2^2]&= 1 - \frac{4}{5n} +\frac{48}{175n^2} + \cO(n^{-3})\;,\\
   \EE[\max\{a_1,b_1\}]&= \frac{\sqrt{3}}{3}\left(1-\frac{1}{5n}-\frac{2}{25 n^2} 
   \right)+ \cO(n^{-3})\;,\\
   \EE[\max\{a_1,b_1\} \max\{a_2,b_2\}] 
   &= \frac{1}{3}\left(1 - \frac{19}{10n} -\frac{31}{50n^2} 
   \right)+ \cO(n^{-3})\;,\\
   \EE[a_1^2 \max\{a_2,b_1 \}] 
   &= \frac{\sqrt{3}}{3}\left(1 -\frac{3}{5n} -\frac{4}{175n^2}
   \right)+ \cO(n^{-3})\;.
\end{align*}
Furthermore,
\begin{align*}
    \EE[\max\{ a_1,a_2 \}]&= \frac{\sqrt{3}}{3} 
    \left(1+ \frac{2}{5n}\right) + \cO(n^{-2})\;,\\
    \EE[a_1^2 \max\{a_2,a_3 \}]&= \frac{\sqrt{3}}{3} + \cO(n^{-2})\;,\\
    \EE[\max \{a_1,a_2 \} \max \{ a_3,a_4 \}]&= \frac{1}{3}\left(1 - \frac{11}{5n}\right) +  \cO(n^{-2})\;,\\
    \EE[a_1^4] 
    &= \frac{9}{5}\left(1-\frac{4}{7n}\right)+ \cO(n^{-2})\;,\\
    \EE[\max\{a_1,b_1\}\max\{a_1,b_2\}]
    &=\frac{3}{5}\left(1-\frac{1}{n}\right)+\cO(n^{-2})\;,\\
    \EE[a_1^2\max\{a_1,b_1\}]
    &=\frac{2\sqrt{3}}{5}\left(1-\frac{1}{2n}\right)+\cO(n^{-2})\;,
\end{align*}
and
\begin{align*}
    \EE[\max \{a_1,a_2 \}^2] 
    &= 1+ \cO(n^{-1})\;,\\
    \EE[\max \{a_1,a_2\} \max \{a_1,a_3\}] 
    &= \frac{3}{5} + \cO(n^{-1})\;,\\
    \EE[a_1^2 \max \{a_1,a_2\}] 
    &= \frac{2 \sqrt{3}}{5}+ \cO(n^{-1})\;,\\
    \EE[\max\{a_1,b_1\}^2]&=1+\cO(n^{-1})\;.
\end{align*}
\end{lemma}

Finally, we are ready to perform calculations needed to establish Lemmas~\ref{lem:var-2nd-moment}
and~\ref{lem:cv-2nd-moment}.

\subsection{Proof of Lemma \ref{lem:var-2nd-moment}} \label{sec:proof-var-2nd}
Our objective is to show that $\EE[\Var_A^2] =\cO(n^2)$ for a random balanced $A$. 
We will use the foregoing lemmas
to write this expression as a sum of conditional expectations of simple functions 
of a small number of faces of the random die. We will then substitute approximations from 
Lemma~\ref{lem:calclemma}
and see that, up to $\cO(n)$ error, it is a polynomial where the coefficients of degree more than
two vanish. We now proceed with this plan.

Let $A$ be a balanced die with faces in $[-\sqrt{3},\sqrt{3}]$.
Squaring the expression in (\ref{eq:varA}) and using $\sum_{i=1}^n a_i =0$, we get
\begin{align*}
    \Var_A^2 =& \frac{n^4}{144} + \frac{n^2}{144} \left( \sum_{i=1}^n a_i^2 \right)^2 + \frac{1}{12} \left( \sum_{i\neq j}\max\{a_i,a_j\} \right)^2 \\
    & + \frac{n^3}{72} \sum_{i=1}^n a_i^2 - \frac{n^2}{12\sqrt{3}} \sum_{i \neq j} \max \{a_i, a_j \} - \frac{n}{12 \sqrt{3}} \left( \sum_{i=1}^n a_i^2\right)  \left( \sum_{i \neq j} \max \{a_i, a_j \} \right).
\end{align*}
Taking expectation over balanced $A$, we have
\begin{align*}
   \EE [ \Var_A^2] 
   &= \frac{n^4}{144} + \frac{n^2}{144}\left( (n^2-n) \EE[ a_1 ^2 a_2^2] + n \EE[ a_1 ^4]\right)  + \frac{n(n-1)}{12} \Big(2 \EE[\max \{a_1,a_2 \}^2] \\
   &\quad+4(n-2) \EE[\max \{a_1,a_2 \}\max \{a_1,a_3 \}] + (n-2)(n-3) \EE[\max \{a_1,a_2 \} \max \{a_3,a_4 \}]   \Big)\\
   &\quad+\frac{n^4}{72} \EE[a_1^2] - \frac{n^3(n-1)}{12 \sqrt{3}} \EE[\max \{ a_1,a_2 \}] - \frac{n^2(n-1)}{12\sqrt{3}} \Big( (n-2)  \EE[a_1^2\max \{a_2,a_3 \}] \\
   &\quad+ 2 \EE[a_1^2 \max\{a_1,a_2 \}] \Big)\\
   &= \frac{n^4}{144} \Bigg( 1+ \EE[ a_1 ^2 a_2^2] +    12  \EE[\max \{a_1,a_2 \}\max \{a_3,a_4 \}]+2\EE[a_1^2]\\
   &\quad- 4 \sqrt{3}  \EE[\max \{a_1,a_2 \} ] - 4 \sqrt{3} \EE[a_1^2 \max \{a_2,a_3 \}] \Bigg)\\
   &\quad+ \frac{n^3}{144} \Bigg(
   - \EE[a_1^2a_2^2]
   +\EE[a_1^4] + 48 \EE[\max \{ a_1,a_2 \} \max \{a_1,a_3 \}]\\
   &\quad- 72 \EE[\max \{a_1,a_2 \}\max \{a_3,a_4 \}] 
   + 4 \sqrt{3} \EE[\max \{a_1,a_2 \}]  
   + 12 \sqrt{3}   \EE[a_1^2\max \{a_2,a_3 \}]\\
   &\quad- 8 \sqrt{3} \EE[a_1^2 \max \{ a_1,a_2 \}]  \Bigg) \\
   &\quad+ \cO(n^{2}).
\end{align*}
In order to show that the coefficients of $n^4$ and $n^3$ are zero we apply Lemma~\ref{lem:calclemma}.
To check the coefficient of $n^4$ we need to look at the leading terms of all summands grouped under
$n^4$. These cancel out in the same way as in the unconditioned case. For the coefficient of $n^3$
we need to check second-order terms of summands grouped by $n^4$ and the most significant terms
of summands grouped by $n^3$. All in all we conclude that $\EE[\Var_A]=\cO(n^2)$.
\qed

\subsection{Proof of Lemma \ref{lem:cv-2nd-moment}}
In this section, we show that 
$\EE[\CVABcond^2] = \Omega(n^2)$, where 
$\CVABcond= CV_{AB} - CV_ACV_B$. 
Note that by triangle inequality:
\begin{align*}
    \EE\big[(CV_{AB}- CV_A CV_B)^2\big]
    &\geq \left( \sqrt{ \EE[CV_{AB}^2]} - \sqrt{\EE[CV_A^2] 
    \EE[CV_B^2] } \right)^2 \\
    &=\left( \sqrt{ \EE[CV_{AB}^2]} - \EE[CV_A^2]\right)^2.
\end{align*}
Therefore, we only need to show that 
\begin{align*}
    \sqrt{ \EE[CV_{AB}^2]} - \EE[CV_A^2] = \Omega(n).
\end{align*}
Indeed, this is established by two claims below. Their subsequent proofs 
complete the justification of Lemma~\ref{lem:cv-2nd-moment}.

\begin{claim}\label{cl:cvab}
$\EE[CV_{AB}^2]=\frac{11}{12600}n^2+\cO(n)$, in particular
$\sqrt{\EE[CV_{AB}^2]}>0.025n$.
\end{claim}
\begin{claim}\label{cl:cva}
$\EE[CV_A^2]=\frac{n}{60}+\cO(1)<0.02 n$.
\end{claim}

Both claims are proved with the same method as we used for Lemma~\ref{lem:var-2nd-moment}.

\begin{proof}[Proof of Claim~\ref{cl:cvab}]
Let us calculate the expression for the covariance $CV_{AB}$ for balanced dice $A,B$.
For that, let $V \sim \mathcal{U}[-\sqrt{3},\sqrt{3}]$. 
It will be useful to recall~\eqref{eq:56}, as well as the formula $\Pr[a<V]=1/2-a/2\sqrt{3}$
and the fact that the face-sums of $A$ and $B$ are zero:
\begin{align*}
   CV_{AB} 
   &= \EE[g_A(V) g_B(V)]=  \sum_{i,j=1}^n \EE\left[\left(\mathbbm{1}(a_i<V) - \frac{V+\sqrt{3}}{2\sqrt{3}}\right) \left(\mathbbm{1}(b_j<V) - \frac{V+\sqrt{3}}{2\sqrt{3}}\right)\right]\\
    &=\sum_{i,j=1}^n\Pr[\max\{a_i,b_j\}<V]
    -\frac{n}{2\sqrt{3}}\sum_{i=1}^n\EE\left[\mathbbm{1}(a_i<V)(V+\sqrt{3})\right]\\
    &\qquad\qquad-\frac{n}{2\sqrt{3}}\sum_{j=1}^n\EE\left[\mathbbm{1}(b_j<V)(V+\sqrt{3})\right]
    +\frac{n^2}{12}\EE\left[(V+\sqrt{3})^2\right]
 \\
 &=\frac{n^2}{2}-\frac{1}{2\sqrt{3}}\sum_{i,j=1}^n\max\{a_i,b_j\}-\frac{n^2}{8}
 +\frac{n}{24}\sum_{i=1}^n a_i^2 - \frac{n^2}{4}
 -\frac{n^2}{8}+\frac{n}{24}\sum_{j=1}^n b_j^2-\frac{n^2}{4}+\frac{n^2}{3}\\
  & = \frac{n^2}{12} + \frac{n}{24} \sum_{i=1}^n (a_i^2 +b_i^2) - \frac{1}{2\sqrt{3}} \sum_{i,j=1}^n \max \{a_i,b_j \}.
\end{align*}
Taking the square we obtain 
\begin{align*}
    CV_{AB}^2 &= \frac{n^4}{144} + \frac{n^2}{576} \left( \sum_{i=1}^n (a_i^2+ b_i^2) \right)^2 + \frac{1}{12} \left( \sum_{i,j} \max \{ a_i,b_j \} \right)^2\\
    &\quad+\frac{n^3}{144} \sum_{i=1}^n (a_i^2+ b_i^2) - \frac{n^2}{12 \sqrt{3}} \sum_{i,j} \max \{ a_i,b_j \}- \frac{n}{24 \sqrt{3}} \left( \sum_{i=1}^n (a_i^2+ b_i^2) \right) \left( \sum_{i,j} \max \{ a_i,b_j \} \right).
\end{align*}
We can then compute the expectation over balanced $A$ and $B$,
\begin{align*}
    \EE[CV_{AB}^2]&=
    \frac{n^4}{144} 
    +\frac{n^3}{288}\Big(\EE[a_1^4] + (n-1) \EE[a_1^2 a_2^2] +n \EE[a_1^2]^2\Big) 
    + \frac{n^2}{12} \Big( \EE[\max \{a_1,b_1 \}^2] \\
    &\qquad+ 2(n-1) \EE[ \max \{ a_1,b_1\} \max \{ a_1,b_2 \}] 
    + (n-1)^2 \EE[\max \{a_1,b_1 \} \max \{a_2,b_2 \}]\Big)\\
    &\qquad+ \frac{n^4}{72}  \EE[a_1^2] - \frac{n^4}{12 \sqrt{3} } \EE[\max\{a_1,b_1\}]  \\
    &\qquad- \frac{n^3}{12 \sqrt{3}}
    \Bigg( \EE[a_1^2 \max \{ a_1, b_1 \}] + (n-1) \EE [ a_1^2\max \{ a_2,b_1\}]\Big)\\
    &= \frac{n^4}{288} \Bigg( 2+ \EE[a_1^2a_2^2]+\EE[a_1^2]^2 
    + 24 \EE[\max \{a_1,b_1 \} \max \{a_2,b_2 \}]+4 \EE[a_1^2]   \\
    &\qquad- 8 \sqrt{3} \EE[\max\{a_1,b_1\}] - 8 \sqrt{3} \EE [ a_1^2\max \{ a_2,b_1\}] \Bigg)\\
    &\;+ \frac{n^3}{288}\Bigg(
    \EE[a_1^4]  - \EE[a_1^2a_2^2] + 48 \EE[ \max \{ a_1,b_1\} \max \{ a_1,b_2 \}] - 48 \EE[ \max \{ a_1,b_1\} \max \{ a_2,b_2 \}] \\
    &\qquad- 8 \sqrt{3} \EE[a_1^2 \max \{ a_1, b_1 \}] + 8 \sqrt{3} \EE [ a_1^2\max \{ a_2,b_1\}] \Bigg)\\
    &\;+ \frac{n^2}{12} \Bigg(\EE[\max \{a_1,b_1 \}^2] -2 \EE[ \max \{ a_1,b_1\} \max \{ a_1,b_2 \}] + \EE[ \max \{ a_1,b_1\} \max \{ a_2,b_2 \}] \Bigg)\;.
\end{align*}
All that remains is a mechanical substitution of estimates from Lemma~\ref{lem:calclemma}
verifying that indeed 
\begin{align*}
   \EE[CV_{AB}^2] = \frac{11}{12600}n^2 + \cO(n).
\end{align*}
In fact, all we need for Lemma~\ref{lem:cv-2nd-moment} is 
$\EE[CV_{AB}^2]>11n^2/12600-\cO(n)$ and to establish
that one only has to check the coefficient of $n^2$ (since 
$CV_{AB}^2\ge 0$, the leading coefficient must be positive, and the positive
leading coefficient by $n^3$ or $n^4$ would only asymptotically 
increase $\EE[CV_{AB}^2]$).
\end{proof}

\begin{proof}[Proof of Claim~\ref{cl:cva}]
Let us now turn to the simpler case of $\EE[CV_A^2]$. 
Again fixing a balanced $A$ and taking $V \sim \mathcal{U} [-\sqrt{3},\sqrt{3}]$, we get
\begin{align*}
      CV_{A}&= 
      \EE[g_A(V) V]= \sum_{i=1}^n \EE\left[ \left( \mathbbm{1}(a_i <V) - \frac{V+\sqrt{3}}{2\sqrt{3}}\right)\cdot V \right] 
      = \frac{3n}{4\sqrt{3}} - \frac{1}{4 \sqrt{3}} \sum_{i=1}^n a_i^2
      - \frac{n}{2\sqrt{3}}\\
      &=\frac{n}{4\sqrt{3}} - \frac{1}{4 \sqrt{3}} \sum_{i=1}^n a_i^2\;,
\end{align*}
whose square is 
\begin{align*}
    CV_A^2 = \frac{1}{48} \left( n^2 + \left( \sum_{i=1}^n a_i^2 \right)^2 -2n \sum_{i=1}^n a_i^2 \right).
\end{align*}
Taking the expectation over balanced $A$ and substituting from Lemma~\ref{lem:calclemma}
as in the previous proofs,
\begin{align*}
    \EE[CV_A^2]
    &=\frac{1}{48} 
    \Big( n^2 + n\EE[a_1^4]+n(n-1) \EE[a_1^2a_2^2]  -2n^2 \EE[a_1^2]\Big)\\
    &=\frac{n^2}{48}\Big( 1+ \EE[a_1^2a_2^2]-2\EE[a_1^2]\Big) 
    + \frac{n}{48} \Big( \EE[a_1^4] - \EE[a_1^2a_2^2] \Big)\\
    &=\frac{n}{60}+\cO(1)\;.\qedhere
\end{align*}
\end{proof}

\bibliographystyle{alpha}
\bibliography{arxiv_final}
  
\end{document}